\documentclass[11pt]{article}
\usepackage{epic,latexsym,graphicx, extpfeil}
\usepackage{latexsym,amsmath,amsthm,amssymb,graphicx,mathrsfs,amsfonts,yfonts,calrsfs,eufrak}
\usepackage{tcolorbox}
\usepackage{amsthm}
\usepackage{tikz}
\usepackage{verbatim}
\usetikzlibrary{backgrounds}
\usetikzlibrary{arrows, shapes, shapes.geometric}

\usepackage{tikz}
\usepackage{relsize}
\DeclareFontEncoding{LS1}{}{}
\DeclareFontSubstitution{LS1}{stix}{m}{n}
\DeclareSymbolFont{symbols3}{LS1}{stixbb}{m}{n}
\DeclareMathSymbol{\bigslopedvee}{\mathbin}{symbols3}{"A7}

        {\qquad\hspace*{\fill}$\square$\end{trivlist}}

        {\qquad\hspace*{\fill}\end{trivlist}}

\newcounter{fig}

\newcounter{clai}
\newcounter{exa}
\newcounter{theorem}[section]
\newcounter{lemma}[section]
\renewcommand{\thetheorem}{\arabic{section}.\arabic{theorem}}
\renewcommand{\thelemma}{\arabic{section}.\arabic{lemma}}
\newcounter{nona}[theorem]
\newcounter{nonanona}[theorem]
\newcounter{nonanonanona}[theorem]
\newcounter{nonanonanonab}[theorem]
\newcounter{nonanonanonabc}[nonanonanonab]

\renewcommand{\thenona}{\Alph{nona}}
\renewcommand{\thenonanona}{\alph{nonanona}}
\renewcommand{\thenonanonanona}{\roman{nonanonanona}}
\renewcommand{\thenonanonanonab}{\roman{nonanonanonab}}
\renewcommand{\thenonanonanonabc}{\roman{nonanonanonabc}}

\newenvironment{nonamenoname}{\begin{trivlist}\item[]\refstepcounter{nonanona}%
        {\bf (\thenonanona)\ \  \ }\nobreak\noindent\sl\ignorespaces}{%
        \ifvmode\smallskip\fi\end{trivlist}}
        
\newenvironment{nonamenonamenonameb}{\begin{trivlist}\item[]\refstepcounter{nonanonanonab}%
        {\bf (\thenona.\thenonanonanonab )\ \  \ }\nobreak\noindent\sl\ignorespaces}{%
        \ifvmode\smallskip\fi\end{trivlist}}
        
\newenvironment{nonamenonamenonamebc}{\begin{trivlist}\item[]\refstepcounter{nonanonanonabc}%
        {\bf (\thenona.\thenonanonanonab.\thenonanonanonabc)\ \  \ }\nobreak\noindent\sl\ignorespaces}{%
        \ifvmode\smallskip\fi\end{trivlist}}        
        
\newenvironment{nonamenonamenoname}{\begin{trivlist}\item[]\refstepcounter{nonanonanona}%
        {\bf (\thenonanona.\thenonanonanona )\ \  \ }\nobreak\noindent\sl\ignorespaces}{%
        \ifvmode\smallskip\fi\end{trivlist}}                

\newenvironment{theorem}{\begin{trivlist}\item[]\refstepcounter{theorem}%
        {\bf\thetheorem\ Theorem}\par\nobreak\noindent\sl\ignorespaces}{%
        \ifvmode\smallskip\fi\end{trivlist}}
        
\newenvironment{theoremnopar}{\begin{trivlist}\item[]\refstepcounter{theorem}%
        {\bf\thetheorem\ Theorem}\ \ \nobreak\noindent\sl\ignorespaces}{%
        \ifvmode\smallskip\fi\end{trivlist}}
\newenvironment{noname}{\begin{trivlist}\item[]\refstepcounter{nona}%
        {\bf (\thenona)\ \ \ }\nobreak\noindent\sl\ignorespaces}{%
        \ifvmode\smallskip\fi\end{trivlist}}
\newenvironment{theoremplus}[1]{\begin{trivlist}\item[]%
        \refstepcounter{theorem}{\bf\thetheorem\ Theorem} {\rm (\,#1\,)}%
        \par\nobreak\noindent\sl\ignorespaces}{%
        \ifvmode\smallskip\fi\end{trivlist}}
\newenvironment{lemma}{\begin{trivlist}\item[]\refstepcounter{lemma}%
        {\bf\thelemma\ Lemma}\par\nobreak\noindent\sl\ignorespaces}{%
        \ifvmode\smallskip\fi\end{trivlist}}
\newenvironment{lemmanopar}{\begin{trivlist}\item[]\refstepcounter{theorem}%
        {\bf\thetheorem\ Lemma}\ \ \nobreak\noindent\sl\ignorespaces}{%
        \ifvmode\smallskip\fi\end{trivlist}}

\newenvironment{proposition}{\begin{trivlist}\item[]\refstepcounter{theorem}%
        {\bf\thetheorem\ Proposition}\par\nobreak\noindent\sl\ignorespaces}{%
        \ifvmode\smallskip\fi\end{trivlist}}
\newenvironment{conjecture}{\begin{trivlist}\item[]%
        \refstepcounter{theorem}{\bf\thetheorem\ Conjecture}\par%
        \nobreak\noindent\sl\ignorespaces}{%
        \ifvmode\smallskip\fi\end{trivlist}}
\newenvironment{conjectureplus}[1]{\begin{trivlist}\item[]%
        \refstepcounter{theorem}{\bf\thetheorem\ Conjecture} %
        {\rm(\,#1\,)}\par\nobreak\noindent\sl\ignorespaces}{%
        \ifvmode\smallskip\fi\end{trivlist}}

        {\endlist\unskip}

\newenvironment{claim}{\begin{trivlist}\item[]\refstepcounter{clai}%
        {\bf Claim \theclai}\mbox{ \ }\sl\ignorespaces}{%
        \ifvmode\smallskip\fi\end{trivlist}}

        {\endtrivlist}
\newcommand{\wideitem}[1]{\leavevmode\hangindent\mathindent\noindent%
        \hbox to \mathindent{#1\hfil}\ignorespaces}

\renewcommand{\emptyset}{\mathchar"001F}
\newcommand{\eop}{\rule{1.4ex}{1.4ex}}

\newcommand{\A}{{\cal A}}
\newcommand{\B}{\mathcal{B}}
\newcommand{\C}{{\cal C}}
\newcommand{\D}{{\cal D}}

\newcommand{\I}{{\cal I}}

\newcommand{\X}{{\cal X}}

\newcommand{\fM}{\mathfrak{M}}
\newcommand{\fB}{\mathfrak{B}}
\newcommand{\fE}{\mathfrak{E}}
\newcommand{\sx}{\mathsf{x}}

\newcommand{\bigsquare}{\mathlarger{\square}}
\newcommand{\bigbowtie}{\mathlarger{\bowtie}}
\newcommand{\bigboxtimes}{\mathlarger{\boxtimes}}
\newcommand{\bigrhd}{\mathlarger{\rhd}}
\newcommand{\bigbigslopedvee}{\mathlarger{\bigslopedvee}}

\newcommand{\swapxy}[2]{
\bigsquare^{ \hspace{-16pt}\raisebox{4pt}
                 {$\mathsmaller{#1}$ \hspace{1pt} $\mathsmaller{#2}$}
               }
                                        }
\newcommand{\swapxyzw}[4]{\ \ \swapxy{#1}{#3}_{\hspace{-25pt}\raisebox{-1.5pt}{$\mathsmaller{#2}$ \hspace{2.7pt}$\mathsmaller{#4}$}}}

\newcommand{\crossswapxy}[2]{
\bigbowtie^{ \hspace{-16pt}\raisebox{4pt}
                 {$\mathsmaller{#1}$ \hspace{1pt} $\mathsmaller{#2}$}
               }
                                        }
\newcommand{\crossswapxyzw}[4]{\ \ \crossswapxy{#1}{#3}_{\hspace{-25pt}\raisebox{-1.5pt}{$\mathsmaller{#2}$ \hspace{2.7pt}$\mathsmaller{#4}$}}}

\newcommand{\exchangexy}[2]{
\bigboxtimes^{ \hspace{-16pt}\raisebox{4pt}
                 {$\mathsmaller{#1}$ \hspace{1pt} $\mathsmaller{#2}$}
               }
                                        }
\newcommand{\exchangexyzw}[4]{\ \ \exchangexy{#1}{#3}_{\hspace{-25pt}\raisebox{-1.5pt}{$\mathsmaller{#2}$ \hspace{2.7pt}$\mathsmaller{#4}$}}}

\newcommand{\singswapxy}[2]{
\bigbigslopedvee^{ \hspace{-16pt}\raisebox{4pt}
                 {$\mathsmaller{#1}$ \hspace{1pt} $\mathsmaller{#2}$}
               }
                                        }
\newcommand{\singswapxyz}[3]{\ \ \ \singswapxy{#1}{#3}_{\hspace{-25pt}\raisebox{-1.5pt}{$\mathsmaller{#2}$}}\ \ \ }

\newcommand{\singswaptri}[3]{\ \ \ \bigrhd^{\hspace{-16pt}\raisebox{4pt}{$\mathsmaller{#1}$}}_{\hspace{-16pt}\raisebox{-1.5pt}{$\mathsmaller{#2}$}}\hspace{2pt}\raisebox{2pt}{$\mathsmaller{#3}$}\ }

\newcommand{\bigs}{\vspace*{2.3in}}

\newcommand{\sms}{\vspace*{.2in}}

\title{\bf Frame Matroids, Toric Ideals, and a Conjecture of White}

\author{\quad\\
{\sc Sean McGuinness \thanks{Research supported by NSERC discovery grant}}\\
 Thompson Rivers University\\}

\date{}

\newpage

\bigs 

\centerline{\Large Frame Matroids, Toric Ideals, and a Conjecture of White}

\sms

\centerline{\large Sean McGuinness}

\centerline{\large Dept. of Mathematics}

\centerline{\large Thompson Rivers University}

\centerline{\large McGill Road, Kamloops BC}

\centerline{\large V2C5N3 Canada}

\centerline{email:  smcguinness@tru.ca}

\begin{document}
\newpage

\pagenumbering{arabic}

\maketitle
\begin{abstract}
\noindent Blasiak verified a conjecture of White for graphic
matroids by showing that the toric ideal of a graphic matroid is
generated by quadrics.  In this paper, we extend this result to frame matroids satisfying a {\it linearity condition}.  Such classes of frame matroids include graphic matroids, bicircular matroids, signed graphic matroids, and more generally frame matroids obtained from group-labelled graphs. 


\end{abstract}

\section{Introduction}\label{sec-intro}

When possible, we shall follow the terminology and notation of \cite{BonMur} for graphs.    For a graph $G$ and vertex $v \in V(G)$, we let $E_G(v)$ denote the set of edges incident with $v$, and we let $d_G(v)$ denote its degree.  For a subset of edges $X\subseteq E(G)$ or a subset of vertices $X\subseteq V(G),$ we let $G[X]$ denote the subgraph induced by $X.$  For all basic concepts and definitions pertaining to matroids, we refer the reader to \cite{Oxl}.  For a matroid $M,$ we let $\B(M)$ denote the set of bases in $M.$  For $B\in \B(M)$ and $e\in E(M)\backslash B,$ we let $C(e,B)$ denote the {\it fundamental circuit} with respect to $B$ which contains $e.$  For $e\in B,$ we let $C^*(e,B)$ denote the {\it fundamental cocircuit} with respect to $B$ which contains $e.$
For convenience, if we create a new base from a base $B$ by deleting $X\subset B$ and adding elements of $Y,$ then we denote the resulting base by $B-X+Y$.  In the case where $X = \{ x_1, \dots ,x_k \}$ and $Y= \{ y_1, \dots ,y_k \},$ we will often write the new base as $B - x_1 - \cdots - x_k +y_1 + \cdots + y_k.$  For example, if $X = \{ e \}$ and $Y = \{ f \},$ then the new base is just $B-e+f.$
For
a pair of bases $B_1, B_2 \in \B(M)$ we have the well-known {\bf
symmetric exchange property}; that is, for all $e\in B_1$, there
exists $f \in B_2$ such that $B_1' = B_1 - e + 
f $ and $B_2'= B_2 - f + e $ are bases of
$M.$  We say that two sequences of sets $(X_1, \dots ,X_k)$ and $(Y_1, \dots ,Y_k)$ are {\bf compatible} if as multisets, $\cup_{i=1}^kX_i = \cup_{i=1}^kY_i.$
In \cite{Whi2}, White made the following well-known conjecture:

\begin{conjectureplus}{White}
If $(A_1, \dots, A_k)$ and $(B_1,\dots ,B_k)$ are compatible sequences of bases in $\B(M)$, then one can transform $(A_1, \dots, A_k)$ into $(B_1, \dots, B_k)$
by a sequence of symmetric exchanges. 
\label{con-white1}
\end{conjectureplus}

The above conjecture can be rephrased in terms of generators of {\it toric ideals}.  For all elements $e\in E(M)$ we associate an indeterminant $x_e$,
and for each base $B$, we associate an indeterminant $y_B.$ Let
$\mathbb{K}$ be a field and let $S_M$ be the polynomial ring
$\mathbb{K}[y_B\ \big|\ B\in \B(M)]$.  The {\bf toric ideal} $\I_M$
of $S_M$ is defined to be the kernel of the homomorphism $\theta_M:
S_M \rightarrow \mathbb{K}[x_e\ \big| \ e\in E(M)]$ where
$\theta_M(y_B) = \prod_{e\in B}x_e$ for each base $B\in \B(M).$   It is known (see \cite{Stu}) that $\I_M$ is generated by {\bf binomials};  that is, polynomials $y_{B_1} y_{B_2} \cdots y_{B_k} - y_{B_1'} y_{B_2'} \cdots y_{B_k'}$ which are the difference of two monomials.   
Conjecture \ref{con-white1} implies the following:

\begin{conjecture}
For any matroid $M$, the toric ideal $\I_M$ is generated by
quadratic binomials $y_{B_1}y_{B_2} - y_{B_1'}y_{B_2'},$ where
$B_1',B_2'$ are bases obtained from $B_1,B_2$ by a symmetric
exchange.\label{con-white2}
\end{conjecture}

Kashiwaba \cite{Kas} verified Conjecture \ref{con-white1} for rank three matroids.  Schweig \cite{Sch} showed that the conjecture holds for lattice path matroids, and Laso\'{n} and Micha\l ek \cite{LasMic} verified it for strongly base orderable matroids.  This class includes all transversal matroids.  Perhaps the best known result relating to White's conjecture is a result of Blasiak \cite{Bla} who verified the conjecture for graphic matroids.
In this paper, we shall extend Blasiak's result by verifying Conjecture \ref{con-white1} for {\it linear} frame matroids.  The class of such matroids includes graphic matroids, bicircular matroids, Dowling matroids, signed graphic matroids and more generally, frame matroids obtained from group-labelled graphs.   Frame matroids are a well-studied object -- see Funk \cite{Fun} for a good, concise overview.  They are also important concept in the {\it graph minors project} initiated by Geelen, Gerards, and Whittle \cite{GeeGerWhi}. 

\subsection{Biased Graphs}\label{subsec-biased}

{\it Biased graphs}, introduced by Zaslavsky \cite{Zas}, have been widely studied, not least in the context of matroids. 
A {\bf theta graph}  is a graph consisting of three cycles $C_i,\ i = 1,2,3$ such that for all distinct $i,j,k\in \{ 1,2,3 \}$ we  have  $C_i\triangle C_j = C_k.$
Let $G$
be a graph, and let $\C$ be a subset of cycles of $G.$     The pair $(G,\C)$ is defined to be a {\bf biased graph} if it satisfies the following property:

\sms

\noindent {\bf Theta Property}:   If two cycles in a theta graph belong
to $\C$, then the third cycle must also belong to $\C.$

\sms
\noindent  
Each cycle
in $\C$ is referred to as a {\bf balanced cycle}; cycles not in $\C$ are said to be {\bf unbalanced}. 
 The pair $\Omega =(G,\C)$
is referred to as a {\bf biased graph}.   Additionally, the collection of cycles $\C$ is said to be {\bf linear} if for any cycle $C$ which can be expressed as the symmetric difference of cycles in $\C$, we have $C\in \C.$  When $\C$ is linear, we refer to $\Omega$ as a {\it linear} biased graph.

\subsection{Frame Matroids}\label{subsec-frame}

In \cite{Zas2}, it was first shown that {\it frame matroids}  can be defined in terms of a biased graph.  The {\bf frame matroid} of a biased graph $\Omega = (G, \C)$, denoted $M_F(\Omega)$,  is the matroid whose circuits are the edge sets represented by
one of the subgraphs below:

\begin{itemize}
\item[(i)]  A balanced cycle.
\item[(ii)]  Two unbalanced cycles sharing exactly one vertex.
\item[(iii)]  Two vertex disjoint unbalanced cycles joined by a path.
\item[(iv)]  A theta graph all of whose cycles are unbalanced.
\end{itemize}

A second matroid associated with $\Omega$, is called the {\bf lift matroid}, and is denoted $M_L(\Omega).$  The circuits of this matroid are defined in nearly the same way, except that we replace (iii) above with a subgraph consisting of two vertex disjoint unbalanced cycles.  Examples of frame and lift matroids include graphic matroids, bicircular matroids, and even cycle matroids.  
The
bases of $M_F(\Omega)$ are just the maximal collections of edges $A\subseteq E(G)$
for which each component of the induced graph $G[A]$ contains
at most one cycle, which if present, must be unbalanced.  The bases for $M_L(\Omega)$ differ slightly in that they are the maximal subsets of edges $A \subseteq E(G)$ where $G[A]$ contains at most one component having a (unbalanced) cycle.  When $\C$ contains all cycles
of $G$, then $M_F(\Omega) = M_L(\Omega)$ is just the graphic matroid $M(G).$ When $\C
= \emptyset,$ $M_F(\Omega)$ is seen to be the bicircular matroid
of $G.$ 

We say the $M_F(\Omega)$ and $M_L(\Omega)$ are {\it linear} if $\Omega$ is linear.  Clearly graphic matroids and bicircular matroids are linear frame matroids, and it can be shown that the frame matroids associated with group-labelled graphs are also linear.

For a collection $\A$ of cycles of $G$, we define a set $\C(\A)$ to be the linear collection of cycles consisting of all cycles which can be expressed as a symmetric difference of cycles in $\A.$ 
We call $\C(\A)$ the {\bf linear completion} of $\A$ and refer to $(G,\C(\A))$ as the biased graph induced by $\A.$  

\subsection{The Main Theorem}\label{subsec-maintheorem}
Our matroids $M$ will be assumed to be such that  $M= M(\Omega)$ where $\Omega = (G,\C)$ is a biased graph.  Suppose $\B = (B_1, \dots ,B_k)$
and $\B' = (B_1', \dots ,B_k')$ are compatible sequences of bases.  
Furthermore, suppose there exist $i,j\in \{ 1, \dots ,k \}$ and    
$e\in B_i$ and $f\in B_j$ such that $\B' = (B_1, \dots ,B_{i-1},
B_i-e+f, B_{i+1}, \dots , B_{j-1}, B_j-f+e, B_{j+1}, \dots ,B_{k}).$
We say that $\B'$ is obtained from $\B$ by a single
symmetric exchange and write $\B \sim_1 \B'.$
If $\B'$ is a sequence of bases obtained from $\B$ by a
sequence of symmetric exchanges, then we write $\B \sim \B'.$   Building on the methods of Blasiak \cite{Bla},  we shall prove the following:

\begin{theorem}
Let $M = M(\Omega)$ be a linear frame matroid where $\Omega =
(G,\C)$ and $r=r(M)\ge 2.$  For all $k\ge 2,$ and for
any two compatible sequences of $k$ bases $\B$ and $\B'$ of $M$, we have
that $\B \sim \B'.$  \label{the-main2}
\end{theorem}

As a result of the above we obtain the following:

\begin{theorem}
Let $M = M(\Omega)$ be a linear frame matroid where $\Omega = (G, \C).$  Then the toric ideal $\I_M$ of $S_M$ is generated by
the quadratic binomials $y_{B_1}y_{B_2} - y_{B_1'}y_{B_2'}$ where
the bases $B_i',\ i = 1,2$ are obtained from $B_i,\ i = 1,2$ by a
symmetric exchange.  \label{the-main1}
\end{theorem}

It should be mentioned that if $\B = (B_1, B_2)$ and $\B' = (B_1',
B_2')$ are compatible sequences of bases of a graphic matroid $M$, then it was shown by Farber et al \cite{FarRicSha} that $\B \sim \B'$.  

\subsection{Extended Base Sequences}\label{subsec-extendedbase}

    Let $\Omega = (G,\C)$ be a linear biased graph and let $M = M(\Omega)$ where  $r(M) = n = |V(G)|.$  Suppose $|E(M)| = kn+1$ and let $u\in V(G).$   Let $\B  = (B_1,\dots ,B_k)$ be a sequence of pairwise disjoint bases in $M$ where $\{ h \} = E(M)\backslash (B_1 \cup \cdots \cup B_k)$ and $h\in E_G(u).$  We refer the pair
$\B^+ = (\B, h )$ as an $\mathbf{u}-${\bf extended base sequence} of $M$ (or just extended base sequence when $u$ is implicit).    Let  $\B'^+ = (\B',  h' )$ be another $u$-extended base sequence where $\B' = (B_1', \dots , B_\kappa').$  We say that $\B'^+$ is obtained from  $\B^+$ by an {\bf exchange} if $\B'^+$ is obtained from $\B^+$ by one of the two operations below:

\begin{itemize}
\item[({\bf BB})]  A symmetric exchange between $B_i$ and $B_j$,  for some $1\le i < j \le \kappa,$ where for some $e\in B_i$ and $f\in B_j$, $B_i' = B_i -e+f,\ B_j' = B_j -f +e,$ and $h' = h.$ 

\item[({\bf EB})] An edge exchange between $h$ and $\B$, where for some $i \in \{ 1, \dots ,k\}$ and some $e \in B_i\cap E_G(u),$  $B_i' = B_i - e +h,$ $B_{j}' = B_j,\ \forall j\in \{ 1, \dots ,k \} \backslash \{ i \},$ and $h' = e.$ 
\end{itemize}

If a  $u$-extended base sequence $\B'^+$ can be obtained from $\B^+$ by a single exchange, then we write $\B^+ \sim_1 \B'^+.$  If $\B'^+$ is obtained from $\B^+$ by a sequence of exchanges, then we write $\B^+ \sim \B'^+.$  Later in this paper, we shall prove the following theorem:

\begin{theoremnopar}
Let $\B_i^+ = (\B_i, h_i),\ i = 1,2$ be $u$-extended base sequences where $\B_i = (B_{i1}, \dots ,B_{ik})$ is a sequence of pairwise disjoint bases.  If $d_G(u) =  2k +2$, then $\B_1^+ \sim \B_2^+.$\label{the-kappa31}
\end{theoremnopar}

\subsection{Overview of the Proof of Theorem \ref{the-main2}}\label{subsec-overview}

Due to the length of this paper, we shall give a rough blueprint describing its organization and the proof of Theorem \ref{the-main2}.  Our basic template is the inductive strategy employed in Blasiak's proof of the graphic matroid case, where we shall use induction on $r(M)$, the rank of $M.$  Adapting Blasiak's approach to frame matroids turns out to be remarkably complex. 
Suppose $\B_i = (B_{i1}, \dots ,B_{i\kappa}),\ i = 1,2$ are two compatible sequences of bases in a linear frame matroid $M = M(\Omega)$, where $\Omega = (G,\C),\ n= |V(G)|.$  Our goal is to show that $\B_1 \sim \B_2.$   By duplicating edges, we may assume that for $i=1,2,$ the bases of $\B_i$ partition $E(G).$  By averaging, there is vertex $v\in V(G)$ for which $d_G(v) \le 2n.$   We shall use this vertex in the initial
{\it reduction} step.   In Section \ref{sec-reductionsI}, we show that $\B_i,\ i = 1,2$ can be transformed into {\it v-reduced base sequences}.  That is, there are base sequences $\B_i' = (B_{i1}', \dots ,B_{ik}'), \ i = 1,2$ where $\B_i' \sim \B_i$ and each base $B_{ij}'$ has at most two edges incident with $v.$    We assume that $\B_i,\ i = 1,2$ are $v$-reduced.  We associate a {\it matching graph} $\fM_{\B_i}$ with $\B_i$, the matching edges corresponding to pairs of edges belonging to bases in $\B_i.$    In Section \ref{sec-matchings}, we show that via symmetric exchanges, the ($v$-reduced) base sequences $\B_j,\ i = 1,2$ can be chosen so that $\fM_{\B_i},\ i = 1,2$ differ by at most one edge.   In Section \ref{sec-newbiased}, we construct a new biased graph $\widehat{\Omega} = (\widehat{G}, \widehat{\C})$ and a biased matroid $\widehat{M} = M(\widehat{\Omega})$ where $V(\widehat{G}) = V(G)\backslash \{v \}$.   Using the matching graphs $\fM_{\B_i},\ i = 1,2,$ we associate base sequences $\widehat{\B}_i,\ i = 1,2$ in $\widehat{M}$ to $\B_i,\ i = 1,2$  When $\fM_{\B_1} = \fM_{\B_2},$ the sequences $\widehat{\B}_i,\ i = 1,2$ are compatible.   In this case, assuming Theorem \ref{the-main2} holds for $\widehat{M},$ we have $\widehat{\B}_1 \sim \widehat{\B}_2.$  Thus there are sequences of bases $\widehat{\B}_1 = \widehat{\D}_1 \sim_1 \widehat{\D}_2 \sim_1 \cdots \sim_1 \widehat{\D}_p = \widehat{\B}_2$.  If $\fM_{\B_1} \ne \fM_{\B_2},$ then $\widehat{\B}_i,\ i = 1,2$ are no longer compatible in $\widehat{M}.$  In this case, we change the sequences slightly, yielding new base sequences $\widehat{\B}_i',\ i=1,2$ which are compatible in $\widehat{M}.$  Arguing with $\widehat{\B}_i',\ i = 1,2$ in place of $\widehat{\B}_i,\ i =1,2$, there are sequences of bases
$\widehat{\B}_1' = \widehat{\D}_1' \sim_1 \widehat{D}_2' \sim_1 \cdots \sim_1 \widehat{\D}_{p}' = \widehat{\B}_2'.$ 
This will complete the reduction phase of the proof.  

In the {\it pull-back} step of the proof, we define the {\it pull-back} of a base sequence $\widehat{\B}$ in $\widehat{M};$ that is,  a collection of base sequences in $M$ which are associated with $\widehat{\B}.$  This is described in Section \ref{sec-pullback}.  For the base sequences $\B_i,\ i = 1,2$ described above,
if $\fM_{\B_1} = \fM_{\B_2},$ then our goal is to create base sequences $\D_i,\ i = 1, \dots ,p$ associated with the base sequences $\widehat{\D}_i, \ i = 1, \dots ,p$
where $\D_1 \sim \D_2 \sim \cdots \sim \D_p$. Otherwise, if $\fM_{\B_1} \ne \fM_{\B_2},$ then our goal to create sequences $\D_i',\ i = 1, \dots ,p$ associated with the base sequences $\widehat{\D}_i', \ i = 1, \dots ,p$ where $\D_1' \sim \D_2' \sim \cdots \sim \D_p'$.      While the task of constructing $\D_1 \sim \D_2 \sim \cdots \sim \D_{p}$, is relatively easy, the pull-back procedure faces a significant technical hurdle when $\fM_{\B_1} \ne \fM_{\B_2}$, the biggest single obstacle when extending Blasiak's theorem from graphic matroids to frame matroids. In this case,  some of the base sequences $\widehat{\D}_i',\ i = 1, \dots ,p$ may not even have a pull-back.  Because of this, we must consider {\it perturbations} $\widehat{\D}_i'',\ i = 1, \dots ,p$ of the sequences $\widehat{\D}_i',\ i = 1, \dots ,p$, a concept we introduce in Section \ref{subsubsec-perturbation}.   We proceed by constructing pull-backs $\D_i'',\ i = 1, \dots ,p$ of $\widehat{\D}_i'',\ i = 1, \dots ,p$ and show that $\D_1'' \sim \D_2'' \sim \cdots \sim \D_{p}''.$  To guarantee that such perturbations exist, we need to the look at the 
{\it switchability} and {\it amenability} of base sequences, concepts which are introduced in Section \ref{sec-BaseChange}.

The above pull-back procedure is complex, even in the case where we have base sequences with two bases.  Because of this, we need to strengthen the statement we prove in the induction step.  Our strategy is to prove Theorems \ref{the-main2} and \ref{the-kappa31} simultaneously.  Assuming Theorems \ref{the-main2} and \ref{the-kappa31} are true for all $k \le \kappa$, we shall first show that Theorem \ref{the-main2} is true for $k =\kappa+1.$   We then use this to prove that Theorem \ref{the-kappa31} holds for $k= \kappa+1$. 

\section{Serial exchanges}\label{sec-serexch}

Let $B_i,\ i = 1,2$ be bases and let $A_i \subseteq B_i,\ i = 1,2$ be subsets.  We say that $A_1$ can be {\bf serially exchanged} with $A_2$ with respect to $B_i,\ i = 1,2$ if for some ordering $a_{11} \prec a_{12} \prec \cdots \prec a_{1k}$ of $A_1$ and some ordering $a_{21} \prec a_{22} \prec \cdots \prec a_{2k}$ of $A_2$ it follows that $B_1 - \{ a_{11}, \dots ,a_{1i} \} + \{a_{21}, \dots ,a_{2i} \}$ and $B_2 - \{ a_{21}, \dots ,a_{2i} \} + \{a_{11}, \dots ,a_{1i} \}$ are bases for $i = 1, \dots ,k.$  We refer to such orderings of $A_i,\ i=1,2$ as {\bf serial orderings}.  For a matroid $M$ and positive integer $k$, we say that $M$ has the $\mathbf{k}${\bf - serial exchange property}, if for any two bases $B_1,B_2 \in \B(M)$ and any subset $A_1 \subseteq B_1$ where $|A_1| =k,$ there is a subset $A_2 \subseteq B_2$ for which $A_1$ can be serially exchanged with $A_2$ with respect to $B_i,\ i = 1,2.$
By the symmetric exchange property, all matroids have the $1$-serial exchange property. Kotlar and Ziv \cite{KotZiv} conjectured that all matroids $M$ have the $k$-serial exchange property, for all $k \le r(M).$  They showed that all matroids of rank at least two have the $2$-exchange property.  

\begin{theoremplus}{Kotlar, Ziv \cite{KotZiv}}
All matroids of rank at least two have the $2$-exchange property. \label{the-KotZiv}
\end{theoremplus} 

The above theorem will be used later on. 
  
\section{Cycle properties in biased graphs}\label{sec-biasprop}

We shall need a few lemmas pertaining to balanced cycles in biased graphs.  Let $\Omega = (G,\C)$ be a linear biased graph.
For a spanning tree $T$ in $G$, and an edge $e\not\in E(T)$, let $C_G(e,T)$ denote the {\it fundamental cycle} containing $e$ in $T\cup \{ e \}.$   When $G$ is implicit, we shall simply write $C(e,T)$ in place of $C_G(e,T).$   

\begin{lemma}
Let $T$ be a spanning tree and let
$P_i,\  i = 1,2$ be two internally vertex-disjoint
$(u,v)$-paths where $E(P_1) \subseteq E(T)$ and $E(P_2) \cap E(T) = \emptyset.$  If
for each $e \in E(P_2),$ $C(e,T)$ is a balanced cycle, then $C' = P_1
\cup P_2$ is also a balanced cycle.\label{lem-lemma1}
\end{lemma}

\begin{proof}
Let $C = \bigtriangleup_{e\in E(P_2)}C(e,T)$ and $C' = P_1 \cup P_2.$  We have that  $C$ is a edge-disjoint union of cycles and $P_2$ is a subgraph of $C$.  This means that $C\triangle C'$ is a subgraph of $T.$  If $C\ne C'$, then
$C\triangle C'$ is an edge-disjoint union of cycles contained in $T.$  This is impossible since $T$ is a tree.  Thus $C= C'$ and it follows by the linearity of $\C$ that $C'$ is balanced. 
\end{proof}

We have the following useful lemma:

\begin{lemma}
Suppose $C$ is an unbalanced cycle  and $T$ is a spanning tree where $E(C) \cap E(T) = \emptyset.$  Then there exists $e\in E(C)$ such that $C(e,T)$ is unbalanced.
\label{cor-lemma1}
\end{lemma}

\begin{proof}
Let $C' = \bigtriangleup_{e\in E(C)}C(e,T).$  Then $C$ is a subgraph of $C'$.  Suppose $C \ne C'$.  Then $C\triangle C'$ is a disjoint union of cycles, and furthermore, $C\triangle C'$ is a subgraph of $T.$  However, this is impossible since $T$ is as tree.  Thus $C= C'.$  Since $C$ is unbalanced, it follows by the linearity of $\C$ that at least one of the fundamental circuits $C(e,T),\ e\in E(C)$ is unbalanced.
\end{proof}

We shall need the following more technical lemma:

\begin{lemma}
Let $C_1$ and $C_2$ be cycles in $G$ where the edges  of $E(C_1)\cap E(C_2)$
induce a path $P$ between vertices $u$ and $v.$  Let $P_i,\ i
= 1,2$ be paths such that $C_i = P_i \cup P,\ i = 1,2.$  Let $e_1\in E(P)$ and suppose $T$ is a spanning tree where $E(C_1)\backslash \{ e_1\} \subseteq E(T)$ and $E(P_2) \cap E(T) = \emptyset.$  Suppose
that $C(e,T)$ is balanced for all
$e \in E(P_2)$. Then $C_1$ is balanced if and only if $C_2$ is
balanced.\label{lem-lemma2}
\end{lemma}

\begin{proof}
Let $C = \bigtriangleup_{e\in E(P_2)} C(e,T).$  Then we see that $C' = C\triangle C_2$ is a subgraph of $T \cup \{ e_1 \}.$  Given that $C'$ is an edge-disjoint union of cycles, it follows that $C' = C(e_1,T) = C_1.$  Thus
$C_1 = C \triangle C_2$ and hence $C_2 = C \triangle C_1$  and $C = C_1 \triangle C_2$ (a cycle) as well.  By the linearity of $\C$, it follows that $C_1$ is balanced iff $C_2$ is balanced.
\end{proof}

 \section{The Biased Graph $\widehat{\Omega} = (\widehat{G}, \widehat{\C} )$}\label{sec-newbiased}
 
 Suppose $\Omega = (G,\C)$ is a linear biased graph where $\C$ contains no balanced loops.   Let $v\in V(G).$  We shall describe a new linear biased graph $\widehat{\Omega} = (\widehat{G}, \widehat{\C})$ obtained from $\Omega$ where $\widehat{\C}$ contains no balanced loops and $V(\widehat{G}) = V(G) \backslash \{ v \}.$  Let $\{ e_1, \dots ,e_m \}$ be the edges incident with $v$, and for $i = 1, \dots ,m$ let $e_i = vv(e_i)$ where $v(e_i) \in V(G)$, possibly $v(e_i)=v,$ when $e_i$ is a loop.  We shall define a graph $\widehat{G}$ where $V(\widehat{G}) = V(G)\backslash \{ v \}$ and $E(\widehat{G})  = E(G-v) \cup \widehat{E},$ the set $\widehat{E}$ which we will now define.
 For all $i,j \in \{ 1, \dots ,m\}$, $i\ne j$ where at least one of $e_i$ or $e_j$ is not a loop, we shall associate an edge $\widehat{e}_{ij}$ with the pair of edges $\{ e_i, e_j \}$ as follows:
 \begin{itemize}
 \item If $e_i = vv(e_i)$ and $e_j = vv(e_j)$ are not loops, then let $\widehat{e}_{ij} = v(e_i)v_(e_j)$ be an edge of $\widehat{G}$ having endvertices $v(e_i)$ and $v(e_j).$
 \item If exactly one of $e_i$ or $e_j$ is a loop, say $e_j$, then let $\widehat{e}_{ij}$ be a loop at the vertex $v(e_i).$  We call $\widehat{e}_{ij}$ a {\bf stem loop}.
 \end{itemize}
 
 \noindent We define $\widehat{E} := \{ \widehat{e}_{ij}\ | \ i,j \in \{ 1, \dots ,m\} \}.$  We now define the {\it pull-back} $\fE(\widehat{e})$ of an edge $\widehat{e} \in E(\widehat{G}).$
 For all edges $\widehat{e} \in E(\widehat{G})\backslash \widehat{E},$ let $\fE(\widehat{e}) = \{ \widehat{e} \}.$  For all $\widehat{e} \in \widehat{E},$ if $\widehat{e} = \widehat{e}_{ij},$ then we define $\fE(\widehat{e}) := \{ e_i, e_j \}.$  For a subset $\widehat{F} \subseteq E(\widehat{G}),$ we define the pull-back of $\widehat{F}$ to be  $\fE(\widehat{F}) := \bigtriangleup_{\widehat{f} \in \widehat{F}}\fE(\widehat{f});$ that is, $\fE(\widehat{F})$ is the symmetric difference of the sets $\fE(\widehat{f}), \widehat{f} \in \widehat{F}.$   For each subgraph $\widehat{H}$ of $\widehat{G},$ the subgraph $H$ of $G$ induced by the edges of $\fE(E(\widehat{H}))$ is referred to as the {\bf pullback} of $\widehat{H}.$    In short, we write $H = \fE(\widehat{H}).$
  
 We shall now define a set of cycles $\widehat{\A}$ of $\widehat{G}.$   Let $\widehat{C}$ be a cycle of $\widehat{G}$ and let $C = \fE(\widehat{C})$.  Suppose that $E(\widehat{C})\cap \widehat{E} \ne \emptyset$ and $\widehat{C}$ is not a stem loop. If  $E(C) \ne \emptyset$, then $C$ has a unique cycle decomposition $C = C_1 \cup \dots \cup C_t$ where $V(C_i) \cap V(C_j) = \{ v \}$, for all $i,j$ where $1\le i < j \le t.$  The cycles $C_1, \dots, C_t$ are referred to as the {\bf petals} of $C$.    See Figure \ref{fig1}.  
We define a subset of cycles $\widehat{\A}$ in $\widehat{G}$ as follows:   suppose $\widehat{C}$ is a cycle in $\widehat{G}$ which is not a stem loop, and let $C = \fE(\widehat{C}).$  Then $\widehat{C} \in \widehat{\A}$ if and only if one of the following holds:
  
\begin{itemize}
\item $\widehat{C} \in \C.$
\item $E(C) = \emptyset.$
\item $E(\widehat{C}) \cap \widehat{E} \ne \emptyset$, $E(C) \ne \emptyset,$ and each petal of $C$ is balanced.
\end{itemize}

\noindent Let $\widehat{\A}'$ be the collection resulting from deleting all balanced loops from $\widehat{\A}.$
We define $\widehat{\C} := \C(\widehat{\A}')$, the linear completion of $\widehat{\A}'$ in $\widehat{G}$ and we define $\widehat{\Omega} := (\widehat{G}, \widehat{\C})$ to be the biased graph induced by $\widehat{\A}'.$   We refer to $\widehat{\Omega}$ as the $\mathbf{v}${\bf-deleted} biased graph obtained from $\Omega.$   Note that each stem loop of $\widehat{G}$ is unbalanced in $\widehat{\Omega}$.

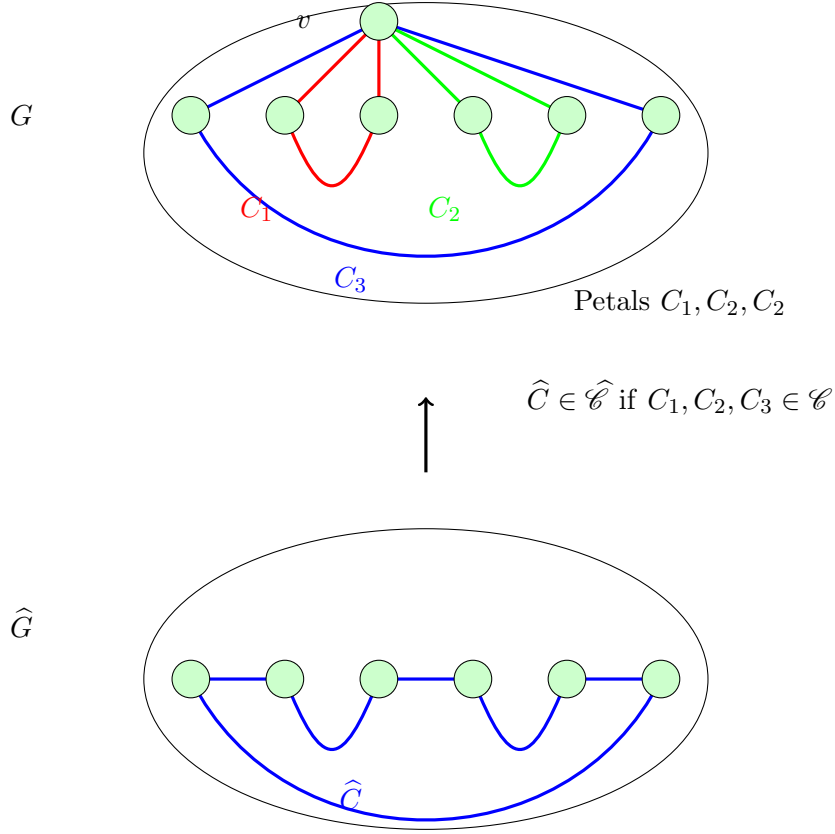
\begin{figure}[t]
\begin{flushright}
\begin{tikzpicture}[scale=.25]


 \pgfmathsetmacro{\minsize}{0.5cm}

\tikzstyle{every circle node} = [draw, fill=green!20,shape=circle,
minimum size = \minsize, inner sep = 0.20]

\vspace*{-4cm}\hspace{-1cm}
\draw (-5,30) node[] {$G$}; \draw (-5,3) node[]
{$\widehat{G}$};

\draw[very thick, color=blue] (10,35) -- (0,30);

\draw[very thick, color=red] (10,35) -- (5,30);

\draw[very thick,color=red] (10,35) -- (10,30);

\draw[very thick, color=blue] (10,35) -- (25,30);

\draw[very thick,color=green] (10,35) -- (15,30);

\draw[very thick, color=green] (10,35) -- (20,30);

\draw[color=blue, very thick](0,30) .. controls (5,20) and (20, 20) .. node[below]{$C_3$}(25,30); 

\draw[color=red, very thick](5,30) .. controls (7,25) and (8, 25) .. node[below]{$C_1$} (10,30); 

\draw[color=green, very thick](15,30) .. controls (17,25) and (18, 25) .. node[below]{$C_2$} (20,30);

\draw[color=blue, very thick](0,0) .. controls (5,-10) and (20, -10) .. node[above]{$\widehat{C}$}(25,0); 

\draw[color=blue, very thick](5,0) .. controls (7,-5) and (8, -5) .. (10,0); 

\draw[color=blue, very thick](15,0) .. controls (17,-5) and (18, -5) .. (20,0);

\draw[color =blue,very thick]((0,0) -- (5,0); \draw[color =blue,very thick]((10,0) -- (15,0); \draw[color =blue,very thick]((20,0) -- (25,0);

\draw( 12.5,28) ellipse (15 and 8); \draw (12.5,0) ellipse (15 and
8);

\draw[very thick,->]  (12.5,11) -- (12.5,15);

\draw node[circle] at (10,35) {$v$}; \draw node[circle] at (0,30) {}
;\draw node[circle] at (5,30){}; \draw node[circle] at (10,30){};
\draw node[circle] at (15,30){}; \draw node[circle] at (20,30){};
\draw node[circle] at (25,30){};

\draw node[circle] at (0,0){};\draw node[circle] at (5,0){};\draw
node[circle] at (10,0){};

\draw node[circle] at (15,0){};\draw node[circle] at (20,0){};\draw
node[circle] at (25,0){};

\draw node[] at (30,20){$\mathrm{Petals}\ C_1, C_2, C_2$};
\draw node[] at (30,15){$\widehat{C} \in \widehat{\C}\  \mathrm{if} \ C_1, C_2,C_3 \in \C$};

\end{tikzpicture}
\end{flushright}
\caption{Petals of a cycle $\widehat{C}$}\label{fig1}
\end{figure}
  
It is imperative to show that each cycle $\widehat{C}$ in $\widehat{G}$ which is an unbalanced cycle in $G$, is also an unbalanced cycle in $\widehat{G}.$  The next lemma establishes this and more. 
   
\begin{lemma}
For all cycles $\widehat{C}$ in $\widehat{G},$ if $C = \fE(\widehat{C})$ is an unbalanced cycle in $\Omega,$ then $\widehat{C}$ is an unbalanced cycle in $\widehat{\Omega}.$\label{lem-newbiased}
\end{lemma}

\begin{proof}
Suppose that $\widehat{C}$ is a cycle in $\widehat{G}$ such that $C=\fE(\widehat{C})$ is an unbalanced cycle in $\Omega.$  Suppose to the contrary that $\widehat{C}$ is balanced.  Then $\widehat{C} = \widehat{C_1} \triangle \cdots \triangle \widehat{C}_k$ for some cycles $\widehat{C}_i \in \widehat{\A},\ i = 1, \dots ,k.$  Let $C_i = \fE(\widehat{C}_i),\ i =1, \dots ,k.$  Then for $i=1, \dots ,k,$ either $C_i$ is a disjoint union of cycles in $\C$ or $E(C_i) = \emptyset.$    Given $\widehat{C} = \widehat{C_1} \triangle \cdots \triangle \widehat{C}_k$, we see that $C = C_1 \triangle \cdots \triangle C_k.$  One now sees that $C = C_{1}^* \triangle C_{2}^* \triangle \cdots \triangle C_{l}^*$, where for $i = 1, \dots ,l$, $C_i^* = C_j$ for some $j$, or $C_i^*$ is a petal of some $C_j.$  Since $\widehat{C}_i \in \widehat{\C},\ i = 1, \dots ,k,$ it follows that $C_i^* \in \C,\ i = 1, \dots ,l.$  Hence it follows by linearity that $C\in \C.$  This yields a contradiction.  Therefore, $\widehat{C}$ is unbalanced.
\end{proof}

\section{Changing Bases in $\widehat{M} = M(\widehat{\Omega})$}\label{sec-BaseChange}

Let $\widehat{\Omega} = (\widehat{G}, \widehat{\C})$ be the $v$-deleted biased graph described in Section \ref{sec-newbiased}.  Let $\widehat{M} = M(\widehat{\Omega})$ be the linear frame matroid associated with $\widehat{\Omega}.$ In this section, we will introduced some lemmas which describe how bases in $\widehat{M}$ can be changed.  This will be important later on when we introduce {\it perturbations}. 
 
 Let $I \subseteq \{ 1, \dots ,m \}$ and let $S_I$ denote the set of permutations of $I$.  Let $\widehat{H}_I$ denote the subgraph of $\widehat{G}$ induced by the edges in $\{ \widehat{e}_{ij}\ \big| \ i,j \in I \}.$  For a base $\widehat{B} \in \B(\widehat{M}),$ let $\widehat{F}_I(\widehat{B}) = \{ \widehat{e} \in E(\widehat{H}_I) \ \big| \ (\widehat{B} \backslash E(\widehat{H})_I) \cup \{ \widehat{e} \} \in \B(\widehat{M}) \}.$  Note that if $\widehat{B} \cap E(\widehat{H}_I) = \{ \widehat{e} \},$ then $\widehat{e} \in \widehat{F}_I(\widehat{B}).$  Also,
 $\widehat{F}_I(\widehat{B}) \ne \emptyset$ iff $|\widehat{B} \cap E(\widehat{H}_I)| = 1.$  
 
 \begin{lemma}
 Let $I = \{ i,j,k\}$ where $i<j<k$ and suppose $\widehat{B} \in \B(\widehat{M})$ is such that $\widehat{B}\cap E(\widehat{H}_I) = \{ \widehat{e} \}.$  Then $\widehat{F}_I(\widehat{B}) \backslash \{ \widehat{e}\}  \ne \emptyset$.\label{lem-Mhat1}
 \end{lemma}
 
 \begin{proof}
 Without loss of generality, we may assume that $I = \{ 1,2,3 \}$ and $\widehat{B} \cap E(\widehat{H}_I) = \{ \widehat{e}_{12} \}.$   By case checking, one can show that $\widehat{H}_I$ is either a $3$-cycle or, a $2$-cycle and a loop or, two stem loops joined by an edge or, three loops meeting at a vertex.  It is straightforward to show that in each of these cases
there is a circuit $\widehat{C}$ in $\widehat{M}$ for which  $\widehat{C} \subseteq E(\widehat{H}_I)$ and $\widehat{e}_{12} \in \widehat{C}.$   
Then there exists $\widehat{e}_{12}' \in \widehat{C}\backslash \{ \widehat{e}_{12} \}$ for which $\widehat{B} - \widehat{e}_{12} + \widehat{e}_{12}'$ is a base.   Thus $\widehat{F}_I(\widehat{B})\backslash \{ \widehat{e}_{12} \} \ne \emptyset.$ 
 \end{proof}
 
  Let $I \subseteq \{ 1,2, \dots ,m \}$ and let $\widehat{B}\in \B(\widehat{M}).$  We say that $\widehat{B}$ is $\mathbf{I}${\bf -cyclic} if there exists a cycle $\sigma \in S_I$ of length $|I|$ such that $\widehat{e}_{i\sigma(i)} \in \widehat{F}_I(\widehat{B}),\ \forall i\in I.$  We say that $\widehat{B}$ is $\mathbf{I}${\bf - singular} if there exists $i\in I$ such that $\widehat{e}_{ij} \in \widehat{F}_I(\widehat{B}), \ \forall j\in I\backslash \{ i \}.$  Moreover, $\widehat{B}$ is {\bf strictly} $\mathbf{I}${\bf -cyclic} (respectively, {\bf strictly} $\mathbf{I}${\bf -singular}) if $\widehat{B}$ is $I$ -cyclic (respectively, $I$ -singular) but not $I$-singular (respectively, $I$ -cyclic).
 
 \begin{lemma}
  Let $I \subseteq \{ 1, \dots, m \}$ where $|I| =4.$   Let $\widehat{B} \in \B(\widehat{M})$ where $\widehat{B} \cap E(\widehat{H}_I) = \{ \widehat{e} \}.$  
 Then $\widehat{B}$ is $I$-cyclic or $I$-singular.  Moreover, either 
 \begin{itemize}
 \item[i)]  there is a $4$-cycle $\sigma \in S_I$ such that $\widehat{e} \in \left\{ \widehat{e}_{i\sigma(i)}\ \big| \ i\in I \right\} \subseteq \widehat{F}_I(\widehat{B}),$
 or 
 \item[ii)] for some $i\in I,$ $\widehat{e} \in \left\{  \widehat{e}_{ij} \ \big| \  j\in I\backslash \{ i \} \right\} \subseteq \widehat{F}_I(\widehat{B})$.
 \end{itemize}
\label{lem-Mhat2} 
 \end{lemma}
 
 \begin{proof}
 It suffices to prove the lemma for $I = \{ 1,2,3,4 \}$ and $\widehat{e} = \widehat{e}_{12}.$  Let $\widehat{F} = \widehat{F}_I(\widehat{B}).$
 By Lemma \ref{lem-Mhat1},
 either $\widehat{e}_{13} \in \widehat{F}$ or $\widehat{e}_{23} \in \widehat{F}.$  Without loss of generality, we may assume that $\widehat{e}_{23}\in \widehat{F}.$  Let $\widehat{B}' = \widehat{B} - \widehat{e}_{12} + \widehat{e}_{23}.$  Then $\widehat{F}' = \widehat{F}_I(\widehat{B}') = \widehat{F}.$ It follows by Lemma \ref{lem-Mhat1} that either $\widehat{e}_{24} \in \widehat{F}'$ or $\widehat{e}_{34} \in \widehat{F}'.$  If $\widehat{e}_{24} \in \widehat{F}'$, then 
 $\widehat{e}_{24}\in \widehat{F}$ and hence ii) holds with $i=2$.   Suppose $\widehat{e}_{34} \in \widehat{F}'.$ Then $\widehat{e}_{34} \in \widehat{F}.$   Since $\widehat{e}_{12} \in \widehat{F},$ it follows by Lemma \ref{lem-Mhat1}
 that either $\widehat{e}_{24} \in \widehat{F}$ or $\widehat{e}_{14} \in \widehat{F}.$  If the former holds, then ii) holds with $i=2.$   If the latter holds, then we see that for $\sigma = (1,2,3,4)$, $\widehat{e}_{i\sigma(i)} \in \widehat{F},\ \forall i\in I,$  and hence i) holds.
 \end{proof}
 
 \subsection{Switchability and Amenability of Base Pairs in $\widehat{M}$}\label{subsec-switchamen}
  
 Let $\widehat{B}_i \in \B(\widehat{M}),\ i = 1,2$ be disjoint bases.  Let $I \subseteq \{ 1, \dots ,m\}$ and let $\widehat{H} = \widehat{H}_I$.   Suppose that $(\widehat{B}_1 \cup \widehat{B}_2)\cap E(\widehat{H}) = \{ \widehat{e}, \widehat{f} \}.$  If either $\widehat{e}$ and $\widehat{f}$ are non-incident or $\{\widehat{e}, \widehat{f} \} \subseteq \widehat{B}_i,$ for some $i \in \{ 1,2 \},$ then we say that the pair $(\widehat{B}_1, \widehat{B}_2)$ is $\mathbf{\widehat{H}}${\bf-viable}.  
 
 For distinct edges $\widehat{e}', \widehat{f}' \in E(\widehat{H}),$ we say that the pair $(\widehat{B}_1, \widehat{B}_2)$ is $\mathbf{\widehat{e}', \widehat{f}'}${\bf-amenable} 
 if there exist disjoint bases $\widehat{B}_i' \in \B(\widehat{M}),\ i = 1,2$ such that
 $\widehat{B}_i'\backslash E(\widehat{H}) =  \widehat{B}_i\backslash E(\widehat{H}),\ i = 1,2$ and $(\widehat{B}_i' \cup \widehat{B}_2') \cap E(\widehat{H}) = \{ \widehat{e}', \widehat{f}' \}.$ 
 We write $(\widehat{B}_1, \widehat{B}_2) \xrightarrow[\widehat{e}', \widehat{f}']{\widehat{e}, \widehat{f}} (\widehat{B}_1', \widehat{B}_2')$.
 

Let $(\widehat{A}_1, \widehat{A}_2)$ and $(\widehat{B}_1, \widehat{B}_2)$ be $\widehat{H}$-viable pairs of disjoint bases where\\ $(\widehat{A}_1 \cup \widehat{A}_2) \cap E(\widehat{H}) = \{ \widehat{a}_1, \widehat{a}_2 \}$ and $(\widehat{B}_1 \cup \widehat{B}_2) \cap E(\widehat{H}) = \{ \widehat{b}_1, \widehat{b}_2 \}.$  Suppose there exist pairs of bases $(\widehat{A}_1', \widehat{A}_2')$ and $(\widehat{B}_1', \widehat{B}_2')$ such that 
\begin{itemize}
\item $(\widehat{A}_1, \widehat{A}_2) \xrightarrow[\widehat{c}_1, \widehat{c}_2]{\widehat{a}_1, \widehat{a}_2} (\widehat{A}_1', \widehat{A}_2')$.
\item $(\widehat{B}_1, \widehat{B}_2) \xrightarrow[\widehat{c}_1, \widehat{c}_2]{\widehat{b}_1, \widehat{b}_2} (\widehat{B}_1', \widehat{B}_2')$. 
\item The pairs $(\widehat{A}_1', \widehat{A}_2')$ and $(\widehat{B}_1', \widehat{B}_2')$ are $\widehat{H}$-viable.
\item $(\widehat{A}_1', \widehat{A}_2') \sim_1 (\widehat{B}_1', \widehat{B}_2')$.
\end{itemize}
Then we write $(\widehat{A}_1, \widehat{A}_2) \curvearrowright_1 (\widehat{B}_1, \widehat{B}_2).$  If there exist base pairs $(\widehat{A}_1^j, \widehat{A}_2^j), \ j =1, \dots ,k$ such that
$$(\widehat{A}_1, \widehat{A}_2) \curvearrowright_1 (\widehat{A}_1^1, \widehat{A}_2^1) \curvearrowright_1 (\widehat{A}_1^2, \widehat{A}_2^2) \curvearrowright_1 \cdots \curvearrowright_1 (\widehat{A}_1^k, \widehat{A}_2^k) \curvearrowright_1 (\widehat{B}_1, \widehat{B}_2),$$ then we write $(\widehat{A}_1, \widehat{A}_2) \curvearrowright (\widehat{B}_1, \widehat{B}_2).$

Suppose that $\widehat{B}_i \in \B(\widehat{M}),\ i = 1,2$ are disjoint bases such that $\widehat{e} \in \widehat{F}_I(\widehat{B}_1)$ and $\widehat{f} \in \widehat{F}_I(\widehat{B}_2))$ where $\widehat{e}$ and $\widehat{f}$ are non-incident edges. We say that  $(\widehat{B}_1, \widehat{B}_2)$ is $\mathbf{\widehat{e}, \widehat{f}}${\bf-switchable} if there exists $(\widehat{B}_1', \widehat{B}_2')$
such that\\ $(\widehat{B}_1, \widehat{B}_2) \curvearrowright (\widehat{B}_1', \widehat{B}_2')$ and $(\widehat{B}_1', \widehat{B}_2')$ is $\widehat{e}', \widehat{f}'$-amenable for some non-incident edges $\widehat{e}',\widehat{f}' \in E(\widehat{H})\backslash \{ \widehat{e}, \widehat{f} \}.$   We shall need the following lemma.

 \begin{lemma}
 Let $I = \{ i_1,i_2,i_3,i_4 \} \subset \{ 1,2, \dots ,m\},$ $\widehat{H} = \widehat{H}_I,$ and let $\widehat{B}_i\in \B(\widehat{M}),\ i = 1,2$ be disjoint bases where $|\widehat{B}_i \cap E(\widehat{H})|=1,\ i = 1,2.$  If $(\widehat{B}_1,\widehat{B}_2)$ is not $\widehat{e}, \widehat{f}$-amenable for all $\{ \widehat{e}, \widehat{f} \} \in \left\{ \{ \widehat{e}_{i_1i_3}, \widehat{e}_{i_2i_4} \},  \{ \widehat{e}_{i_1i_4}, \widehat{e}_{i_2i_3} \} \right\}$, then one of two possibilities hold: 
 \begin{itemize}
 \item[i)] $\widehat{B}_i,\ i = 1,2$ are strictly $I$-singular and
 for some $i\in I,$ $\widehat{F}_I(\widehat{B}_1) = \widehat{F}_I(\widehat{B}_2) = \{ \widehat{e}_{ij}\ \big| \ j\in I\backslash \{ i \} \},$
 
 or 
 
 \item[ii)] $\widehat{B}_i,\ i = 1,2$ are strictly $I$-cyclic and \\$\{ \widehat{F}_I(\widehat{B}_1), \widehat{F}_I(\widehat{B}_2) \}  = \{ \{ \widehat{e}_{i_1i_2}, \widehat{e}_{i_3i_4}, \widehat{e}_{i_1i_3}, \widehat{e}_{i_2i_4} \},\  \{ \widehat{e}_{i_1i_2}, \widehat{e}_{i_3i_4}, \widehat{e}_{i_1i_4}, \widehat{e}_{i_2i_3} \} \}.$ 
 \end{itemize}
 \label{lem-Mhat3}
 \end{lemma}
 
\begin{proof}
 For convenience, we may assume that $i_j = j,\ j= 1,2,3,4$ and $(\widehat{B}_1, \widehat{B}_2)$ is not $\widehat{e}, \widehat{f}$-amenable for all $\{ \widehat{e}, \widehat{f} \} \in \left\{ \{ \widehat{e}_{13}, \widehat{e}_{24} \},  \{ \widehat{e}_{14}, \widehat{e}_{23} \} \right\}$.
 Let $\widehat{F}_i = \widehat{F}_I(\widehat{B}_i),\ i = 1,2.$  It follows by Lemma \ref{lem-Mhat2} that for $i = 1,2,$ $\widehat{B}_i$ is either $I$-cyclic or $I$-singular.  Suppose that $\widehat{B}_1$ is $I$-singular.  Without loss of generality, we may assume that $\widehat{e}_{1j} \in \widehat{F}_1,\ \forall j\in \{ 2,3,4 \} .$ If $\widehat{B}_2$ is also $I$-singular, then for some $i\in I,$ $\widehat{e}_{ij} \in \widehat{F}_2, \ \forall j\in I\backslash \{ i \}.$ If $i\in \{ 2,3 \},$ then $\widehat{e}_{14} \in \widehat{F}_1,$ and $\widehat{e}_{23}\in \widehat{F}_2.$  If $i=4,$ then $\widehat{e}_{13} \in \widehat{F}_1,$ and $\widehat{e}_{24}\in \widehat{F}_2.$  In both cases, $(\widehat{B}_1, \widehat{B}_2)$ is $\widehat{e},\widehat{f}$-amenable for some
 $\{\widehat{e}, \widehat{f}\} \in \left\{ \{ \widehat{e}_{13}, \widehat{e}_{24} \},  \{ \widehat{e}_{14}, \widehat{e}_{23} \} \right\}$,
 contradicting our assumptions.  It follows that $i =1$ and $ \widehat{e}_{1j} \in\widehat{F}_1\cap \widehat{F}_2, \ j = 1,2,3.$
Suppose instead that $\widehat{B}_2$ is $I$-cyclic.  Then it is seen that either $\{ \widehat{e}_{13}, \widehat{e}_{24} \} \subset \widehat{F}_2$ or $\{ \widehat{e}_{14}, \widehat{e}_{23} \} \subset \widehat{F}_2.$
 In the former case, $\widehat{e}_{13} \in \widehat{F}_1$ and  $\widehat{e}_{24}\in \widehat{F}_2$, implying that $(\widehat{B}_1, \widehat{B}_2)$ is $\widehat{e}_{13}, \widehat{e}_{24}$-amenable, a contradiction.
 In the latter case, $\widehat{e}_{14}\in \widehat{F}_1$ and $\widehat{e}_{23}\in \widehat{F}_2$,  implying that $(\widehat{B}_1, \widehat{B}_2)$ is $\widehat{e}_{14}, \widehat{e}_{23}$-amenable, a contradiction. We conclude that if $\widehat{B}_1$ is $I$-singular, then $\widehat{B}_2$ is strictly $I$-singular and $\widehat{F}_1 = \widehat{F}_2$ (and moreover, $\widehat{B}_1$ is strictly $I$-singular).   From this it also follows that if $\widehat{B}_1$ is $I$-cyclic, then both $\widehat{B}_i,\ i = 1,2$ are strictly $I$-cyclic and  $\{ \widehat{F}_I(\widehat{B}_1), \widehat{F}_I(\widehat{B}_2) \}  = \{ \{ \widehat{e}_{12}, \widehat{e}_{34}, \widehat{e}_{13}, \widehat{e}_{24} \},\  \{ \widehat{e}_{12}, \widehat{e}_{34}, \widehat{e}_{14}, \widehat{e}_{23} \} \}.$
 \end{proof}
  
 \begin{lemma}
 Let $\widehat{B} \in \B(\widehat{M})$ and let $\widehat{e} \in \widehat{B}$ and let $\widehat{e}' \in E(\widehat{M})\backslash \widehat{B}.$
 Suppose $\widehat{B}' = \widehat{B} - \widehat{e} + \widehat{e}' \in \B(\widehat{M}).$  Let $I \subseteq \{ 1, \dots ,m \}$ and let $\widehat{H} = \widehat{H}_I.$
 Suppose $\widehat{B}\cap E(\widehat{H}) =
  \{ \widehat{e}_{ij} \}$ and $\widehat{F}_I(\widehat{B}') \ne \emptyset.$
 
 \begin{itemize}
 \item[i)] If $\widehat{F}_I(\widehat{B})\backslash \widehat{F}_I(\widehat{B}') \ne \emptyset,$ then $\widehat{B} - \widehat{e}_{ij} + \widehat{e}' \in \B(\widehat{M}).$
 \item[ii)]  If $\widehat{e}_{i'j'} \in \widehat{F}_I(\widehat{B}) \backslash \widehat{F}_I(\widehat{B}'),$ then $\widehat{B} - \widehat{e} + \widehat{e}_{i'j'} \in \B(\widehat{M}).$
 \end{itemize}\label{lem-Mhat3.1}
 \end{lemma}
 
 \begin{proof}  We first remark that if $\widehat{e}_{ij} \not\in \widehat{B}',$ then $\widehat{e} = \widehat{e}_{ij}$ and $\widehat{e}' \in \widehat{F}_I(\widehat{B}')$ (since $\widehat{F}_I(\widehat{B}') \ne \emptyset$).  This would imply that $\widehat{F}_I(\widehat{B}) = \widehat{F}_I(\widehat{B}')$, which is contrary to our assumptions.  Thus $\widehat{e}_{ij} \in \widehat{B}'.$
 To prove i),  let $\widehat{C} = C(\widehat{e}', \widehat{B})$ be the fundamental circuit with respect to $\widehat{B}$ containing $\widehat{e}'.$  We observe that $\widehat{e} \in \widehat{C}.$  We shall show that $\widehat{e}_{ij} \in \widehat{C},$ from which is will follow that $\widehat{B} - \widehat{e}_{ij} + \widehat{e}' \in \B(\widehat{M}).$   Suppose to the contrary that $\widehat{e}_{ij} \not\in \widehat{C}.$   Let $\widehat{e}_{i'j'} \in \widehat{F}_I(\widehat{B})\backslash \widehat{F}_I(\widehat{B}').$
 Let $\widehat{C}' = C(\widehat{e}_{i'j'}, \widehat{B}'),$ the fundamental circuit with respect to $\widehat{B}'$ containing $\widehat{e}_{i'j'}.$  We observe that $\widehat{e}_{ij} \not\in \widehat{C}'$ since $\widehat{e}_{i'j'} \not\in \widehat{F}_I(\widehat{B}').$  However, $\widehat{e}' \in \widehat{C}';$ for otherwise, $\widehat{C}' \subseteq \widehat{B} -  \widehat{e}_{ij} + \widehat{e}_{i'j'},$ implying that $\widehat{e}_{i'j'} \not\in \widehat{F}_I(\widehat{B}),$ contradicting our assumptions.  By the circuit elimination property, there exists a circuit $\widehat{C}'' \subseteq (\widehat{C} \cup \widehat{C}')\backslash \{ \widehat{e}' \}$ where $\widehat{e}_{i'j'} \in \widehat{C}''.$  However, we see that $\widehat{C}'' \subseteq \widehat{B} -  \widehat{e}_{ij} + \widehat{e}_{i'j'},$ implying that $\widehat{e}_{i'j'} \not\in \widehat{F}_I(\widehat{B}),$ again contradicting our assumptions.  Thus $\widehat{e}_{ij} \in \widehat{C}.$ This completes the proof of i).
 
 To prove ii), suppose to the contrary that $\widehat{B}'' = \widehat{B} - \widehat{e} + \widehat{e}_{i'j'} \not\in \B(\widehat{M}).$  Then 
 $\widehat{e} \not\in \widehat{C} = C(\widehat{e}_{i'j'}, \widehat{B}).$  Furthermore, since $\widehat{e}_{i'j'} \in \widehat{F}_I(\widehat{B}),$ it follows that $\widehat{e}_{ij} \in \widehat{C}.$  Since $\widehat{e}_{i'j'} \not\in \widehat{F}_I(\widehat{B}'),$ it follows that $\widehat{e}_{ij} \not\in \widehat{C}' = C(\widehat{e}_{i'j'}, \widehat{B}').$  By the circuit elimination property, there exists a circuit $\widehat{C}'' \subseteq (\widehat{C} \cup \widehat{C}')\backslash \{ \widehat{e}_{i'j'} \}$ where $\widehat{e}_{ij} \in \widehat{C}''$.  However, $\widehat{C}'' \subseteq \widehat{B}',$ a contradiction.  Thus $\widehat{B}'' \in \B(\widehat{M}).$  This proves ii).  
 \end{proof}
 
 The following is an important theorem which we will need later:
 
 \begin{theorem}
  Let $\widehat{B}_i \in \B(\widehat{M}),\ i  = 1,2$ and let $I = \{ i_1,i_2,i_3,i_4 \}.$   Let $\widehat{F}_i = \widehat{F}_I(\widehat{B}_i),\ i = 1,2$ and assume that $\widehat{e}_{i_1i_2} \in \widehat{F}_1$ and $\widehat{e}_{i_3i_4} \in \widehat{F}_2.$   Let $\widehat{H} = \widehat{H}_I.$
 Suppose that $(\widehat{B}_1, \widehat{B}_2)$ is not $\widehat{e}_{i_1i_2}, \widehat{e}_{i_3i_4}$-switchable.
 Let $(\widehat{B}_1' , \widehat{B}_2')$ be a pair of bases obtained from performing one symmetric exchange between $\widehat{B}_1$ and  $\widehat{B}_2$ and let $\widehat{F}_i' = \widehat{F}_I(\widehat{B}_i'),\ i = 1,2.$  Then either $\{ \widehat{F}_1', \widehat{F}_2'\} = \{ \widehat{F}_1, \widehat{F}_2\}$ or $\widehat{F}_1' = \widehat{F}_2' = \emptyset.$  \label{the-Mhat5}
 \end{theorem}
 
 \begin{proof}
 For convenience, we may assume that $i_j =j,\ j = 1,2,3,4$ and $(\widehat{B}_1, \widehat{B}_2)$ is not $\widehat{e}_{12}, \widehat{e}_{34}$-switchable.  By Lemma \ref{lem-Mhat3}, $\widehat{B}_i,\ i= 1,2$ are strictly $I$-cyclic, and we may assume $\widehat{F}_1 = \{ \widehat{e}_{12}, \widehat{e}_{34}, \widehat{e}_{23}, \widehat{e}_{14} \}, \  \widehat{F}_2 = \{ \widehat{e}_{12}, \widehat{e}_{34}, \widehat{e}_{13}, \widehat{e}_{24} \}.$  We observe that 
 $\widehat{F}_i = E(\widehat{H}) \backslash \mathrm{cl}_{\widehat{M}}(\widehat{B}_i\backslash E(\widehat{H})),\ i = 1,2.$  Since $(\widehat{B}_1', \widehat{B}_2')$ is obtained from $\widehat{B}_1$ and $\widehat{B}_2$ by a single symmetric exchange, we 
have that $\widehat{B}_1' = \widehat{B}_1 -\widehat{e} + \widehat{f}$ and $\widehat{B}_2' = \widehat{B}_2 -\widehat{f} + \widehat{e}$ for some $\widehat{e} \in \widehat{B}_1$ and $\widehat{f}\in \widehat{B}_2.$  We may assume that $\widehat{F}_i' \ne \emptyset,\ i = 1,2$.   It remains to show that $\{ \widehat{F}_1', \widehat{F}_2' \} = \{ \widehat{F}_1, \widehat{F}_2 \}.$  We first note that $(\widehat{B}_1', \widehat{B}_2')$ is not $\widehat{e}_{12}, \widehat{e}_{34}$-switchable.  Thus it follows by Lemma \ref{lem-Mhat3} that $\widehat{B}_i',\ i = 1,2$ are strictly $I$-singular, or they are strictly $I$-cyclic.  If the latter holds, then it follows that  $\{ \widehat{F}_1', \widehat{F}_2' \} = \{ \widehat{F}_1, \widehat{F}_2 \}.$  Suppose the former holds.  Given that $(\widehat{B}_1', \widehat{B}_2')$ is not $\widehat{e}_{12}, \widehat{e}_{34}$-switchable, we may assume that for $i= 1,2,$ $\widehat{F}_i' = \{ \widehat{e}_{12}, \widehat{e}_{13}, \widehat{e}_{14} \}.$  Given that $\widehat{F}_I(\widehat{B}_1) \cap \widehat{F}_I(\widehat{B}_1') = \{ \widehat{e}_{12}, \widehat{e}_{14} \}$ and $\widehat{F}_I(\widehat{B}_2) \cap \widehat{F}_I(\widehat{B}_2') = \{ \widehat{e}_{12}, \widehat{e}_{13} \}$, we may assume that $\widehat{e}_{12} \in \widehat{B}_1$ and $\widehat{e}_{13} \in \widehat{B}_2.$

Since $\widehat{F}_1\backslash \widehat{F}_1'  \ne \emptyset,$ it follows by Lemma \ref{lem-Mhat3.1} i) that $\widehat{B}_1''= \widehat{B}_1 - \widehat{e}_{12} + \widehat{f} \in \B(\widehat{M}).$  
  Since $\widehat{e}_{24} \in \widehat{F}_I(\widehat{B}_2)\backslash \widehat{F}_I(\widehat{B}_2'),$ it follows by Lemma \ref{lem-Mhat3.1} ii) that $\widehat{B}_2'' = \widehat{B}_2 - \widehat{f} + \widehat{e}_{24}  \in \B(\widehat{M}).$  Thus $\{ \widehat{e}_{13}, \widehat{e}_{24} \} \subset \widehat{B}_2''.$  To obtain a contradiction, we need only show that $(\widehat{B}_1, \widehat{B}_2) \curvearrowright (\widehat{B}_1'', \widehat{B}_2'').$  To see this, we note that either $\widehat{B}_2'' - \widehat{e}_{24} + \widehat{e}_{12} \in \B(\widehat{M})$ or $\widehat{B}_2'' - \widehat{e}_{24} + \widehat{e}_{23} \in \B(\widehat{M}).$  We may assume the former holds.  Let $\widehat{B}_1''' = \widehat{B}_1'',$ and $\widehat{B}_2''' = \widehat{B}_2'' - \widehat{e}_{24} + \widehat{e}_{12}.$
  Then $(\widehat{B}_1, \widehat{B}_2) \sim_1 (\widehat{B}_1''', \widehat{B}_2''')$ and we see that  $(\widehat{B}_1, \widehat{B}_2) \curvearrowright_1 (\widehat{B}_1'', \widehat{B}_2'').$  However, this implies that $(\widehat{B}_1, \widehat{B}_2)$ is $\widehat{e}_{12}, \widehat{e}_{34}$-switchable, yielding a contradiction.
This completes the proof. 
 \end{proof}
 
 \subsection{Base Pairs $(\widehat{B}_1, \widehat{B}_2)$ which are not $\widehat{e}_{i_1i_2}, \widehat{e}_{i_3i_4}$-switchable}\label{sec-note12e34switch}
 
 As before, suppose $\widehat{B}_i \in \B(\widehat{M}),\ i = 1,2$ are disjoint bases.  Let $I = \{ i_1,i_2,i_3,i_4 \},\ \widehat{H} = \widehat{H}_I$ and assume that $\widehat{B}_1 \cap E(\widehat{H}) = \{ \widehat{e}_{i_1i_2} \}$ and\\ $\widehat{B}_2 \cap E(\widehat{H}) = \{ \widehat{e}_{i_3i_4} \}$.  When $(\widehat{B}_1, \widehat{B}_2)$ is not $\widehat{e}_{i_1i_2}, \widehat{e}_{i_3i_4}$-switchable, the bases $\widehat{B}_i,\ i = 1,2$ have a particular structure, something which we will exploit later on.
 
For $i = 1,2$, let $\widehat{C}_i$ be the (unique) circuit in $\widehat{B}_i \cup \{ \widehat{e}_{i_1i_2}, \widehat{e}_{i_3i_4} \}$ and let $\widehat{L}_i$ be the subgraph of $\widehat{G}$ induced by $\widehat{C}_i.$
 
 \begin{theorem}
 Suppose $(\widehat{B}_1, \widehat{B}_2)$ is not $\widehat{e}_{i_1i_2}, \widehat{e}_{i_3i_4}$-switchable.  Then for $i = 1,2,$ $\widehat{L}_i$ consists of two vertex-disjoint cycles
 in $\widehat{G}$ joined by a non-trivial path, each cycle of $\widehat{L}_i$ containing exactly one edge of $\{ \widehat{e}_{i_1i_2}, \widehat{e}_{i_3i_4} \}$.  \label{the-note12e34switch}
 \end{theorem}

\begin{proof}
For convenience, we shall assume that $\widehat{e}_{i_j} = \widehat{e}_j,\ j=1,2,3,4.$ Let $v_i = v(e_i),\ i = 1,2,3,4.$  By Lemma \ref{lem-Mhat3}, we may assume that $\widehat{F}_I(\widehat{B}_1) = \{ \widehat{e}_{12},\widehat{e}_{34}, \widehat{e}_{14}, \widehat{e}_{23} \}$ and $\widehat{F}_I(\widehat{B}_2) = \{ \widehat{e}_{12},\widehat{e}_{34}, \widehat{e}_{13}, \widehat{e}_{24} \}$.  
By assumption,\\ $\widehat{B}_1 - \widehat{e}_{12} + \widehat{e}_{13}$ contains a (unique) circuit which we denote by $\widehat{C}_{13}.$  Likewise,  $\widehat{B}_1 - \widehat{e}_{12} + \widehat{e}_{24}$ contains a circuit which we denote by $\widehat{C}_{24}.$  Similarly, let $\widehat{D}_{14}$ (respectively, $\widehat{D}_{23}$) be the circuit contained in $\widehat{B}_2 - \widehat{e}_{34} + \widehat{e}_{14}$ (respectively, $\widehat{B}_2 - \widehat{e}_{34} + \widehat{e}_{23}$).  

Let $\widehat{B}_1' = \widehat{B}_1 - \widehat{e}_{12} + \widehat{e}_{23}$ and  $\widehat{B}_2' = \widehat{B}_2 - \widehat{e}_{34} + \widehat{e}_{12}.$  By the symmetric exchange property, there exists $\widehat{e} \in \widehat{B}_1'$ such that $\widehat{B}_1'' = \widehat{B}_1' - \widehat{e} + \widehat{e}_{12}$ and $\widehat{B}_2'' = \widehat{B}_2' - \widehat{e}_{12} + \widehat{e}$ are bases.  Clearly $\widehat{e} \ne \widehat{e}_{23}$ since $\widehat{e}_{23} \not\in \widehat{F}_I(\widehat{B}_2).$  However, it is seen that $\widehat{e} \in \widehat{C}_{13}.$


Suppose that $\widehat{C}_{13} \cap \widehat{C}_{24} = \emptyset.$  Then $\widehat{C}_{24} \backslash \{ \widehat{e}_{24} \} \subseteq \widehat{B}_1''.$  By Lemma \ref{lem-Mhat1}, either $\widehat{B}_1'' - \widehat{e}_{12} + \widehat{e}_{14} \in \B(\widehat{M})$ or $\widehat{B}_1'' - \widehat{e}_{12} + \widehat{e}_{24} \in \B(\widehat{M})$.  The latter cannot occur since $\widehat{C}_{24} \subseteq \widehat{B}_1'' - \widehat{e}_{12} + \widehat{e}_{24}.$  Thus $\widehat{B}_1''' = \widehat{B}_1'' - \widehat{e}_{12} + \widehat{e}_{14} \in \B(\widehat{M})$.  Let $\widehat{B}_2''' = \widehat{B}_2''.$  We see that $\{ \widehat{e}_{14}, \widehat{e}_{23} \} = \widehat{B}_1''' \cap E(\widehat{H})$ and $(\widehat{B}_1, \widehat{B}_2) \curvearrowright_1 (\widehat{B}_1''', \widehat{B}_2''')$. Thus it follows that $(\widehat{B}_1, \widehat{B}_2)$ is $\widehat{e}_{12}, \widehat{e}_{34}$-switchable, a contradiction.
We conclude that $\widehat{C}_{13} \cap \widehat{C}_{24} \ne \emptyset,$ and moreover, the above arguments imply that $\widehat{e} \in \widehat{C}_{13} \cap \widehat{C}_{24}.$  In a similar fashion, one can show that $\widehat{D}_{14} \cap \widehat{D}_{23} \ne \emptyset.$

%

Suppose $\widehat{L}_1$ is a balanced cycle.  Then $\widehat{L}_1 - \widehat{e}_{12} - \widehat{e}_{34}$ consists of two vertex-disjoint paths $\widehat{P}_1$ and $\widehat{P}_2$, where for $i =1,2,$ $\widehat{P}_i$ originates at $v_i$ and terminates at a vertex in $\{ v_3, v_4 \}.$  If $\widehat{P}_1$ terminates at $v_4$ (and $\widehat{P}_2$ at $v_3$), then $\widehat{B}_1 - \widehat{e}_{12} + \widehat{e}_{13}$ is seen to be a base in $\widehat{M},$ contradicting the fact that $\widehat{e}_{13} \not\in \widehat{F}_I(\widehat{B}_1).$  Thus $\widehat{P}_1$ terminates at $v_3$ and $\widehat{P}_2$ at $v_4.$  However, this implies that $\widehat{C}_{13} = E(\widehat{P}_1) \cup \{ \widehat{e}_{13} \}$ and $\widehat{C}_{24} = E(\widehat{P}_2) \cup \{ \widehat{e}_{24} \},$ implying that $\widehat{C}_{13} \cap \widehat{C}_{24} = \emptyset,$ a contradiction.  Thus $\widehat{L}_1$ is not a balanced cycle and consequently, $\widehat{L}_1$ consists of three balanced cycles forming a theta graph, or two vertex-disjoint unbalanced cycles joined by a path, or two unbalanced cycles meeting at exactly one vertex.  If the last case occurs, then one can show that either $\{ \widehat{e}_{13}, \widehat{e}_{24} \} \cap \widehat{F}_I(\widehat{B}_1)
\ne \emptyset,$ or $\widehat{C}_{13} \cap \widehat{C}_{24} = \emptyset.$  Thus only the first or second case can occur.  Arguing the same way, one can show that $\widehat{L}_2$ consists of three unbalanced cycles forming a theta graph, or it consists of two vertex-disjoint unbalanced cycles joined by a path.
  
Suppose $\widehat{L}_1$ is a theta graph.  If there is a cycle in $\widehat{L}_1$ which does not contain $\widehat{e}_{12}$ or $\widehat{e}_{34},$ then we see again that either $\{ \widehat{e}_{13}, \widehat{e}_{24} \} \cap \widehat{F}_I(\widehat{B}_1)
\ne \emptyset,$ or $\widehat{C}_{13} \cap \widehat{C}_{24} = \emptyset.$  Thus we may assume that $\widehat{L}_{11}$ and $\widehat{L}_{12}$ are distinct cycles in $\widehat{L}_1$ such that $\widehat{e}_{(2i-1)2i} \in E(\widehat{L}_{1i}),\ i = 1,2$ and $\{ \widehat{e}_{12}, \widehat{e}_{34} \} \subseteq E(\widehat{L}_{11} \triangle \widehat{L}_{12} ).$  Since $\widehat{C}_{13} \cap \widehat{C}_{24} \ne \emptyset,$ it is seen that $\widehat{L}_{11} \triangle \widehat{L}_{12} - \widehat{e}_{12} - \widehat{e}_{34}$ consists of two vertex-disjoint paths; one going from $v_1$ to $v_4$ and the other from $v_2$ to $v_3.$
Let $\widehat{P}_1$ be the path from $v_1$ to $v_4.$  Note that $\widehat{C}_{14} = E(\widehat{P}_1) \cup \{ \widehat{e}_{14}\}$ is not a circuit and corresponds to an unbalanced cycle $\widehat{C}_{14}'$ since $\widehat{e}_{14} \in \widehat{F}_I(\widehat{B}_1).$

Irregardless of whether $\widehat{L}_2$ is a theta graph or not, there is a path $\widehat{P}_2$ in $\widehat{L}_2 - \widehat{e}_{12} - \widehat{e}_{34}$ from $v_1$ to $v_4.$  We observe that $\widehat{D}_{14} = E(\widehat{P}_2) \cup \{ \widehat{e}_{14}\}$ is a circuit corresponding to a balanced cycle $\widehat{D}_{14}'$ since $\widehat{e}_{14} \not\in \widehat{F}_I(\widehat{B}_2).$  Let $\widehat{L}_1'$ be the component of $\widehat{L}_1 - \widehat{e}_{12} - \widehat{e}_{34} -e$ containing $\widehat{P}_1$.  Suppose there is a vertex of $\widehat{P}_2$ not belonging to $L_1'.$  Let $\widehat{f} = uu'$ be the first such edge we encounter as we move from $v_1$ to $v_4$ along $\widehat{P}_2$ where $u\in V(\widehat{L}_1')$ and $u' \not\in V(\widehat{L}_1').$  Then it is seen that $\widehat{B}_1''' = \widehat{B}_1'' - \widehat{e}_{12} +\widehat{f}$ and $\widehat{B}_2''' = \widehat{B}_2'' - \widehat{f} + \widehat{e}_{14}$ are bases in $\widehat{M}.$  However, we see that  $\widehat{e}_{14} \in \widehat{B}_2'''$ and $\widehat{e}_{23} \in  \widehat{B}_1'''$, and one can easily show that $(\widehat{B}_1, \widehat{B}_2) \curvearrowright  (\widehat{B}_1''', \widehat{B}_2''')$, implying that $(\widehat{B}_1, \widehat{B}_2)$ is $\widehat{e}_{12}, \widehat{e}_{34}$-switchable.  This gives a contradiction.  Thus $V(\widehat{P}_2) \subseteq V(\widehat{L}_1').$
The cycles  $\widehat{C}_{14}'$ and $\widehat{D}_{14}'$ are unbalanced and balanced, respectively, and intersect exactly on the edge $\widehat{e}_{14}.$  
Using Lemma \ref{lem-lemma2} one can show that there exists $\widehat{f} \in E(\widehat{P}_2)$ such that $C(\widehat{f}, \widehat{L}_1')$ is an unbalanced cycle.  Again, we see that 
$\widehat{B}_1''' = \widehat{B}_1'' - \widehat{e}_{12} +\widehat{f}$ and $\widehat{B}_2''' = \widehat{B}_2'' - \widehat{f} + \widehat{e}_{14}$ are bases in $\widehat{M},$
yielding a contradiction.  It follows that $\widehat{L}_1$ is not a theta graph  and thus it consists of two vertex-disjoint cycles joined by a path, each cycle containing exactly one of the edges $\widehat{e}_{12}$ or $\widehat{e}_{34}.$  By symmetry, the same applies to $\widehat{L}_2.$
\end{proof}

\section{The Pull-Back Operation}\label{sec-pullback}
 
  In this section, we describe the {\it pull-back} operation where we associate to each base $\widehat{B} \in \B(\widehat{M})$ a collection of bases $\fB_M(\widehat{B}) \subseteq \B(M).$  
  
 \subsection{The Base-Set Pull-Back of $\widehat{B}$}\label{subsec-basesetpullback} 
  For a subset $\widehat{X} \subseteq E(\widehat{G}),$ we let $\underline{\widehat{X}} = \widehat{X}\backslash \widehat{E}.$ Let $\widehat{B} \in \B(\widehat{M}).$     If $\widehat{B} \cap \widehat{E} = \emptyset,$ we define $\fB_M(\widehat{B}) = \{ \widehat{B} \cup \{ e \}\ \big| \ e\in E_G(v) \}.$  If $\widehat{B} \cap \widehat{E} \ne \emptyset,$ we define $$\fB_M(\widehat{B}) = \{ \underline{\widehat{B}} \cup A\ \big| \ A \subseteq \bigcup_{\widehat{e}\in \widehat{B} \cap \widehat{E}}\fE(\widehat{e}),\  \underline{\widehat{B}} \cup A \in \B(M) \}.$$  From the definition, we see that $\fB_M(\widehat{B}) \subseteq \B(M).$  We call $\fB_M(\widehat{B})$ the {\bf base-set pullback} of $\widehat{B}.$ 

\begin{lemmanopar}
For all $\widehat{B} \in \B(\widehat{M}),$ $\fB_M(\widehat{B}) \ne \emptyset.$\label{lem-pullbacknonempt}
\end{lemmanopar}

 \begin{proof} Let $B \in \B(\widehat{M}).$  We first observe that the lemma holds if $\widehat{B} \cap \widehat{E} = \emptyset.$   Thus we may assume that $\widehat{B} \cap \widehat{E} \ne \emptyset.$  Let $\widehat{G}_i,\ i = 1, \dots ,s$ be the components of $\widehat{G}[\widehat{B}]$ and let $\widehat{C}_i$ be the (unbalanced) cycle contained in $\widehat{G}_i.$  We claim that there exists $A \subseteq \bigcup_{\widehat{e} \in \widehat{B} \cap \widehat{E}}\fE(\widehat{e})$ where $|A| = |\widehat{B} \cap \widehat{E}| +1$ and  $\underline{\widehat{B}} \cup A \in \B(M).$  Let $\widehat{E}_i = E(\widehat{G}_i) \cap \widehat{E}, \ i = 1,\dots ,s.$  We may assume that $\widehat{E}_1 \ne \emptyset.$  If $E(\widehat{C}_1) \cap \widehat{E}_1 =\emptyset,$  then it is easy to see that one can choose a set $A_1 \subset \bigcup_{\widehat{e} \in \widehat{E}_1}\fE(\widehat{e})$ where $|A_1| = |\widehat{E}_1| +1$ and the subgraph $G_1$ induced by $(E(\widehat{G}_1)\backslash \widehat{E}_1)\cup A_1$ is connected and contains exactly one (unbalanced) cycle, namely $\widehat{C}_1.$  
It is straightforward to show that for $i = 2, \dots ,s$ one can choose $A_i \subset \bigcup_{\widehat{e} \in \widehat{E}_i}\fE(\widehat{e})$ such that $|A_i| = |\widehat{E}_i|$ and all the components of the subgraph $G_i$ induced by $(E(\widehat{G}_i)\backslash \widehat{E}_i)\cup A_i$, with the exception of the component containing $v$ (which contains no cycles), contain exactly one unbalanced cycle.  Then $A = A_1 \cup \cdots \cup A_s$ is the desired set.  On the other hand, suppose that $E(\widehat{C}_1) \cap \widehat{E} \ne \emptyset.$  Let $C_1 = \fE(\widehat{C}_1)$ and let $C_1 = C_{11} \cup C_{12} \cup \cdots \cup C_{1t}$ be the cycle decomposition for $C_1;$ that is $C_{11}, \dots ,C_{1t}$ are the petals of $C_1.$  Since $\widehat{C}_1 \not\in \widehat{\C},$ it follows that $C_{1j} \not\in \C$ for some $j\in \{ 1, \dots ,t\}.$  Without loss of generality, we may assume that $C_{11} \not\in \C$ (i.e. $C_{11}$ is unbalanced).   In this case, we may choose $A_1$ such that $E(C_{11}) \cap E_G(v) \subseteq A_1,$ $|A_1| = |\widehat{E}_1| +1,$ and the subgraph $G_1$ induced by $(E(\widehat{G}_1)\backslash \widehat{E}_1)\cup A_1$ is connected and contains exactly one (unbalanced) cycle, namely $C_{11}.$   For $i = 2, \dots ,s$ we can choose $A_i \subset \bigcup_{\widehat{e}\in \widehat{E}_i}\fE(\widehat{e})$ as before.   Then $A = A_1 \cup \cdots \cup A_s$ is the desired set.
 \end{proof}
 
 An important property that we will need in our reduction is that for any base $B$ in $M$, there is there a base $\widehat{B}$ in $\widehat{M}$ such that $B$ is in the base-set pull-back of $\widehat{B}.$  In the next lemma, we affirm this property.
 
 \begin{lemmanopar}
 Let $B\in \B(M).$  Then there is a base $\widehat{B} \in \B(\widehat{M})$ such that $B \in \fB_M(\widehat{B}).$  Furthermore, if $B$ contains a cycle containing edges $e_i$ and $e_j$, then $\widehat{B}$ can be chosen such that $\widehat{e}_{ij} \in \widehat{B}.$ \label{lem-pullbacknonempt2}
 \end{lemmanopar}
 
 \begin{proof}
 Let $\underline{B} = B\backslash \{ e_1, \dots ,e_m \}$.  By Lemma \ref{lem-newbiased}, each (unbalanced) cycle contained in $\underline{B}$ is unbalanced in $\widehat{M}.$  Thus if all the cycles in $B$ are contained also in $\underline{B}$, then it straightforward to see that there is a subset $\widehat{A} \subseteq \widehat{E}$ for which $\widehat{B} = \underline{B} \cup \widehat{A} \in \B(\widehat{M})$ and $B \in \fB_M(\widehat{B}).$  Suppose instead that a cycle $C$ in $B$ contains $v$.  If $C$ is a loop, then we can easily adapt the previous argument.  Assume without loss of generality that $\{ e_1, e_2 \} \subseteq E(C).$  Then $\widehat{C} = (C\backslash v) \cup \{ \widehat{e}_{12} \}$ is a cycle in $\widehat{G}.$  Furthermore, since $\fE(\widehat{C}) = C,$ it follows by Lemma \ref{lem-newbiased} that $\widehat{C}$ is balanced.  Now we see that there is a subset $\widehat{A} \subseteq \widehat{E}$ where $\widehat{e}_{12} \in \widehat{A},\ \widehat{B} = \underline{B} \cup \widehat{A} \in \B(\widehat{M}),$ and $B \in \fB_M(\widehat{B}).$
 \end{proof}

\subsection{The Base-Set Pull-Back of $\widehat{\B}$}\label{subsec-basesetpullback2}

For a $\kappa$-tuple of disjoint bases $\widehat{\B} = (\widehat{B}_1, \dots ,\widehat{B}_{\kappa})$ in $\widehat{M}$, we define the {\bf base-set pullback} $\fB_M(\widehat{\B})$ of $\widehat{\B}$ to be the set of $\kappa$-tuples 
$(B_1,B_2, \dots ,B_{\kappa}) \in \fB_M(\widehat{B}_1) \times \fB_M(\widehat{B}_2) \times \cdots \times \fB_M(\widehat{B}_k)$ where $B_i \cap B_j =\emptyset,$ for all $i\ne j.$  We note that because of the requirement that
the bases $B_1, \dots ,B_{\kappa}$ be pairwise disjoint, it is possible that $\fB_M(\widehat{\B}) = \emptyset.$

\section{The Biased Graph $\Omega_{\kappa} = (G_\kappa, \C_\kappa)$ and $V$-Reduced Sequences of Bases}\label{sec-reductionsI}

In this section, we initiate the reduction step of the proof of Theorem \ref{the-main2}.   Let $\Omega
= (G,\C)$ be a linear biased graph and let $\kappa$ be an integer where $\kappa \ge 2$ and $n = |V(G)|.$  Let $M = M(\Omega)$.  Given that Theorem \ref{the-main2} is true for graphic matroids, we shall assume that $M$ is non-graphic.  Furthermore, we may assume that $G$ is connected and hence $r(M) = |V(G)| =n.$  Let
$\B_i = ( B_{i1}, \dots ,B_{i\kappa} ),\ i = 1,2$ be two compatible
sequences of bases.  We shall define another linear biased graph $\Omega_{\kappa} =
(G_{\kappa},\C_\kappa)$ using $\Omega$ as follows: Let $V(G_{\kappa}) = V(G).$   For all edges
$e\in E(G)$ let $e\in E(G_{\kappa})$ and furthermore, add $\omega(e)-1$ {\it parallel} edges to
$e$ where $\omega(e)$ is the number of bases in $\B_1$ (or $\B_2$)
which contain the edge $e.$  When $e$ is a loop, each parallel edge is a loop at the same vertex as $e$. Let $E_{G_\kappa}(e)$ denote the set of $\omega(e)$ parallel edges (which includes $e$) and let $\C_{\kappa}'$ denote the set of $2$-subsets $C'$ where $C' \subseteq E_{G_\kappa}(e)$ for some $e\in E(G).$  Let $\C_\kappa$ be the linear completion of $\C \cup \C_\kappa'$ and let $\Omega_\kappa = (G_\kappa, \C_\kappa)$ and 
$M_\kappa = M(\Omega_\kappa).$

Let $\B_i = (B_{i1}, \dots ,B_{i\kappa}),\ i = 1,2$ be the corresponding base sequences in $M_\kappa$ where the bases of $\B_i$ partition $E(M_\kappa).$
\begin{lemmanopar}
There exists $v\in V(G_\kappa)$ such that $d_{G_\kappa}(v) \le
2\kappa.$\label{lemma1}
\end{lemmanopar}

\begin{proof}
Since the base sequences $\B_i,\ i = 1,2$ partition $E(G_\kappa)$, we have $|E(G_{\kappa})| =\kappa r(M) = \kappa n$.  Thus $\sum_{u\in V(G_{\kappa})}
d_{G_{\kappa}}(u) = 2|E(G_{\kappa})| = 2\kappa n.$  Accordingly, the average degree of $G_{\kappa}$ is
$2\kappa$ and hence there is a vertex of degree at most $2\kappa.$ 
\end{proof}

\subsection{$V$-Reduced Sequences}\label{subsec-vreduced}

For the remainder, we shall assume that $v$ is a vertex
of $G_{\kappa}$ having $m$ incident edges where $m \le 2\kappa.$  Let $E_{G_\kappa}(v) = \{ e_1, \dots ,e_m \}$. We
say that a base $B\in \B(M_\kappa)$ is $\mathbf{v}${\bf -reduced} if $|B \cap E_{G_\kappa}(v)|\le 2.$  Moreover, a  
sequence of disjoint bases $\B = ( B_1, \dots ,B_{\kappa} )$ in $M_{\kappa}$
is $v$-reduced if for $i = 1, \dots, \kappa,$ $B_i,$ is $v$-reduced.   Let $\widehat{\Omega}_\kappa = (\widehat{G}_\kappa, \widehat{\C}_\kappa)$ be the $v$-deleted biased graph obtained from $\Omega_\kappa$ and let $\widehat{M}_\kappa = M(\widehat{\Omega}_\kappa)$.  We shall need the following lemma:

\begin{lemma}
Let $\B = ( B_1, \dots ,B_{\kappa} )$ be a sequence of disjoint bases of $M_\kappa$. Then there exists a $v$-reduced sequence
of bases $\B'$ for which $\B \sim \B'.$
\label{lemma2}
\end{lemma}

\begin{proof}
For $i = 1,\dots ,\kappa,$ let $G_i = G_\kappa[B_i]$ denote the subgraph of $G_\kappa$ induced by $B_i.$
We may assume that $\B$ is not $v$-reduced and contains a base, say
$B_1$, where $|B_1 \cap E_{G_\kappa}(v)| \ge 3.$ Since
$m \le 2\kappa,$ there is a base in $\B$, say $B_2$, for which $|B_2 \cap E_{G_\kappa}(v)| =1.$
Let $B_1 \cap E_{G_{\kappa}}(v)
= \{ e_1, \dots ,e_s \}$ and let $B_2 \cap E_{G_k}(v) = \{ f \}.$  We remark that $\{ e_1, \dots ,e_s \}$ may contain at most one loop,  and if it does, then we shall assume that it is $e_1.$
Let $e_i = vv_i,\ i = 1, \dots ,s$ and let $f = vu$, noting that possibly $v_1 = v.$   It suffices to show that we can replace
$B_i,\ i =1,2$ in $\B$ with bases $B_i',\ i = 1,2$ such that $(B_1', B_2') \sim (B_1,B_2),$ and  $|B_i'\cap E_{G_\kappa}(v)| = |B_i\cap E_{G_\kappa}(v)| + 2i-3,\ i = 1,2.$ 

Let $K_1, \dots ,K_p$ (resp. $L_1,\dots ,L_q$) denote the
components of $G_{\kappa}[B_1]$ (resp. $G_\kappa[B_2]$).  We may assume that $v\in V(K_1)$ and $v\in V(L_1).$    For $i=1, \dots ,p$ let $C_i$ denote the (unbalanced)
cycle in $K_i$ and for $i = 1, \dots ,q$, let $D_i$ denote the (unbalanced) cycle in $L_i$.   Let $K_{1i},\ i = 1, \dots ,s'$ denote the components of $K_1 -v.$

\vspace{0.1in}

\noindent{\bf Case 1}:  Suppose $v\not\in V(C_1).$

\vspace{0.1in}

In this case, $s' =s \ge 3$ and there exists $i\in \{ 1, \dots ,s \}$ such that $u \not\in V(K_{1i})$ and $C_1$ is not contained in $K_{1i}.$    We may assume $v_i \in V(K_{1i}).$  By the symmetric exchange property, there exists $e\in B_2$ such $B_1' = B_1 -e_i + e$ and $B_2' = B_2-e+e_i$ are bases in $M_\kappa.$
Since $u\not\in V(K_{1i})$ and $C_1$ is not contained in $K_{1i},$ it follows that $e \ne f.$  Thus $B_i', \ i = 1,2$ are seen to be the desired bases.

\vspace{0.1in}

\noindent{\bf Case 2}:  Suppose $v\in V(C_1).$

\vspace{0.1in}

If $e_1$ is a loop, then $E(C_1) = \{ e_1 \}$ and none of the components $K_{1i}, \ i = 1, \dots ,s'$ contain $C_1.$  Also, $s' \ge 2,$ and hence there is a component $K_{1i}$ where $u\not\in V(K_{1i}).$  Now we can find $B_i',\ i = 1,2$ as in Case 1.  

By the above, we may assume that $e_1$ is not a loop, and hence none of $e_1,\dots ,e_s$ are loops.   Without loss of generality, we may assume $e_1, e_2 \in E(C_1).$  If $s > 3$, then $s'\ge 3$ and we can argue as in Case 1.  Thus we may assume that $s'=2,$ $C_1$ is a subgraph of $K_{11}$, and $\{ u,v_3 \}  \subseteq V(K_{12}).$    For convenience, we shall let $f = e_4$ and $u=v_4$.  
Let $\widehat{C}_1 = (C_1\backslash \{ e_1, e_2 \}) \cup \{ \widehat{e}_{12} \}$, which is seen to be a cycle in $\widehat{G}_{\kappa}.$  It follows by Lemma \ref{lem-newbiased} that $\widehat{C}_1$ is unbalanced in $\widehat{M}_{\kappa}.$  Moreover, every cycle in $\widehat{M}_{\kappa}$ which is an unbalanced cycle in $M_{\kappa}$ is also unbalanced in $\widehat{M}_{\kappa}.$  Thus
$\widehat{B}_1 = (B_1\backslash \{ e_1, e_2, e_3 \}) \cup \{ \widehat{e}_{12}, \widehat{e}_{23}\}$ is a base in $\widehat{M}_{\kappa}$ as is $\widehat{B}_2 = B_2-f.$ 
By the symmetric exchange property, there exists $e\in B_2$ such that $B_1^* = B_1 -e_1 + e$ and $B_2^* = B_2 - e + e_1$ are bases.  If $e \ne e_4,$ then $B_i^*,\ i = 1,2$ are the desired bases.   Thus we may assume that $e=e_4$ and consequently $B_1^* = B_1 -e_1 + e_4$ and $B_2^* = B_2 - e_4 + e_1$.  We observe that since $v_3, v_4 \in V(K_{12})$, $e_3$ and $e_4$ belong to a cycle in $G_\kappa[B_1^*]$, say $C$, which must be unbalanced.  However, this implies that $\widehat{C} = (C\backslash \{ e_3, e_4 \} ) \cup \{ \widehat{e}_{34} \}$ is an unbalanced cycle in $\widehat{M}_\kappa$ (by Lemma \ref{lem-newbiased}), implying that $\widehat{B}_1' = \widehat{B}_1 - \widehat{e}_{23} + \widehat{e}_{34} \in \B(\widehat{M}_\kappa).$ 
 Applying the symmetric exchange property to $\widehat{B}_1' $ and $\widehat{B}_2,$  there exists $\widehat{e} \in \widehat{B}_2$ such that $\widehat{B}_1''= \widehat{B}_1' - \widehat{e}_{34} + \widehat{e}$ and $\widehat{B}_2' = \widehat{B}_2 - \widehat{e} + \widehat{e}_{34}$ are bases in $\widehat{M}_\kappa.$  Now $B_1' = B_1-e_3 + \widehat{e}$ and $B_2' = B_2 -\widehat{e} + e_3$ are seen to be the desired bases.   
\end{proof}

\section{The Graph $\mathfrak{M}$ of a $v$-reduced Sequence}\label{sec-matchings}

Following Blasiak's proof for graphic matroids, we shall construct a
{\it matching graph} of a $v$-reduced sequence of bases.   Let $\B = ( B_1,
\dots ,B_{\kappa} )$ be a $v$-reduced sequence of bases of $M_\kappa.$   Let $E_{G_{\kappa}}(v) = \{ e_1, \dots ,e_m \}.$ We shall construct
a graph $\mathfrak{M}=\mathfrak{M}_{\B}$ for $\B$ as follows: Let $V(\mathfrak{M}) =
\{ \mathsf{x}_1, \dots ,\mathsf{x}_m \}.$  For $i = 1, \dots, \kappa$: 
\begin{itemize}
\item If $B_i \cap E_{G_\kappa}(v) = \{ e_j \},$ then $\mathsf{x}_j$ will be an isolated vertex.
\item If $B_i \cap E_{G_\kappa}(v) = \{ e_j, e_k \},$ then there is an edge from $x_j$ to $x_k$,  denoted $\mathsf{x}_j\mathsf{x}_k.$
\end{itemize}
\noindent We see that each vertex in $\mathfrak{M}$ has degree $0$ or $1$.
Ideally, we would like to show that if $\B_i,\ i = 1,2$ are two
$v$-reduced, compatible sequences of bases, then there exist $v$-reduced
sequences $\B_i',\ i =1,2$ such that $\B_i \sim \B_i',\ i = 1,2$
and $\mathfrak{M}_{\B_1'} = \mathfrak{M}_{\B_2'}.$   In \cite{Bla}, Blasiak was able to do this in the case where $M$ is graphic.  Unfortunately, examples show that this is not always possible for frame matroids.  The main goal of this section is to show that one can find $\B_i', \ i = 1,2$ for which $\mathfrak{M}_{\B_1'}$ and $\mathfrak{M}_{\B_2'}$ differ by at most two edges. 

For $v$-reduced sequences of bases $\B$ and $\B'$, we observe that  $|E(\fM_{\B})| = |E(\fM_{\B'})|.$  Let $E(\B, \B') = E(\fM_{\B}) \cap E(\fM_{\B'})$ and let $\varepsilon(\B, \B') = |E(\B, \B')|.$  We let $\fM_{\B} \triangle \fM_{\B'}$ denote the graph where
$V(\fM_{\B} \triangle \fM_{\B'}) = \{ \mathsf{x}_1, \dots ,\mathsf{x}_m \}$ and $E(\fM_{\B} \triangle \fM_{\B'}) = E(\fM_{\B}) \triangle E(\fM_{\B'}).$  We observe that the non-trivial components of $\fM_{\B} \triangle \fM_{\B'}$ are either
paths or even cycles whose edges alternate between. $\fM_{\B}$ and $\fM_{\B'}.$  The main theorem of this section is the following:

\begin{theorem}
Let $\B_i,\ i = 1,2$ be $v$-reduced sequences of bases.   Then there exist $v$-reduced sequences of bases $\B_i',\ i = 1,2$ where $\B_i' \sim \B_i,\ i = 1,2$ and either $\fM_{\B_1'} = \fM_{\B_2'},$ or $\fM_{\B_1'} \ne \fM_{\B_2'}$ and 
$\fM_{\B_1'} \triangle \fM_{\B_2'}$ has exactly one non-trivial component which is a $4$-cycle.\label{theorem3.2}
\end{theorem}

For a $v$-reduced sequence of bases $\B = (B_1, \dots ,B_{\kappa})$ and edges $e=\mathsf{x}_i\mathsf{x}_j$ and $f=\mathsf{x}_{i'}\mathsf{x}_{j'}$ in $\mathfrak{M}_{\B}$, we say that $e,f$ are
$\mathbf{\B}${\bf -switchable} if there is a $v$-reduced sequence $\B' = (B_1', \dots ,B_{\kappa}')$
for which $\B \sim \B'$ and $\mathfrak{M}_{\B'} = (\mathfrak{M}_{\B} \backslash \{ e,f \}
) \cup \{ e',f' \}$, where $e' = \mathsf{x}_i\mathsf{x}_{i'}$ and $f' = \mathsf{x}_j\mathsf{x}_{j'}$, or $e' = \mathsf{x}_i\mathsf{x}_{j'}$ and
$f' = \mathsf{x}_j\mathsf{x}_{i'}.$ 
  
The following lemma will be instrumental in proving the main theorem.  

\begin{lemma}
Let $\B = (B_1,\dots ,B_{\kappa})$ be a $v$-reduced sequence of bases in $M_\kappa.$  For all edges $e=\mathsf{x}_i\mathsf{x}_j \in E(\fM_{\B})$ there is at most one edge $f=\mathsf{x}_k\mathsf{x}_l \in E(\fM_{\B})\backslash \{ e \}$ for which
$e,f$ are not $\B$-switchable.\label{lem-notswitchable}
\end{lemma} 

\begin{proof}
Suppose the lemma is false.  Without loss of generality, we may assume that for $i = 1,2,$ $\mathsf{x}_1\mathsf{x}_2$ and $\mathsf{x}_{2i +1}\mathsf{x}_{2(i+i)}$ are not $\B$-switchable and $\{ e_{2i-1}, e_{2i} \} \subset B_i,\ i =1,2,3.$  Let $\widehat{\Omega}_\kappa = (\widehat{G}_\kappa, \widehat{\C}_\kappa)$ be the $v$-deleted biased graph obtained from $\Omega_\kappa$ and let $\widehat{M}_\kappa = M(\widehat{\Omega}_\kappa).$
 By Lemma \ref{lem-pullbacknonempt2}, there is a sequence of bases
$\widehat{\B} = (\widehat{B}_1, \dots ,\widehat{B}_\kappa)\in \B(\widehat{M}_\kappa)$ where $\B  \in \B_{M_\kappa}(\widehat{\B})$ and $\widehat{e}_{(2j-1)2j} \in \widehat{B}_j,\ j = 1,2,3.$  Given that  $\mathsf{x}_1\mathsf{x}_2$ and $\mathsf{x}_{2i +1}\mathsf{x}_{2(i+i)}$ are not $\B$-switchable, for $i = 1,2,$ it follows that $(\widehat{B}_1, \widehat{B}_{i+1})$ is not 
$\widehat{e}, \widehat{f}$-amenable for all $\{ \widehat{e}, \widehat{f} \} \in \left\{ \{ \widehat{e}_{1(2i+2)}, \widehat{e}_{2(2i+1)} \}, \{ \widehat{e}_{1(2i+1)}, \widehat{e}_{2(2i+2)} \} \right\}.$
  Let $I_1 = \{ 1,2,3,4 \},\ I_2 = \{ 1,2,5,6\}$ and let 
$$\begin{array}{ccc}\widehat{E}_1 &= \{ \widehat{e}_{12}, \widehat{e}_{34}, \widehat{e}_{23}, \widehat{e}_{14} \}, \ \widehat{E}_2 &= \{ \widehat{e}_{12}, \widehat{e}_{34}, \widehat{e}_{13}, \widehat{e}_{24} \}\\
\widehat{E}_3 &= \{ \widehat{e}_{12}, \widehat{e}_{56}, \widehat{e}_{25}, \widehat{e}_{16} \}, \ \widehat{E}_4 &= \{ \widehat{e}_{12}, \widehat{e}_{56}, \widehat{e}_{15}, \widehat{e}_{26} \}\end{array}
.$$   Let 
$$\begin{array}{ccc}
&\widehat{B}_1' = \widehat{B}_1 - \widehat{e}_{12} + \widehat{e}_{56}&\\
&\widehat{B}_2' = \widehat{B}_2 - \widehat{e}_{34} + \widehat{e}_{12}, &\widehat{B}_2'' = \widehat{B}_2 - \widehat{e}_{34} + \widehat{e}_{56}\\
&\widehat{B}_3' = \widehat{B}_3 - \widehat{e}_{56} + \widehat{e}_{12}, &\widehat{B}_3'' = \widehat{B}_3 - \widehat{e}_{56} + \widehat{e}_{34}\end{array}.$$
 
\noindent By Lemma \ref{lem-Mhat3}, we have that $\{ \widehat{F}_{I_1}(\widehat{B}_1), \widehat{F}_{I_1}(\widehat{B}_2) \} = \{ \widehat{E}_1,  \widehat{E}_2 \}$  
and\\  $\{ \widehat{F}_{I_2}(\widehat{B}_1), \widehat{F}_{I_2}(\widehat{B}_3) \} = \{ \widehat{E}_3,  \widehat{E}_4 \}.$   Thus $\widehat{B}_i',\ i = 1,2,3$ are bases. Letting $\widehat{B}_3'$  play the role of $\widehat{B}_1$ and $\widehat{B}_1'$ play the role of $\widehat{B}_3,$  it follows by Lemma \ref{lem-Mhat3} that  $\{ \widehat{F}_{I_1}(\widehat{B}_3'), \widehat{F}_{I_1}(\widehat{B}_2) \} = \{ \widehat{E}_1,  \widehat{E}_2 \}.$  In particular, we see that $\widehat{B}_3''\in \B(\widehat{M}_\kappa).$   Since $\{ \widehat{F}_{I_1}(\widehat{B}_1), \widehat{F}_{I_1}(\widehat{B}_2) \} = \{ \widehat{E}_1, \widehat{E}_2 \},$ it follows that
$\widehat{F}_{I_1}(\widehat{B}_1) = \widehat{F}_{I_1}(\widehat{B}_3').$
On the other hand, letting $\widehat{B}_2'$ and $\widehat{B}_1'$ play the roles of $\widehat{B}_1$ and $\widehat{B}_3$, it follows by Lemma \ref{lem-Mhat3}  that  $\{ \widehat{F}_{I_2}(\widehat{B}_2'), \widehat{F}_{I_2}(\widehat{B}_1') \} = \{ \widehat{E}_3,  \widehat{E}_4 \}.$   Thus $\widehat{B}_2''\in \B(\widehat{M}_\kappa).$  Now letting  $\widehat{B}_3''$ and $\widehat{B}_2''$ play the roles of $\widehat{B}_2$ and  $\widehat{B}_3,$ we see that (by Lemma \ref{lem-Mhat3}) $\{ \widehat{F}_{I_1}(\widehat{B}_1), \widehat{F}_{I_1}(\widehat{B}_3'') \} = \{ \widehat{E}_1,  \widehat{E}_2 \}.$  However, this is impossible since $\widehat{F}_{I_1}(\widehat{B}_1) = \widehat{F}_{I_1}(\widehat{B}_3') = \widehat{F}_{I_1}(\widehat{B}_3'').$
\end{proof}

For a graph $G$, let $C_0(G)$ denote the set of components of $G$ and let $C_1(G)$ denote the set of non-trivial components of $G.$  Let $c_0(G) = |C_0(G)|,\ c_1(G) = |C_1(G)|,$ and $c_2(G) = \sum_{C\in C_1(G)}|V(C)|^2.$
Let $\B$ be a $v$-reduced sequence of bases in $M_\kappa$ and let $\mathsf{x}_i\mathsf{x}_j, \mathsf{x}_k\mathsf{x}_l \in E(\fM_{\B}).$  Suppose that $\mathsf{x}_i\mathsf{x}_j, \mathsf{x}_k\mathsf{x}_l$ are $\B$-switchable. If $\B'$ is a $v$-reduced sequence of bases where $\B \sim \B'$ and  $E(\fM_{\B'}) = (E(\fM_{\B}) \backslash \{ \mathsf{x}_i\mathsf{x}_j , \mathsf{x}_k\mathsf{x}_l \} ) \cup  \{ \mathsf{x}_i\mathsf{x}_k, \mathsf{x}_j\mathsf{x}_l\} ,$ then we write $\B \swapxyzw{x_i}{x_j}{x_k}{x_l} \B'.$  If $\B'$ is a $v$-reduced sequence of bases where $E(\fM_{\B'}) = (E(\fM_{\B}) \backslash \{ \mathsf{x}_i\mathsf{x}_j , \mathsf{x}_k\mathsf{x}_l \} ) \cup \{ \mathsf{x}_i\mathsf{x}_l, \mathsf{x}_j\mathsf{x}_k \}$ then we write $\B \crossswapxyzw{x_i}{x_j}{x_k}{x_l} \B'.$  If  $\B \swapxyzw{x_i}{x_j}{x_k}{x_l} \B'$ or $\B \crossswapxyzw{x_i}{x_j}{x_k}{x_l} \B'$, then we write $\B \exchangexyzw{x_i}{x_j}{x_k}{x_l} \B'.$  Suppose $\mathsf{x}_i\mathsf{x}_j \in E(\fM_{\B})$ and let $\mathsf{x}_k$ be an isolated vertex of 
$\fM_{\B}$.  If there exists a $v$-reduced sequence $\B'$ such that $\B \sim \B'$ and $E(\fM_{\B'}) = (E(\fM_{\B}) \backslash \{ \mathsf{x}_i\mathsf{x}_j\} ) \cup \{ \mathsf{x}_j\mathsf{x}_k\},$ then we write
$\B \singswapxyz{x_i}{x_j}{x_k} \B'$.  We write $\B \singswaptri{x_i}{x_j}{x_k} \B'$ if either $\B \singswapxyz{x_i}{x_j}{x_k} \B'$ or $\B \singswapxyz{x_j}{x_i}{x_k} \B'.$

\begin{lemma}
Let $\B = (B_1, \dots ,B_{\kappa})$ be a $v$-reduced sequence of bases of $M$ and let $\sx_i\sx_j \in E(\mathfrak{M}_{\B})$.
 If $\sx_l$ is an isolated vertex of $\fM_{\B},$
then there exists a $v$-reduced sequence $\B' = (B_1', \dots ,B_{\kappa}')$ such that $\B \sim \B'$ and
$\B \singswaptri {x_i}{x_j}{x_k} \B'$.\label{lemma3}
\end{lemma}

\begin{proof}
Let $\widehat{\Omega}_\kappa = (\widehat{G}_\kappa, \widehat{\C}_\kappa)$ be the $v$-deleted linear biased graph obtained from $\Omega_\kappa$ and let $\widehat{M}_\kappa = M(\widehat{\Omega}_\kappa).$  We may assume that
$e_i, e_j \in B_1$ and $e_k \in B_2.$  Then $\widehat{B}_1 = (B_1\backslash \{ e_i, e_j \}) \cup \{ \widehat{e}_{ij} \} \in \B(\widehat{M}_\kappa).$  By Lemma \ref{lem-Mhat1}, either $\widehat{B}_1' = \widehat{B}_1 - \widehat{e}_{ij} + \widehat{e}_{ik}$ or
$\widehat{B}_1'' = \widehat{B}_1 - \widehat{e}_{ij} + \widehat{e}_{jk}$ are bases in $\widehat{M}_\kappa.$  If $\widehat{B}_1' \in \B(\widehat{M}_\kappa)$, then $B_1' = B_1-e_j + e_k$ and  $B_2' = B_2 - e_k + e_j$, are bases in $M_\kappa.$  Otherwise, if $\widehat{B}_1'' \in \B(\widehat{M}_\kappa),$ then $B_1' = B_1-e_i + e_k$ and  $B_2' = B_2 - e_k + e_i$ are bases in $M_\kappa$.  Let $\B' = (B_1', \dots ,B_\kappa'),$ where $B_i' = B_i,\ i = 3, \dots ,\kappa .$  Then $\B' \sim \B$ and $\B \singswaptri {x_i}{x_j}{x_k} \B'$.
\end{proof}

\begin{lemma}
Let $\B_i, \ i =1,2$ be $v$-reduced sequence of bases.  Let $\B_i',\ i = 1,2$ be $v$-reduced sequence of bases where

\begin{itemize}
\item[(i)]  $\B_i \sim \B_i', \ i = 1,2,$
\item[(ii)]  $\varepsilon(\B_1',\B_2')$ is maximum, and 
\item[(iii)]  $c_2(\fM_{\B_1'} \triangle \fM_{\B_2'})$ is maximum subject to (i) and (ii).
\end{itemize}
Then either $\fM_{\B_1'} = \fM_{\B_2'},$ or $\fM_{\B_1'} \triangle \fM_{\B_2'}$ has exactly one non-trivial component which is either a $4$-cycle or a $4$-path.\label{lemma3.1}
\end{lemma}

\begin{proof}
Assume $\fM_{\B_1'} \ne \fM_{\B_2'}$ and let $\fM' = \fM_{\B_1'} \triangle \fM_{\B_2'}.$  As noted before, the non-trivial components of $\fM'$ are either paths or even cycles.  By Lemma \ref{lem-notswitchable}, each edge of $\fM_{\B_i'},\ i =1,2$ is $\B_i'$-switchable with all other edges of $\fM_{\B_i'}$, with the exception of at most one edge.  

\begin{claim}
Every cycle of $\fM'$ is a $4$-cycle.\label{claim1lemma3.1}
\end{claim}  

\begin{proof}
Let $C\in C_1(\fM')$ where $C$ is a cycle.  We may assume\\ $C = \mathsf{x}_1f_1\mathsf{x}_2 \cdots \mathsf{x}_{2k}f_{2k}\mathsf{x}_1.$  We may also assume that $f_i \in E(\fM_{\B_1'})$ (resp. $f_i \in E(\fM_{\B_2'})$) when $i$ is
odd (resp. even).    Suppose $k\ge 3.$  We have that either $\mathsf{x}_1\mathsf{x}_2,  \mathsf{x}_3\mathsf{x}_4$ are $\B_1'$-switchable, or $\mathsf{x}_1\mathsf{x}_2,  \mathsf{x}_{2k-1}\mathsf{x}_{2k}$ are $\B_1'$-switchable.
Suppose the former occurs and let $\B_1''$ be a $v$-reduced sequence where $\B_1' \exchangexyzw {x_1}{x_2}{x_3}{x_4} \B_1''.$  If $\B_1' \crossswapxyzw{x_1}{x_2}{x_3}{x_4} \B_1'',$  then $\mathsf{x}_2\mathsf{x}_3 \in 
E(\B_1'', \B_2')$ and hence $\varepsilon (\B_1'', \B_2') > \varepsilon (\B_1', \B_2'),$ contradicting the choice of $\B_1'$ and $\B_2'$ (by (ii)).  It follows that $\B_1' \swapxyzw {x_1}{x_2}{x_3}{x_4} \B_1''.$  
We also have that either $\mathsf{x}_2\mathsf{x}_3,\ \mathsf{x}_4\mathsf{x}_5$ or $\mathsf{x}_2\mathsf{x_3},\ \mathsf{x}_{2k}\mathsf{x}_1$ are $\B_2'$-switchable.    Suppose $\mathsf{x}_2\mathsf{x}_3,\ \mathsf{x}_4\mathsf{x}_5$
are $\B_2'$-switchable, and let $\B_2''$ be a $v$-reduced sequence such that $\B_2' \exchangexyzw{x_2}{x_3}{x_4}{x_5} \B_2''.$  If $\B_2' \crossswapxyzw{x_2}{x_3}{x_4}{x_5} \B_2'',$  then $\mathsf{x}_3\mathsf{x}_4 \in E(\B_1', \B_2'')$ and hence
$\varepsilon(\B_1', \B_2'') > \varepsilon(\B_1', \B_2'),$ a contradiction.  If $\B_2' \swapxyzw{x_2}{x_3}{x_4}{x_5} \B_2'',$ then $\mathsf{x}_2\mathsf{x}_4 \in E(\B_1'', \B_2'')$ and hence $\varepsilon(\B_1'', \B_2'') > \varepsilon(\B_1', \B_2'),$ a contradiction.  If $\mathsf{x}_2\mathsf{x}_3, \ \mathsf{x}_{2k}\mathsf{x}_1$ are $\B_2'$-switchable, then we obtain a contradiction using similar arguments
\end{proof} 

\begin{claim}
$\fM'$ does not contain components $P$ and $Q$ which are paths where a terminal vertex of $P$ is an isolated vertex in $\fM_{\B_2'}$ and a terminal vertex of $Q$ is an isolated vertex in $\fM_{\B_1'}.$\label{claim2lemma3.1}  
\end{claim}

\begin{proof}
Suppose to the contrary that such components $P$ and $Q$ exist.  We may assume that $P = \mathsf{x}_1f_1\mathsf{x}_2 f_2 \cdots \mathsf{x}_{k-1}f_{k-1}\mathsf{x}_k$ where $\mathsf{x}_1$ is an isolated vertex in $\fM_{\B_2'}.$
It follows that $\mathsf{x}_1\mathsf{x}_2 \in E(\fM_{\B_1'}).$  Let $\mathsf{x}_l$ be the terminal vertex of $Q$ which is an isolated vertex in $\fM_{\B_1'}.$  We may assume that the length of $Q$ is at least the length of $P$ (otherwise, we can simply interchange the roles of $\B_1'$ and $\B_2'$).  By Lemma \ref{lemma3}, there exists a $v$-reduced sequence $\B_1''$ such that $\B_1' \singswaptri{x_1}{x_2}{x_l} \B_1''.$  We have
$\varepsilon(\B_1'',\B_2') = \varepsilon(\B_1', \B_2').$ However, given that the length of $Q$ is at least the length of $P,$ it follows that $c_2(\fM_{\B_1''}\triangle \fM_{\B_2'}) > c_2( \fM_{\B_1''}\triangle \fM_{\B_2'}) ,$ contradicting our choice of $\B_1'$ and $\B_2'$ (for which (iii) applies).  Thus no such paths
$P$ and $Q$ can exist.
\end{proof}

As a result of the above claim, we obtain the following:

\begin{claim}
$\fM'$ has no component which is an odd path.\label{claim3lemma3.1}
\end{claim}

\begin{proof}
Suppose to the contrary that $P$ is a component of $\fM'$ which is an odd path.   Then it is seen that both of the terminal vertices in $P$ are either isolated vertices in $\fM_{\B_1'}$ or they are isolated vertices in $\fM_{\B_2'}.$
Without loss of generality, we may assume that both terminal vertices are isolated vertices in $\fM_{\B_1'}.$  Since $\fM_{\B_1'}$ and $\fM_{\B_2'}$ have the same number of edges, they have the same number of isolated vertices.
Now it follows by a simple parity argument that there exists a component of $\fM'$ which is an odd path $Q$ whose terminal vertices are isolated vertices in $\fM_{\B_2'}.$  However, $P$ and $Q$ contradict Claim \ref{claim2lemma3.1}.  Thus no such path $P$ exists.
\end{proof}

\begin{claim}
$\fM'$ does not contain two or more components which are paths.\label{claim4lemma3.1}
\end{claim}

\begin{proof}
Suppose $P$ and $Q$ are components of $\fM'$ which are paths.  Then by Claim \ref{claim3lemma3.1}, they are even paths.  Now it is seen that one terminal vertex of $P$ belongs to $\fM_{\B_1'},$ and the other belongs to
$\fM_{\B_2'}.$  The same applies to $Q.$  However, $P$ and $Q$ contradict Claim \ref{claim2lemma3.1}.
\end{proof}

\begin{claim}
$\fM'$ has at most one non-trivial component.\label{claim5lemma3.1}
\end{claim}

\begin{proof}
Suppose to the contrary that $\fM'$ contains two non-trivial components, say $C$ and $P$.  By Claims \ref{claim1lemma3.1} and  \ref{claim4lemma3.1}, we may assume that $C$ is a $4$-cycle.  Furthermore, we may assume $C = \mathsf{x}_1f_1\mathsf{x}_2f_2\mathsf{x}_3f_3\mathsf{x}_4f_4\mathsf{x}_1$ where $\mathsf{x}_1\mathsf{x}_2, \mathsf{x}_3\mathsf{x}_4 \in E(\fM_{\B_1'})$.  Suppose $P$ is a path.  By Claim \ref{claim3lemma3.1}, $P$ is an even path.  Then one of the terminal vertices of $P$ is an isolated vertex in $\fM_{\B_1'},$ say $\mathsf{x}_l.$  By Lemma \ref{lemma3}, there exists a $v$-reduced sequence $\B_1''$ such that 
$\B_1' \singswaptri{x_1}{x_2}{x_l} \B_1''$.  However, one sees that $\varepsilon(\B_1'', \B_2') = \varepsilon(\B_1', \B_2')$ and $c_2(\fM_{\B_1''} \triangle \fM_{\B_2'}) > c_2(\fM'),$ contradicting our choice of $\B_1'$ and $\B_2'.$
Suppose $P$ is a cycle.  Then by Claim \ref{claim1lemma3.1}, $P$ is a $4$-cycle.  We may assume $\mathsf{x}_5\mathsf{x}_6 \in E(P)$ and $\mathsf{x}_5\mathsf{x}_6 \in E(\fM_{\B_1'}).$  We observe that either $\mathsf{x}_1\mathsf{x}_2, \ \mathsf{x}_5\mathsf{x}_6$ is $\B_1'$-switchable, or $\mathsf{x}_3\mathsf{x}_4, \ \mathsf{x}_5\mathsf{x}_6$ is $\B_1'$-switchable.  Without loss of generality, we may assume $\mathsf{x}_1\mathsf{x}_2, \ \mathsf{x}_5\mathsf{x}_6$ is $\B_1'$-switchable.  Then there exists a $v$-reduced sequence $\B_1''$ such that $\B_1' \exchangexyzw{x_1}{x_2}{x_5}{x_6} \B_1''$.  However, one sees that $\varepsilon(\B_1'', \B_2') = \varepsilon(\B_1', \B_2')$ and $c_2(\fM_{\B_1''} \triangle \fM_{\B_2'}) > c_2(\fM'),$ contradicting our choice of $\B_1'$ and $\B_2'.$  Thus $\fM'$ has at most one non-trivial component.
\end{proof}

\begin{claim}
If $P$ is a non-trivial component of $\fM'$ which is a path, then it is a $4$-path.\label{claim6lemma3.1}
\end{claim}

\begin{proof}
Suppose $P$ is a component of $\fM'$ which a path.  Then by Claim \ref{claim3lemma3.1}, $P$ is an even path.  We may assume $P = \mathsf{x}_1f_1\mathsf{x}_2 f_2 \cdots \mathsf{x}_{2k}f_{2k}\mathsf{x}_{2k+1}$
where $f_i \in E(\fM_{\B_1'})$ (resp. $f_i \in E(\fM_{\B_2'})$) when $i$ is odd (resp. even).  Suppose $k\ge3.$  We observe that either $\mathsf{x}_1\mathsf{x}_2,\  \mathsf{x}_3\mathsf{x}_4$ are $\B_1'$-switchable or $\mathsf{x}_3\mathsf{x}_4,\  \mathsf{x}_5\mathsf{x}_6$ are $\B_1'$-switchable.  

Suppose $\mathsf{x}_1\mathsf{x}_2,\  \mathsf{x}_3\mathsf{x}_4$ are $\B_1'$-switchable.  Then there exists a $v$-reduced sequence $\B_1''$ such that
$\B_1' \ \exchangexyzw{x_1}{x_2}{x_3}{x_4} \B_1''.$  If $\B_1' \ \crossswapxyzw{x_1}{x_2}{x_3}{x_4} \B_1'',$  then $\mathsf{x}_2\mathsf{x}_3 \in E(\B_1'', \B_2')$ and hence $\varepsilon(\B_1'', \B_2') > \varepsilon(\B_1', \B_2'),$ contradicting the choice of $\B_1'$ and $\B_2'$.  Thus $\B_1' \ \swapxyzw{x_1}{x_2}{x_3}{x_4} \B_1''.$   Observing that $\mathsf{x}_1$ is an isolated vertex of $\fM_{\B_2'},$ if follows by Lemma \ref{lemma3} that there exists a $v$-reduced sequence $\B_2''$ for which $\B_2'\  \singswaptri{x_2}{x_3}{x_1} \B_2''.$  If $\B_2' \singswapxyz{x_3}{x_2}{x_1} \B_2'',$ then $\mathsf{x}_1\mathsf{x}_2 \in E(\B_1', \B_2'')$ and hence $\varepsilon(\B_1', \B_2'') > \varepsilon(\B_1', \B_2'),$ a contradiction.  Thus $\B_2' \singswapxyz{x_2}{x_3}{x_1} \B_2''.$  However, we now see that $\mathsf{x}_1\mathsf{x}_3 \in E(\B_1'', \B_2'')$ and hence $\varepsilon(\B_1'', \B_2'') > \varepsilon(\B_1', \B_2'),$ a contradiction.

Suppose  $\mathsf{x}_3\mathsf{x}_4,\  \mathsf{x}_5\mathsf{x}_6$ are $\B_1'$-switchable.  Then there exists a $v$-reduced sequence $\B_1''$ for which $\B_1' \ \exchangexyzw{x_3}{x_4}{x_5}{x_6} \B_1''.$  If $\B_1' \ \crossswapxyzw{x_3}{x_4}{x_5}{x_6} \B_1'',$  then $\mathsf{x}_4\mathsf{x}_5 \in E(\B_1'', \B_2')$ and hence $\varepsilon(\B_1'', \B_2') > \varepsilon(\B_1', \B_2'),$ contradicting the choice of $\B_1'$ and $\B_2'$.  Thus $\B_1' \ \swapxyzw{x_3}{x_4}{x_5}{x_6} \B_1''.$  We observe that either $\mathsf{x}_2\mathsf{x}_3,\  \mathsf{x}_4\mathsf{x}_5$ are $\B_2'$-switchable or $\mathsf{x}_4\mathsf{x}_5,\  \mathsf{x}_6\mathsf{x}_7$ are $\B_2'$-switchable.  Suppose the former is true.  There there exists a $v$-reduced sequence $\B_2''$ such that $\B_2' \exchangexyzw{x_2}{x_3}{x_4}{x_5} \B_2''.$  If $\B_2' \crossswapxyzw{x_2}{x_3}{x_4}{x_5} \B_2'',$  then $\mathsf{x}_3\mathsf{x}_4 \in 
E(\B_1', \B_2''),$ and hence $\varepsilon(\B_1', \B_2'') > \varepsilon(\B_1', \B_2'),$ a contradiction.  Thus  $\B_2' \swapxyzw{x_2}{x_3}{x_4}{x_5} \B_2''.$  However, one sees that $\mathsf{x}_3\mathsf{x}_5 \in E(\B_1'', \B_2'')$ and hence $\varepsilon(\B_1'', \B_2'') > \varepsilon(\B_1', \B_2'),$ a contradiction.

If $\mathsf{x}_4\mathsf{x}_5,\  \mathsf{x}_6\mathsf{x}_7$ are $\B_2'$-switchable, then we obtain a contradiction in a similar fashion.  Thus $k\le 2.$  Suppose $k=1.$  Then $P = \mathsf{x}_1f_1\mathsf{x}_2 f_2 \mathsf{x}_{3}.$
There are $v$-reduced sequences $\B_i'',\ i = 1,2$ where $\B_1' \singswaptri {x_1}{x_2}{x_3} \B_1''$ and $\B_2' \singswaptri {x_2}{x_3}{x_1} \B_2''.$  As before, we may assume
$\B_1' \singswapxyz{x_2}{x_1}{x_3} \B_1''$ and $\B_2' \singswapxyz{x_2}{x_3}{x_1} \B_2''.$
However, we see that $\mathsf{x}_1\mathsf{x}_3 \in E(\B_1'', \B_2'')$ and hence $\varepsilon(\B_1'', \B_2'') > \varepsilon(\B_1', \B_2'),$ a contradiction.  We conclude that $k=2$ and $P$ is a $4$-path.
\end{proof}

The proof of the lemma now follows by Claims \ref{claim1lemma3.1}, \ref{claim5lemma3.1}, and \ref{claim6lemma3.1}.
\end{proof}

\subsection{Proof of Theorem \ref{theorem3.2}}

By Lemma \ref{lemma3.1}, there exist $v$-reduced sequences $\B_i', \ i =1,2$ such that $\B_i' \sim \B_i,\ i = 1,2$ and either $\fM_{\B_1'} = \fM_{\B_2'},$ or $\fM_{\B_1'} \triangle \fM_{\B_2'}$ has exactly one non-trivial component, say $P$, which is either a $4$-path or a $4$-cycle.   Suppose $P$ is a $4$-path.  We may assume $P = \mathsf{x}_1f_1\mathsf{x}_2f_2\mathsf{x}_3f_3\mathsf{x}_4f_4\mathsf{x}_5$ where $f_i \in E(\fM_{\B_1'})$ (resp. $f_i \in E(\fM_{\B_2'})$) when $i$ is odd (resp. even).  In particular, $\mathsf{x}_1$ is an isolated vertex of $\fM_{\B_2'}$ and $\mathsf{x}_5$ is an isolated vertex of $\fM_{\B_1'}.$  By Lemma \ref{lemma3}, there exist $v$-reduced sequences $\B_i'',\ i = 1,2$ such that $\B_1' \singswaptri{x_1}{x_2}{x_5}\ \B_1''$ and $\B_2' \singswaptri{x_4}{x_5}{x_1}\B_2''.$  If $\B_1' \singswapxyz{x_1}{x_2}{x_5}\ \B_1''$, then we see that $\fM_{\B_1''} \triangle \fM_{\B_2'}$ has exactly one non-trivial component which is a $4$-cycle.  Thus we may assume $\B_1' \singswapxyz{x_2}{x_1}{x_5}\ \B_1''$. Similarly, we may assume $\B_2' \singswapxyz{x_4}{x_5}{x_1}\ \B_2''$.    Now $\fM_{\B_1''} \triangle \fM_{\B_2''}$ has exactly one non-trivial path $Q = \mathsf{x}_2f_2\mathsf{x}_3f_3\mathsf{x}_4$ where $\mathsf{x}_2\mathsf{x}_3 \in E(\fM_{\B_2''})$ and $\mathsf{x}_3\mathsf{x}_4 \in E(\fM_{\B_1''}).$  Thus $\mathsf{x}_4$ is an isolated vertex in
$\fM_{\B_2''}$ and thus (by Lemma \ref{lemma3}), there exists a $v$-reduced sequence $\B_2'''$ such that $\B_2'' \singswaptri{x_2}{x_3}{x_4}\B_2'''.$  Similarly, there exists a $v$-reduced sequence $\B_1'''$ such that 
$\B_1'' \singswaptri{x_3}{x_4}{x_1} \B_1'''.$  If $\B_2'' \singswapxyz{x_2}{x_3}{x_4} \B_2''',$ then $\fM_{\B_1''} = \fM_{\B_2'''}.$  Thus we may assume $\B_2'' \singswapxyz{x_3}{x_2}{x_4} \B_2'''.$  Likewise, we may assume
$\B_1'' \singswapxyz{x_3}{x_4}{x_2}\B_1'''.$  However, we now see that $\fM_{\B_1'''} = \fM_{\B_2'''}.$
This completes the proof.

\section{Linked Sequences}\label{sec-linked}

A base $B\in \B(M_\kappa)$ is said to be $\mathbf{v}${\bf-anchored} if $B$ contains exactly one edge of $E_{G_\kappa}(v)$ (which could be a loop); otherwise, $B$ is $\mathbf{v}${\bf-unanchored}.  For all bases $B \in \B(M_\kappa),$ let $\tilde{B} = B\backslash E_{G_\kappa}(v).$
For  compatible sequences of bases $\B = (B_1, \dots ,B_\kappa)$ 
and $\B' = (B_1', \dots ,B_\kappa')$, we say that $\B$ and $\B'$ are {\bf linked} if exist $j,j'\in \{ 1, \dots, \kappa \}$ such that either $B_j = B_{j'}'$ or $B_j$ and $B_{j'}$ are $v$-anchored and $\tilde{B}_j = \tilde{B}_{j'}.$  We write $\B \bumpeq \B'$. 
We shall need the following theorem.

\begin{theorem}
Assume that $|E_{G_\kappa}(v)| = m = 2\kappa.$  Furthermore, assume that Theorem \ref{the-kappa31} holds for all $k<\kappa$ and Theorem \ref{the-main2} holds for all $k < \kappa$, if $\kappa \ge 3.$  If $\B \bumpeq \B',$ then $\B \sim \B'.$\label{the-linked}
\end{theorem}

\begin{proof}
Suppose $\B \bumpeq \B'.$  Suppose first that for some $j,j' \in \{ 1, \dots, \kappa \}$, we have $B_j = B_{j'}'.$  Let $\D$ (resp. $\D'$) be the sequence of bases resulting from deleting $B_j$ (resp. $B_{j'}'$).  Then $\D$ and $\D'$ are seen to be compatible sequences of $\kappa-1$ bases.  Note that if $\kappa =2$, then $\D = \D'$ and hence $\B = \B'.$  Assuming  $\kappa \ge 3,$ by assumption, Theorem \ref{the-main2} holds for $k = \kappa -1$ and consequently $\D \sim \D'.$
It now follows that $\B \sim \B'.$  Suppose instead that for some $j,j' \in \{ 1, \dots, \kappa \}$, $B_j$ and $B_{j'}$ are $v$-anchored and $\tilde{B}_j = \tilde{B}_{j'}.$
Let $\{ h \} = B_j \cap E_{G_\kappa}(v)$ and $\{ h' \} = B_{j'} \cap E_{G_\kappa}(v).$  As before, let $\D$ (resp. $\D'$) be the sequence of base resulting from deleting $B_j$ (resp. $B_{j'}'$).  Now we construct $v$-extended base sequences $\D^+ = (\D, h)$ and $\D'^+ = (\D', h'),$ noting that $|d_{G_\kappa}(v)| = m = 2\kappa$ and $\D^+$ and $\D'^+$ each contain $\kappa -1$ bases.  Let $G_{\kappa}' = G_{\kappa}\backslash \tilde{B}_j.$  Then $d_{G_{\kappa}'}(v) = d_{G_\kappa}(v) = 2\kappa.$  Since Theorem \ref{the-kappa31} holds for $k= \kappa-1,$ it follows that $\D^+ \sim \D'^+.$  However, given that $B_j$ and $B_{j'}$ are $v$-anchored, we see that for all $e\in E_{G_\kappa}(v)$, $\tilde{B}_j \cup \{ e \} \in \B(M_\kappa)$ and $\tilde{B}_{j'}' \cup \{ e \} \in \B(M_{\kappa}).$  Thus it is seen that $\B \sim \B'.$
\end{proof}

\subsection{Non-Incidental Bases in $\widehat{M}_{\kappa}$}\label{subsec-nonincidental}

We say that two edges $\widehat{e}$ and $\widehat{f}$ in $\widehat{G}_\kappa$ are {\bf incidental} if $\fE(\widehat{e}) \cap \fE(\widehat{f}) \ne  \emptyset.$  Otherwise, they are said to be {\bf non-incidental}.  Extending this to sets, we say that a set $\widehat{A} \subseteq \widehat{E}$ is incidental if two of its edges are incidental;  otherwise, it is non-incidental. We say that two subsets of edges $\widehat{S} \subseteq \widehat{E}$ and $\widehat{T} \subseteq \widehat{E}$ are incidental if for some $\widehat{e} \in \widehat{S}$ and $\widehat{f} \in \widehat{T},$ $\widehat{e}$ and $\widehat{f}$ are incidental; otherwise, $\widehat{S}$ and $\widehat{T}$ are non-incidental.  Finally, we say that a sequence of disjoint bases
$\widehat{\B} = (\widehat{B}_1, \dots ,\widehat{B}_\kappa)$ in $\widehat{M}_\kappa$ is {\bf incidental} if for some distinct $i,j$ the sets $\widehat{B}_i \cap \widehat{E}$ and $\widehat{B}_j \cap \widehat{E}$ are incidental; otherwise, it is non-incidental.   It follows from Lemma \ref{lem-pullbacknonempt} and the definition of the pull-back that for any non-incidental sequence of bases $\widehat{\B}$, $\fB_{M_{\kappa}}(\widehat{\B}) \ne \emptyset.$

\begin{lemma}
Assume that $|E_{G_\kappa}(v)| = m = 2\kappa.$  Furthermore, assume that Theorem \ref{the-kappa31} holds for all $k<\kappa$ and Theorem \ref{the-main2} holds for all $k < \kappa$, if $\kappa \ge 3.$  Let $\widehat{\B} = (\widehat{B}_1, \dots ,\widehat{B}_\kappa)$ be a disjoint non-incidental sequence of bases in $\widehat{M}_{\kappa}.$
Then for all $\B, \B' \in \fB_{M_\kappa}(\widehat{\B}),$  $\B \sim \B'.$\label{lem-linkedpullback}
\end{lemma}

\begin{proof}
Let $\B, \B' \in \fB_{M_\kappa}(\widehat{\B})$ where $\B = (B_1, \dots ,B_\kappa)$ and $\B' = (B_1', \dots ,B_\kappa').$    We observe that $\tilde{B}_j = \tilde{B}_j',\ j = 1, \dots ,\kappa$ since $\B, \B' \in \B_{M_\kappa}(\widehat{\B}).$  If $\kappa =2,$ then it is seen that $|B_1'\backslash B_1| \le 1$ and hence
$\B \sim_1 \B'$.  We may assume $\kappa \ge 3$ and Theorem \ref{the-main2} holds for all $k < \kappa$.  If $B_i$ is $v$-anchored for some $i,$ then $B_i'$ is also $v$-anchored and given that $\tilde{B}_i = \tilde{B}_i'$, it follows that $\B \bumpeq \B'.$  Thus $\B \sim \B'$ (by Theorem \ref{the-linked}).   Assuming that $\B$ (and hence also $\B'$) has no $v$-anchored bases, it follows that $|\widehat{B}_i \cap \widehat{E}| = 1,$ for all $i.$   It follows by the construction of the pullback that for all $i$, $B_i = B_i'$ and hence $\B = \B'.$  We conclude that  $\B \sim \B'.$
\end{proof}

\subsection{Induced Sequences of Bases}\label{subsec-induced}

Suppose $\widehat{\B} = (\widehat{B}_1, \dots ,\widehat{B}_{\kappa})$ and $\widehat{\B}' = (\widehat{B}_1', \dots ,\widehat{B}_{\kappa}')$ are disjoint, non-incidental sequences of bases in $\widehat{M}$ where $\widehat{\B}$ and $\widehat{\B}'$ need not be compatible.
Let $\B = (B_1, \dots ,B_\kappa) \in \fB_{M_\kappa}(\widehat{\B}).$   It follows by the definition of the pull-back, there is a sequence $\B' = (B_1', \dots ,B_\kappa') \in \fB_{M_\kappa}(\widehat{\B})$ where for $i = 1, \dots ,\kappa,$ $B_i' = B_i$ if $\widehat{B}_i' = \widehat{B}_i$ and $\widehat{B}_i \cap \widehat{E} \ne \emptyset.$  We refer to $\B'$ as a pull-back of $\widehat{\B}'$ {\bf induced} by $\B$ and write $\B \vdash\B'.$  We have the following useful observations.

\sms
\noindent{\bf Observation 1}:  Suppose $\widehat{\B}$ and $\widehat{\B}'$ are non-incidental sequences and $|E_{G_\kappa}(v)| = m = 2\kappa.$ Let $\B \in \fB_{M_\kappa}(\widehat{\B})$ and $\B' \in \fB_{M_\kappa}(\widehat{\B}')$ where $\B \vdash \B'.$  Suppose that Theorem \ref{the-kappa31} holds for all $k < \kappa$ and Theorem \ref{the-main2} holds for all $k < \kappa$, if $\kappa \ge 3.$
If $\widehat{\B} \sim_1 \widehat{\B}',$  then $\B \sim \B'.$  
\begin{proof}
Suppose $\kappa \ge 3.$   Then it follows that $\widehat{B}_i = \widehat{B}_i',$ for some $i.$  If $\widehat{B}_i \cap \widehat{E} \ne \emptyset$, then $B_i = B_i'$ since $\B \vdash \B'.$ Otherwise, if $\widehat{B}_i \cap \widehat{E} = \emptyset,$ then 
$B_i$ is $v$-anchored and $\tilde{B}_i = \tilde{B}_i'.$  In either case, we $\B \bumpeq \B'$ and hence $\B \sim \B'$ by Theorem \ref{the-linked}.

Suppose $\kappa =2.$ Then $\widehat{\B} = (\widehat{B}_1, \widehat{B}_2),\ \widehat{\B}' = (\widehat{B}_1', \widehat{B}_2')$ and $m =4.$  We may assume that $(\widehat{B}_1 \cup \widehat{B}_2) \cap \widehat{E} = \{ \widehat{e}_{12}, \widehat{e}_{34} \}.$  Let $I = \{ 1,2,3,4 \}$. Since $\widehat{\B} \sim_1 \widehat{\B}'$, there exists $\widehat{e} \in \widehat{B}_1$ and $\widehat{f} \in \widehat{B}_2$ such that 
$\widehat{B}_1' = \widehat{B}_1 - \widehat{e} + \widehat{f}$ and $\widehat{B}_2' = \widehat{B}_1 - \widehat{f} + \widehat{e}.$ If $\{ \widehat{e},\widehat{f} \} \cap \{ \widehat{e}_{12}, \widehat{e}_{34} \} = \emptyset,$ then it is easy to show that $\B \sim_1 \B'.$  Without loss of generality, we may assume that $\widehat{f} = \widehat{e}_{34}.$  

Suppose $\widehat{e} = \widehat{e}_{12}.$  Then $\fB_{M_\kappa}(\widehat{\B}) = \{ \B \}$ where $\B = (B_1, B_2),$ $B_1 = (\widehat{B}_1\backslash \widehat{e}_{12} ) \cup \{ e_1, e_2 \}$ and $B_2 = (\widehat{B}_2\backslash \widehat{e}_{34} ) \cup \{ e_3, e_4\}.$  Similarly, $\fB_{M_\kappa}(\widehat{\B}') = \{ \B' \}$ where $\B' = (B_1', B_2'),$ $B_1' = (\widehat{B}_1'\backslash \widehat{e}_{34} ) \cup \{ e_3, e_4 \}$ and $B_2' = (\widehat{B}_2'\backslash \widehat{e}_{12} ) \cup \{ e_1, e_2\}.$
Suppose that $(\widehat{B}_1, \widehat{B}_2)$ is $\widehat{e}', \widehat{f}'$-amenable for some $\{ \widehat{e}', \widehat{f}' \} \in \{ \{ \widehat{e}_{14}, \widehat{e}_{23} \}, \{ \widehat{e}_{13}, \widehat{e}_{24} \} \}.$  Without loss of generality, we may assume $\widehat{\B}$ is $\widehat{e}_{14}, \widehat{e}_{23}$-amenable and $\widehat{e}_{14} \in \widehat{F}_I(\widehat{B}_1)$ and $\widehat{e}_{23} \in \widehat{F}_I(\widehat{B}_2).$
Then $\B'' = (B_1'', B_2'')$ where $B_1'' = B_1 -e_2 + e_4$ and $B_2'' = B_2- e_4 + e_2$ is seen to be a base sequence in $M_\kappa.$  We also see that\\ $\B \sim_1 \B'' \sim_1 \B'$ and hence $\B \sim \B'.$  Thus we may assume that $(\widehat{B}_1, \widehat{B}_2)$ is not $\widehat{e}', \widehat{f}'$-amenable for all $\{ \widehat{e}', \widehat{f}' \} \in \{ \{ \widehat{e}_{14}, \widehat{e}_{23} \}, \{ \widehat{e}_{13}, \widehat{e}_{24} \} \}.$  By Theorem \ref{the-KotZiv},  there exists $B = \{ f_1, f_2 \} \subset B_2$ such that $A = \{ e_1, e_2 \}$ is serially-exchangeable with $B.$  We may assume that $e_1 \prec e_2$ and $f_1 \prec f_2$ are serial orderings for $A$ and $B$, respectively.  Then $\B'' = (B_1'',B_2'')$ where $B_1'' = B_1 - e_1 + f_1$ and $B_2'' = B_2 - f_1 +e_1$ is a base sequence in $M_\kappa.$ If $f_1 \in \{ e_3, e_4 \}$, Then $(\widehat{B}_1, \widehat{B}_2)$ is seen to be $\widehat{e}', \widehat{f}'$-amenable for some $\{ \widehat{e}', \widehat{f}' \} \in \{ \{ \widehat{e}_{14}, \widehat{e}_{23} \}, \{ \widehat{e}_{13}, \widehat{e}_{24} \} \},$ contradicting our assumptions. Thus $f_1 \not\in \{ e_3, e_4 \}.$  Thus $B_1''$ is $v$-anchored, and we must have that $f_2 \in \{ e_3, e_4 \}.$  Without loss of generality, we may assume $f_1 = e_3.$
Let $\B''' = (B_1''', B_2'''),$ where $B_1''' = B_1 - e_2 - e_2 + f_1 + f_2$ and $B_2''' = B_2 - f_1 -f_2 + e_1 + e_2.$  Now we see that $B_1' = B_1''' - f_1 + e_4$ and $B_2' = B_2''' -e_4 + f_1.$
Thus it follows that $\B \sim_1 \B'' \sim_1 \B''' \sim_1 \B'.$  Consequently, $\B \sim \B'.$

From the above, we may assume that $\widehat{e} \ne \widehat{e}_{12}.$  Let $\B = (B_1, B_2) \in \fB_{M_\kappa}(\B)$ and let $\B' = (B_1', B_2') \in \fB_{M_\kappa}(\B).$
Suppose $\widehat{e}_{12} \in \widehat{B}_1.$  Then $\{ \widehat{e}_{12}, \widehat{e}_{34} \} \subset \widehat{B}_1'$ and consequently,
$B_2'$ is $v$-anchored and $|B_1' \cap \{ e_1,e_2,e_3, e_4 \} | =3.$  We also have
$B_1 = (\widehat{B}_1\backslash \widehat{e}_{12} ) \cup \{ e_1, e_2 \}$ and $B_2 = (\widehat{B}_2\backslash \widehat{e}_{34} ) \cup \{ e_3, e_4\}.$  If $\{ e_1, e_2 \} \subset B_1',$ then it is easy to show that
$\B \sim_1 \B'.$  Thus we may assume that $\{e_3, e_4 \} \subset B_1'.$  In addition, we may assume that $e_1\in B_1'$ and $e_2 \in B_2'.$  Since $\widehat{e}_{34} \in \widehat{B}_1',$ it follows by Lemma \ref{lem-Mhat1} that either $\widehat{e}_{23} \in \widehat{F}_I(\widehat{B}_1')$ or $\widehat{e}_{24} \in \widehat{F}_I(\widehat{B}_1').$  Suppose the former holds, then it seen that $\B'' = (B_1'', B_2'')$ where $B_1'' = B_1' - e_4 + e_2$ and $B_2'' = B_2' - e_2 + e_4$, is a base sequence.  Moreover, it is seen that $\B \sim_1 \B'' \sim_1 \B'$ and hence $\B \sim \B'.$  A similar argument can be used if $\widehat{e}_{24} \in \widehat{F}_I(\widehat{B}_1').$   From the above, we may assume that $\widehat{e}_{12} \in \widehat{B}_2.$  But now one can use the previous argument with the roles of $\B$ and $\B'$ switched to show that $\B' \sim \B.$ 
\end{proof}

\sms
\noindent{\bf Observation 2}:  Suppose $|E_{G_\kappa}(v)| = m = 2\kappa$, and $\widehat{\B}$ and $\widehat{\B}'$ are non-incidental sequences of bases.   Let $\B \in \fB_{M_\kappa}(\widehat{\B})$ and $\B' \in \fB_{M_\kappa}(\widehat{\B}')$.  Assume that Theorem \ref{the-kappa31} holds for all $k < \kappa$, and Theorem \ref{the-main2} holds for all $k < \kappa,$ if $\kappa \ge 3.$  If $\widehat{\B}$ and $\widehat{\B}'$ share at least one common base, then $\B \sim \B'.$

\begin{proof}
Suppose $\widehat{\B}$ and $\widehat{\B}'$ share at least one common base.  Let $\B'' \in \fB_{M_\kappa}(\widehat{\B}')$ where $\B \vdash \B''.$  Then it is seen that $\B \bumpeq \B''.$  It follows by Theorem \ref{the-linked} that $\B \sim \B''.$  However, by Lemma \ref{lem-linkedpullback}, $\B'' \sim \B'.$  Thus it follows that $\B \sim \B'.$
\end{proof} 

\begin{lemma}
Assume that $|E_{G_\kappa}(v)| = m = 2\kappa$.   Furthermore, assume that Theorem \ref{the-kappa31} holds for all $k < \kappa$ and Theorem \ref{the-main2} holds for all $k < \kappa,$ if $\kappa \ge 3.$   Suppose that $\widehat{\B} = (\widehat{B}_1, \dots ,\widehat{B}_\kappa)$ and  $\widehat{\B}' = (\widehat{B}_1', \dots ,\widehat{B}_\kappa')$ are compatible, non-incidental sequences of bases in $\widehat{M}_\kappa$ where $\widehat{\B} \sim \widehat{\B}'.$  Then  $\forall \B \in \fB_{M_\kappa}(\widehat{\B})$ and $\forall \B' \in \fB_{M_\kappa}(\widehat{\B}')$ we have
$\B \sim \B'$.
 \label{lem-noninbases}
\end{lemma}

\begin{proof}
Let $\B \in \fB_{M_\kappa}(\widehat{\B})$ and let $\B' \in \fB_{M_\kappa}(\widehat{\B}').$  By assumption, we have $\widehat{\B} \sim \widehat{\B}'.$     Thus for some integer $p$, we can transform $\widehat{\B}$ into $\widehat{\B}'$ by $p$ successive symmetric exchanges.  Let $\widehat{\D}_1 = \widehat{\B}$ and for $i = 1, \dots ,p-1$, let
$\widehat{\D}_{i+1} = (\widehat{D}_{(i+1)1}, \dots ,\widehat{D}_{(i+1)\kappa})$ denote the sequence of bases obtained after the $i$'th symmetric exchange, where $\widehat{\D}_p = \widehat{\B}';$ that is, $\widehat{\D}_i \sim_1 \widehat{\D}_{i+1},\ i =1, \dots ,p-1.$  We shall show that for $i = 1, \dots ,p-1$ and for all $\D_i \in \fB_{M_\kappa}(\widehat{\D}_i),$ there exists $\D_{i+1} \in \fB_{M_\kappa}(\widehat{\D}_{i+1})$ such that $\D_i \sim \D_{i+1}.$  Assume that we have already chosen $\D_j = (D_{j1}, \dots ,D_{j\kappa}) \in \B_{M_\kappa}(\widehat{\D}_j), \ j=1, \dots,i$ where $\D_1= \B$ and $\D_j \sim \D_{j+1},\ j = 1, \dots ,i-1.$   Choose $\D_{i+1}$ such that $\D_i \vdash \D_{i+1}.$
Since $\widehat{\D}_i \sim_1 \widehat{\D}_{i+1},$ it follows by Observation 1 that $\D_i \sim \D_{i+1}.$
Continuing, we may choose $\D_i \in \fB_{M_\kappa}(\widehat{\D}_i),\ i = 1, \dots ,p$ such that $\D_1 = \B$ and $\D_i \sim \D_{i+1}, \ i=1, \dots ,p-1.$   Thus $\B \sim \D_p$.     It now follows by Lemma \ref{lem-linkedpullback} that $\B\sim \D_p \sim \B'$ and thus $\B \sim \B'.$
\end{proof}

\section{The Proof of Theorems \ref{the-main2} and \ref{the-kappa31}: Part 1}\label{sec-mainproof1}
The remainder of this paper is devoted to proving Theorems \ref{the-main2} and \ref{the-kappa31} which we will do by induction on $k + r(M).$   When $k+r(M) =3$, the theorem is clearly true.   To complete the proofs of Theorems \ref{the-main2} and \ref{the-kappa31}, it will suffice to prove the following theorem: 

\begin{theorem}
Suppose that Theorem \ref{the-main2} is true for all $k < \kappa$, if $\kappa \ge 3.$  In addition, suppose that Theorem \ref{the-kappa31} is true for all $k<\kappa.$
Then 
\begin{itemize}
\item[a)] Theorem \ref{the-main2} is true for $k=\kappa$

and
\item[b)] Theorem \ref{the-kappa31} is true for $k=\kappa.$
\end{itemize}
\label{the-inductivestep}
\end{theorem}

We shall start by proving a), which we will use later to prove part b). 
Let $\B_i = (B_{i1}, \dots ,B_{i\kappa})\ i = 1,2$ be $v$-reduced, compatible sequences of bases in $M_\kappa.$    
By Theorem \ref{theorem3.2}, we may assume that either $\fM_{\B_1} = \fM_{\B_2}$, or $\fM_{\B_1} \ne \fM_{\B_2}$ and $\fM_{\B_1} \triangle \fM_{\B_2}$ has exactly one non-trivial component which is a $4$-cycle.  We need to consider both cases separately, the former case being much simpler than the latter.

Let $\widehat{\Omega}_\kappa = (\widehat{G}_\kappa, \widehat{\C}_\kappa)$ be the $v$-deleted linear biased graph obtained from $\Omega_\kappa$ and let $\widehat{M}_\kappa = M(\widehat{\Omega}_\kappa).$  We shall assume that $\B_i$ contains exactly $l$ $v$-unanchored bases (and $\kappa - l$ $v$-anchored ones).  Let $\{ e_1, \dots ,e_m \}$ denote the edges incident with $v$, noting that $m = \kappa + l.$  Our strategy to show that $\B_1 \sim \B_2$ involves first reducing $\B_i,\ i = 1,2$ to compatible sequences of bases $\widehat{\B}_i,\ i = 1,2$ in $\widehat{M}_\kappa.$   Next, assuming Theorem \ref{the-main2} holds for $\widehat{M}_\kappa,$ we have $\widehat{\B}_1 \sim \widehat{\B}_2$ and thus there are sequences of bases $\widehat{\D}_i,\ i = 1, \dots ,p$ in $\widehat{M}_\kappa$ such that 
$\widehat{\B}_1 = \widehat{\D}_1 \sim_1 \widehat{\D}_2 \sim_1 \cdots \sim_1 \widehat{\D}_p = \widehat{\B}_2.$ The biggest technical hurdle in the proof is showing that, in the case where $\fM_{\B_1} \ne \fM_{\
B_2},$ each sequence $\widehat{\D}_i$, has an associated sequence of bases $\D_i$ in $M_\kappa$ where $\D_1 = \B_1,\ \D_p = \B_2,$ and $\D_i \sim \D_{i+1},\ i = 1, \dots ,p-1.$
 
 \subsection{The Proof of Theorem \ref{the-inductivestep} a): The Case  $\fM_{\B_1} = \fM_{\B_2}$}\label{subsec-equal}
 
Suppose that $\fM_{\B_1} = \fM_{\B_2} = \fM.$  For simplicity, we may assume $\fM = \{\mathsf{x}_{2j-1}\mathsf{x}_{2j}\ \big| \ j = 1, \dots ,l \}$.  By swapping pairs of bases, we may assume that $\{ e_{2j-1}, e_{2j} \} \subset B_{ij},\ i = 1,2; j = 1, \dots ,l.$   For $i = 1,2$ and $j = 1, \dots ,l$ let $\widehat{B}_{ij} = (B_{ij}\backslash \{ e_{2j-1}, e_{2j} \} ) \cup \{ \widehat{e}_{(2j-1)2j} \}.$  For $i=1,2$ and $j= l+1, \dots ,\kappa$, let $\widehat{B}_{ij} = B_{ij}\backslash E_{G_\kappa}(v).$   We observe that if there is an unbalanced cycle $C$ where $E(C) \subseteq B_{ij}$ and $\{ e_{2j-1,2j} \} \subset E(C),$ then there is a cycle $\widehat{C}$ in $\widehat{G}$ corresponding to $C$ where $E(\widehat{C}) = (E(C)\backslash \{ e_{2j-1}, e_{2j} \} ) \cup \{ \widehat{e}_{(2j-1)2j} \}.$  It follows by Lemma \ref{lem-newbiased} that $\widehat{C}$ is unbalanced since $\fE(\widehat{C}) = C.$  In particular, $\widehat{B}_{ij}$ is seen to be a base of $\widehat{M}_{\kappa}.$
It follows that $\widehat{\B}_i = (\widehat{B}_{i1}, \dots ,\widehat{B}_{i\kappa}),\ i =1,2$ are compatible, non-incidental sequences of bases in $\widehat{M}_\kappa.$  By the inductive assumption, Theorem \ref{the-main2} holds for $\widehat{M}_\kappa.$  Thus there are base sequences $\widehat{\D}_i = (\widehat{D}_{i1}, \dots ,\widehat{D}_{i\kappa}),\ i = 1, \dots ,p$ in $\widehat{M}_\kappa$,  where $\widehat{\B}_1 = \widehat{\D}_1 \sim_1 \widehat{\D}_2 \sim_1 \cdots \sim_1 \widehat{\D}_p = \widehat{\B}_2.$  For $i=1, \dots ,p$, let $\D_i = (D_{i1}, \dots ,D_{i\kappa}) \in \fB_{M_\kappa}(\widehat{\D}_i).$

\sms
\noindent {\bf Note}:  If $l < \kappa,$ then by construction of the pull-back, the edges $e_{2l+1}, \dots ,e_m$ will belong to $v$-anchored bases in $\D_i,\ i =1, \dots ,p.$  If $\kappa =2,$ then $m = \kappa + l \le 3$ and it is fairly straightforward to show that $\D_i \sim \D_{i+1}$. 
Suppose $\kappa \ge 3.$  Since $\widehat{\D}_i \sim_1 \widehat{\D}_{i+1},$ it follows that $\widehat{D}_{ij} = \widehat{D}_{(i+1)j},$ for some $j\in \{ 1, \dots ,\kappa \}.$  Assuming that $\D_i \vdash \D_{i+1},$ if $\widehat{D}_{ij} \cap \widehat{E} \ne \emptyset,$ then $D_{(i+1)j} = D_{ij}.$  If $\widehat{D}_{ij} \cap \widehat{E} = \emptyset,$ then $D_{ij}$ and $D_{(i+1)j}$ are both $v$-anchored.
By switching edges, if necessary, we may assume that $e_{2l+1} \in D_{ij} \cap D_{(i+1)j}.$  It then follows that $D_{ij} = D_{(i+1)j}.$  Let $\D_i'$ (resp. $\D_{i+1}'$) be the base sequence obtained by deleting $D_{ij}$ (resp. $D_{(i+1)j}$) from $\D_i$ (resp. $\D_{i+1}$).  Then $\D_i'$ and $\D_{i+1}'$ are compatible sequences of $\kappa -1$ bases in $M_\kappa$.  Therefore, it follows that $\D_i' \sim \D_{i+1}'$ and consequently, $\D_i \sim \D_{i+1}.$  Since this holds for $i = 1, \dots ,p-1$, it follows that $\D_1 \sim \D_p$ and hence $\B_1 \sim \B_2.$  Because of this, we may assume that $\kappa =l$ and hence $m = 2\kappa.$
 
 \sms  

Given $m=2\kappa$ (from the above note) and
$\B_i\in \fB_{M_\kappa}(\widehat{\B}_i),\ i =1,2$, it follows by Lemma \ref{lem-noninbases} that $\B_1 \sim \B_2.$  This completes the case where $\fM_{\B_1} = \fM_{\B_2}$.

\subsection{The Proof of Theorem \ref{the-inductivestep} a): The Case  $\fM_{\B_1} \ne \fM_{\B_2}$}\label{subsec-notequal}

In this section, we shall assume $\fM_{\B_1} \ne \fM_{\B_2}$ and $\fM_{\B_1} \triangle \fM_{\B_2}$ has exactly one non-trivial component which is a $4$-cycle.  Furthermore,
we will assume throughout that there are no $v$-reduced sequences of bases $\B_i',\ i = 1,2$ for which $\B_i' \sim \B_i,\ i = 1,2$ and $\fM_{\B_1'} = \fM_{\B_2'}.$ 
Without loss of generality, we may assume that $\mathsf{x}_1\mathsf{x}_2\mathsf{x}_3\mathsf{x}_4\mathsf{x}_1$ is the $4$-cycle component of $\fM_{\B_1} \triangle \fM_{\B_2}$ where $\mathsf{x}_1\mathsf{x}_2, \mathsf{x}_3\mathsf{x}_4 \in \fM_{\B_1},$ and $\mathsf{x}_1\mathsf{x}_4, \mathsf{x}_2\mathsf{x}_3 \in \fM_{\B_2}.$   If $\mathsf{x}_1\mathsf{x}_2,\ \mathsf{x}_3\mathsf{x}_4$ is $\B_1$-switchable and $\mathsf{x}_1\mathsf{x}_4,\ \mathsf{x}_2\mathsf{x}_3$ is $\B_2$-switchable, then one can find sequences of bases $\B_i',\ i = 1,2$ where $\B_i' \sim \B_i,\ i = 1,2$ and
$\fM_{\B_1'} = \fM_{\B_2'}.$  Thus we may assume that $\mathsf{x}_1\mathsf{x}_2,\ \mathsf{x}_3\mathsf{x}_4$ is not $\B_1$-switchable.  We may also assume that $\{ e_{2j-1}, e_{2j} \} \in B_{1j},$
$\{ e_{j}, e_{5-j} \} \in B_{2j},\ j = 1,2$ and $\{ e_{2j-1}, e_{2j} \} \subset B_{1j} \cap B_{2j},\ j = 3, \dots ,l.$  We define $\widehat{\B}_i = (\widehat{B}_{i1}, \dots ,\widehat{B}_{i\kappa}),\ i = 1,2$ where 

\begin{align*}
\widehat{B}_{1j} &= (B_{1j}\backslash \{ e_{2j-1}, e_{2j} \} ) \cup \{ \widehat{e}_{(2j-1)2j} \},\ j = 1,2.\\ 
\widehat{B}_{2j} &= (B_{2j}\backslash \{ e_j, e_{5-j} \} ) \cup \{ \widehat{e}_{j(5-j)} \},\ j = 1,2.\\ 
\widehat{B}_{ij} &= (B_{ij}\backslash \{ e_{2j-1}, e_{2j} \} ) \cup \{ \widehat{e}_{(2j-1)2j} \},\ i =1,2; \ j = 3,\dots ,l.\\ 
\widehat{B}_{ij} &= B_{ij}\backslash E_{G_\kappa}(v),\ i =1,2;\ j = l+1, \dots ,\kappa.\\
\end{align*}

 Let $I = \{ 1,2,3,4\}$ and let $\widehat{H} = \widehat{H}_I.$  We note that $\widehat{\B}_i,\ i = 1,2$ are not compatible sequences of bases in $\widehat{M}_\kappa$.  Also, since $\mathsf{x}_1\mathsf{x}_2, \mathsf{x}_3\mathsf{x}_4$ are not $\B_1$-switchable, it follows that $(\widehat{B}_{11}, \widehat{B}_{12})$ is not $\widehat{e}_{12}, \widehat{e}_{34}$-switchable.  
 Thus Lemma \ref{lem-Mhat3} implies that $$\{ \widehat{F}_I(\widehat{B}_{11}), \widehat{F}_I(\widehat{B}_{12}) \} = \{ \{ \widehat{e}_{12}, \widehat{e}_{34}, \widehat{e}_{14}, \widehat{e}_{23} \}, \ \{ \widehat{e}_{12}, \widehat{e}_{34}, \widehat{e}_{13}, \widehat{e}_{24} \} \}.$$  We may assume
$\widehat{F}_I(\widehat{B}_{11}) = \{ \widehat{e}_{12}, \widehat{e}_{34}, \widehat{e}_{14}, \widehat{e}_{23} \}$ and $\widehat{F}_I(\widehat{B}_{12}) = \{ \widehat{e}_{12}, \widehat{e}_{34}, \widehat{e}_{13}, \widehat{e}_{24} \}.$  We observe that $(\widehat{B}_{21}, \widehat{B}_{22})$ is not $\widehat{e}_{12}, \widehat{e}_{34}$-amenable as (by assumption) there is no $v$-reduced sequence $\B_2'$ for which $\B_2 \sim \B_2'$ and $\fM_{\B_1} = \fM_{\B_2'}.$  

\subsubsection{Defining Compatible Sequences $\widehat{\B}_1'$ and $\widehat{\B}_2'$}\label{subsubsec-B1'B2'}

Our first task is to alter the sequences $\widehat{\B}_i,\ i =1,2$ slightly so that they become compatible sequences.  We shall need the following simple lemma.

\begin{lemma}
There exists $j\in \{ 1,2 \}$ such that 
\begin{equation}  \widehat{F}_I(\widehat{B}_{11}) \cap \widehat{F}_I(\widehat{B}_{2j}) \cap \{ \widehat{e}_{14}, \widehat{e}_{23} \} \ne \emptyset \label{eqn1}\end{equation}
and
\begin{equation}  \widehat{F}_I(\widehat{B}_{12}) \cap \widehat{F}_I(\widehat{B}_{2(3-j)}) \cap \{ \widehat{e}_{13}, \widehat{e}_{24} \}\ne \emptyset. \label{eqn2}\end{equation}
\label{lem-efhat}
\end{lemma}

\begin{proof}
Suppose first that $(\widehat{B}_{21}, \widehat{B}_{22})$ is not $\widehat{e}_{13}, \widehat{e}_{24}$-amenable.  Then (noting that $(\widehat{B}_{21}, \widehat{B}_{22})$ is not $\widehat{e}_{12}, \widehat{e}_{34}$-amenable) Lemma \ref{lem-Mhat3} implies that 
$$\{\widehat{F}_I(\widehat{B}_{21}), \widehat{F}_I(\widehat{B}_{22}) \} = \{ \{ \widehat{e}_{14}, \widehat{e}_{23}, \widehat{e}_{12}, \widehat{e}_{34} \}, \{ \widehat{e}_{14}, \widehat{e}_{23}, \widehat{e}_{13}, \widehat{e}_{24} \} \}.$$
If $\widehat{F}_I(\widehat{B}_{21}) = \{ \widehat{e}_{14}, \widehat{e}_{23}, \widehat{e}_{12}, \widehat{e}_{34} \}$ and $\widehat{F}_I(\widehat{B}_{22}) = \{ \widehat{e}_{14}, \widehat{e}_{23}, \widehat{e}_{13}, \widehat{e}_{24} \}$,
then (\ref{eqn1}) and (\ref{eqn2}) hold for $j=1.$  Otherwise, if $\widehat{F}_I(\widehat{B}_{21}) = \{ \widehat{e}_{14}, \widehat{e}_{23}, \widehat{e}_{13}, \widehat{e}_{24} \}$ and $\widehat{F}_I(\widehat{B}_{22}) = \{ \widehat{e}_{14}, \widehat{e}_{23}, \widehat{e}_{12}, \widehat{e}_{34} \}$,
then (\ref{eqn1}) and (\ref{eqn2}) hold for $j=2.$  Suppose now that $(\widehat{B}_{21}, \widehat{B}_{22})$ is $\widehat{e}_{13}, \widehat{e}_{24}$-amenable.   Thus either $\widehat{e}_{13} \in \widehat{B}_{21}$ and $\widehat{e}_{24} \in \widehat{B}_{22},$ or $\widehat{e}_{24} \in \widehat{B}_{21}$ and $\widehat{e}_{13} \in \widehat{B}_{22}.$ 
In both cases, (\ref{eqn1}) and (\ref{eqn2}) hold when $j=1$.
\end{proof}

\subsubsection{The Sequences $\B_i',\ i = 1,2$}

By Lemma \ref{lem-efhat}, we may assume $\widehat{e}_{23} \in \widehat{F}_I(\widehat{B}_{11}) \cap \widehat{F}_I(\widehat{B}_{21})$ and $\widehat{e}_{24} \in \widehat{F}_I(\widehat{B}_{12}) \cap \widehat{F}_I(\widehat{B}_{22}).$ Let

\begin{itemize} 
\item $\widehat{B}_{11}' = \widehat{B}_{11} - \widehat{e}_{12} + \widehat{e}_{23}$, $\widehat{B}_{12}' = \widehat{B}_{12} - \widehat{e}_{34} + \widehat{e}_{24},$ and $\widehat{B}_{1j}' = \widehat{B}_{1j},\ j = 3, \dots ,\kappa.$
\item $\widehat{B}_{21}' = \widehat{B}_{21} - \widehat{e}_{14} + \widehat{e}_{23}$, $\widehat{B}_{22}' = \widehat{B}_{22} - \widehat{e}_{23} + \widehat{e}_{24},$ and $\widehat{B}_{2j}' = \widehat{B}_{1j},\ j = 3, \dots ,\kappa.$
\item $\widehat{\B}_i' = (\widehat{B}_{i1}', \dots ,\widehat{B}_{i\kappa}'),\ i = 1,2$ 
\end{itemize}

We observe that $\widehat{\B}_i',\ i = 1,2$ are compatible sequences of bases in $\widehat{M}_{\kappa}$ and $\widehat{F}_I(\widehat{B}_{ij}') = \widehat{F}_I(\widehat{B}_{ij}),$ $\forall i,j.$

\subsubsection{The Sequences $\widehat{\D}_i',\ i = 1, \dots ,p$}\label{subsubsec-Di'}

Given that Theorem \ref{the-main2} holds for $\widehat{M}_\kappa,$ we have $\widehat{\B}_1' \sim \widehat{\B}_2'$ and thus there are sequences of bases $\widehat{\D}_i',\ i = 1, \dots ,p$ in $\widehat{M}_\kappa$ such that $\widehat{\B}_1' =\widehat{\D}_1' \sim_1 \widehat{\D}_2' \sim_1 \cdots \sim_1 \widehat{\D}_p' = \widehat{\B}_2'.$
Our immediate task is to show that for each sequence $\widehat{\D}_i'$, one can associate a sequence of bases in $M_\kappa.$ 
The biggest technical problem here is that the base sequence $\widehat{\D}_i'$ may be incidental since its bases use the edges $\widehat{e}_{23}$ and $\widehat{e}_{24}$ which are incidental.  As a result, it is possible that
$\fB_{M_\kappa}(\widehat{\D}_i') = \emptyset.$  To get around this problem, we shall look {\it perturbations} of the sequence $\widehat{\D}_i'.$

\section{Perturbations, Amenability, and Switchability of 
Base Sequences in $\widehat{M}_\kappa$}\label{subsubsec-perturbation}

Let $\widehat{\D} = (\widehat{D}_1, \dots ,\widehat{D}_\kappa)$ be a sequence of bases in $\widehat{M}_\kappa$ which is compatible with $\widehat{\B}_i',\ i = 1,2.$ 

\subsection{Split and Fused Base Sequences}\label{subsec-splitfused}
 
We say that $\widehat{\D}$ is {\bf split} if $\widehat{e}_{23}$ and $\widehat{e}_{24}$ belong to different bases in $\widehat{\D};$  otherwise, we say that $\widehat{\D}$ is {\bf fused}.   Note that if $\widehat{\D}$ is fused, then it is non-incidental and hence $\fB_{M_\kappa}(\widehat{\D}) \ne \emptyset.$

\subsection{$\widehat{e}, \widehat{f}$-Amenable, Switchable  Sequences}\label{subsec-efhatamenable}
 
Let $I = \{ 1,2,3,4 \},\ \widehat{H} = \widehat{H}_I,$ and let  $\widehat{e}, \widehat{f} \in \widehat{E} = \{ \widehat{e}_{ij}\ \big| \ i,j \in \{ 1, \dots ,m \}$ be distinct edges.   
Assume  $(\widehat{D}_{j_1} \cup \widehat{D}_{j_2}) \cap \widehat{E} = \{ \widehat{e}_{23}, \widehat{e}_{24} \}.$ Let $\widehat{\D}' = (\widehat{D}_1', \dots ,\widehat{D}_\kappa')$ where $\widehat{D}_j' = \widehat{D}_j,\ \forall j\in \{ 1,2, \dots ,\kappa\} \backslash \{ j_1, j_2 \}$, and $(\widehat{D}_{j_1}', \widehat{D}_{j_2}')$ is $\widehat{H}$-viable. 

If $(\widehat{D}_{j_1}, \widehat{D}_{j_2}) \xrightarrow[\widehat{e}, \widehat{f}]{\widehat{e}_{23}, \widehat{e}_{24}} (\widehat{D}_{j_1}', \widehat{D}_{j_2}'),$  then we refer to $\widehat{\D}'$ as a $\mathbf{\widehat{e}, \widehat{f}}${\bf- perturbation} of $\widehat{\D}.$

If $(\widehat{D}_{j_1}, \widehat{D}_{j_2}) \curvearrowright (\widehat{D}_{j_1}', \widehat{D}_{j_2}')$ and $(\widehat{D}_{j_1}' \cup \widehat{D}_{j_2}') \cap \widehat{E} \in \left\{ \{ \widehat{e}_{13}, \widehat{e}_{24} \}, \{ \{ \widehat{e}_{14}, \widehat{e}_{23} \} \right\}$, then $\widehat{\D}'$ is called a $\mathbf{\widehat{e}_{12}, \widehat{e}_{34}}${\bf-switch} of $\widehat{\D}$ and we say that $\widehat{\D}$ is $\mathbf{\widehat{e}_{12}, \widehat{e}_{34}}${\bf-switchable}.

%

As a first step towards completing the proof of Theorem \ref{the-inductivestep} a), we shall prove the following theorem:

\begin{theorem} 
Suppose that for some $i\in \{ 1, \dots , \kappa \}$, $\widehat{\D}_i'$ is not 
$\widehat{e}_{12}, \widehat{e}_{34}$-switchable but either fused or $\widehat{e}_{12}, \widehat{e}_{34}$-amenable.  Then one of following hold:
\begin{itemize}
\item $\widehat{\D}_{i+1}'$ is $\widehat{e}_{12}, \widehat{e}_{34}$-switchable.
\item $\widehat{\D}_{i+1}'$ is fused and not $\widehat{e}_{12}, \widehat{e}_{34}$-switchable.
\item $\widehat{\D}_{i+1}'$ is split, not $\widehat{e}_{12}, \widehat{e}_{34}$-switchable, but $\widehat{e}_{12}, \widehat{e}_{34}$-amenable.
\end{itemize}\label{the-preproextend2}
\end{theorem}  
 
This theorem will follow directly from the following proposition.
\begin{proposition}
Let $\widehat{\X}_i = (\widehat{X}_{i1}, \dots ,\widehat{X}_{i\kappa}),\ i = 1,2$ be sequences of bases in $\widehat{M}_\kappa$ where\\ $\widehat{\X}_i$ is compatible with $\B_i'$, $\widehat{\X}_1 \sim_1 \widehat{\X}_2$ and for $i = 1,2,$ $\widehat{\X}_i$ is not $\widehat{e}_{12}, \widehat{e}_{34}$-switchable.
Then we have the following: 
\begin{itemize}
\item[a)]  If $\widehat{\X}_1$ is $\widehat{e}_{12}, \widehat{e}_{34}$-amenable and $\widehat{\X}_i,\ i = 1,2$ are split, then $\widehat{\X}_2$ is $\widehat{e}_{12}, \widehat{e}_{34}$-amenable.
\item[b)]  If $\widehat{\X}_1$ is fused and $\widehat{\X}_2$ is split, then $\widehat{\X}_2$ is $\widehat{e}_{12}, \widehat{e}_{34}$-amenable.
\end{itemize}
\label{pro-extend2}
\end{proposition}

\begin{proof}
Let $j_1, j_2 \in \{ 1, \dots ,\kappa \}$ where a symmetric exchange between $\widehat{X}_{1j_1}$ and $\widehat{X}_{1j_2}$ transforms $\widehat{\X}_1$ into $\widehat{\X}_{2}.$  To prove a), assume that $\widehat{\X}_i,\ i = 1,2$ are split.  We may assume that $\widehat{e}_{23} \in \widehat{X}_{11}$ and $\widehat{e}_{24} \in \widehat{X}_{12}.$
It follows from Lemma \ref{lem-Mhat3} that $\widehat{F}_I(\widehat{X}_{11}) = \{ \widehat{e}_{12}, \widehat{e}_{34}, \widehat{e}_{23}, \widehat{e}_{14} \}$ and $\widehat{F}_I(\widehat{X}_{12}) = \{ \widehat{e}_{12}, \widehat{e}_{34}, \widehat{e}_{13}, \widehat{e}_{24} \}.$  If $\{j_1,j_2\} = \{ 1,2\},$ then it follows from Theorem \ref{the-Mhat5} that $(\widehat{X}_{21}, \widehat{X}_{22})$ is $\widehat{e}_{12}, \widehat{e}_{34}$-amenable, implying that $\widehat{\X}_2$ is $\widehat{e}_{12}, \widehat{e}_{34}$-amenable.  Thus we may assume that $\{ j_1, j_2 \} \ne \{ 1,2 \}.$  If $\{ j_1,  j_2 \} \cap \{ 1,2 \} = \emptyset,$ then $\widehat{X}_{1j} = \widehat{X}_{2j},\ j = 1,2$ and thus it is clear that $\widehat{\X}_2$ is $\widehat{e}_{12}, \widehat{e}_{34}$-amenable, since $\widehat{\X}_1$ is.  Thus we may assume $| \{ j_1, j_2 \} \cap \{ 1,2 \} | =1.$
In particular, $\widehat{X}_{2j} = \widehat{X}_{1j}$, for some $j\in \{ 1,2 \}.$  Without loss of generality, we may assume that $\widehat{X}_{11} = \widehat{X}_{21}$ and consequently $\widehat{F}_I(\widehat{X}_{21}) = \widehat{F}_I(\widehat{X}_{11}) = \{ \widehat{e}_{12}, \widehat{e}_{34}, \widehat{e}_{23}, \widehat{e}_{14} \}$.  Since $\widehat{\X}_2$ is not $\widehat{e}_{12}, \widehat{e}_{34}$-switchable, it follows from Lemma \ref{lem-Mhat3} that $\widehat{F}_I(\widehat{X}_{22}) = \{ \widehat{e}_{12}, \widehat{e}_{34}, \widehat{e}_{13}, \widehat{e}_{24} \}.$  It now follows that $\widehat{\X}_2$ is $\widehat{e}_{12}, \widehat{e}_{34}$-amenable.  This proves a).

To prove b),  we may assume that $\widehat{e}_{23}, \widehat{e}_{24} \in \widehat{X}_{11}$ and $j_1 = 1, j_2=2.$   Furthermore, we may assume that there exists $\widehat{e} \in \widehat{X}_{12}$ such that $\widehat{X}_{21} = \widehat{X}_{11} -\widehat{e}_{24} +\widehat{e}$ and $\widehat{X}_{22} = \widehat{X}_{12} -\widehat{e} + \widehat{e}_{24}.$  Note that $\widehat{e}_{24} \in C(\widehat{e},\widehat{X}_{11}),$ the fundamental circuit with respect to $\widehat{X}_{11}$ which contains $\widehat{e}.$
Suppose that $\widehat{\X}_{2}$ is not $\widehat{e}_{12}, \widehat{e}_{34}$-amenable.  Then it follows from Lemma \ref{lem-Mhat3} that $\widehat{F}_I (\widehat{X}_{2j}) = \{ \widehat{e}_{12}, \widehat{e}_{23}, \widehat{e}_{24} \},\ j = 1,2.$
  
We shall first show that either $\widehat{X}_{11} - \widehat{e}_{24} + \widehat{e}_{12} \in \B(\widehat{M}_\kappa)$ or  $\widehat{X}_{11} - \widehat{e}_{23} + \widehat{e}_{12} \in \B(\widehat{M}_\kappa).$
 Since $\widehat{e}_{12} \in \widehat{F}_I(\widehat{X}_{21})$, it follows that $\widehat{X}_{21}' = \widehat{X}_{21} - \widehat{e}_{23} + \widehat{e}_{12} \in \B(\widehat{M}_\kappa).$  Thus $\widehat{e}_{23} \in C(\widehat{e}_{12}, \widehat{X}_{21}),$ the fundamental circuit with respect to $\widehat{X}_{21}$ which contains $\widehat{e}_{12}.$
If  $\widehat{e} \in C(\widehat{e}_{12}, \widehat{X}_{21}),$ then by the circuit elimination axiom there is a circuit $\widehat{C} \subseteq (C(\widehat{e}, \widehat{X}_{11})\cup C(\widehat{e}_{12}, \widehat{X}_{21})\backslash \{ \widehat{e} \}$ containing $\widehat{e}_{24}$ (noting that $\widehat{e}_{24} \in C(\widehat{e}, \widehat{X}_{11})\backslash C(\widehat{e}_{12}, \widehat{X}_{21})$).  If $\widehat{e}_{12} \not\in \widehat{C},$ then $\widehat{C} \subset \widehat{X}_{11},$ a contradiction.  Thus $\widehat{e}_{12} \in \widehat{C},$ and hence $\widehat{C} = C(\widehat{e}_{12}, \widehat{X}_{11})$.  Now $\widehat{X}_{11} - \widehat{e}_{24} + \widehat{e}_{12} \in \B(\widehat{M}_\kappa)$ since $\widehat{e}_{24} \in \widehat{C}.$
   In this case, our assertion is true.  We suppose therefore that $\widehat{e} \not\in C(\widehat{e}_{12}, \widehat{X}_{21}).$   Then $C(\widehat{e}_{12}, \widehat{X}_{11}) = C(\widehat{e}_{12}, \widehat{X}_{21}).$  Since $\widehat{e}_{12} \in \widehat{F}_I(\widehat{X}_{21}),$ it follows that $\widehat{e}_{23} \in C(\widehat{e}_{12}, \widehat{X}_{21}) = C(\widehat{e}_{12}, \widehat{X}_{11}).$  Thus $\widehat{X}_{11} - \widehat{e}_{23} + \widehat{e}_{21} \in \B(\widehat{M}_\kappa).$
   
   From the above, we may assume (without loss of generality) that $\widehat{X}_{11}' = \widehat{X}_{11} - \widehat{e}_{24} + \widehat{e}_{12} \in \B(\widehat{M}_\kappa).$  Let $\widehat{X}_{12}' = \widehat{X}_{12}.$  Then $(\widehat{X}_{11}, \widehat{X}_{12}) \xrightarrow[\widehat{e}_{12},\widehat{e}_{23}]{\widehat{e}_{23},\widehat{e}_{24}} (\widehat{X}_{11}', \widehat{X}_{12}').$ Note that $(\widehat{X}_{11}', \widehat{X}_{12}')$ is also $\widehat{e}_{12}, \widehat{e}_{13}$-amenable since $\{ \widehat{e}_{12}, \widehat{e}_{13}, \widehat{e}_{23} \}$ is a circuit.
Consider the fundamental circuit $C(\widehat{e}_{13}, \widehat{X}_{22}).$  Since $\widehat{e}_{13} \not\in \widehat{F}_I(\widehat{X}_{22}),$ it follows that $\widehat{e}_{24} \not\in C(\widehat{e}_{13}, \widehat{X}_{22}).$   Thus $C(\widehat{e}_{13}, \widehat{X}_{12}) = C(\widehat{e}_{13}, \widehat{X}_{22}).$ 
Let $\widehat{C}^* = C^*(\widehat{e}_{23}, \widehat{X}_{11}')$ be the fundamental cocircuit with respect to $\widehat{X}_{11}'$ which contains $\widehat{e}_{23}.$ 
  We observe that $\widehat{e}_{13} \not\in \widehat{L} = E(\widehat{M}_\kappa)\backslash \widehat{C}^*;$ a hyperplane; otherwise, $\widehat{e}_{12}, \widehat{e}_{13} \in \widehat{L},$ implying that $\widehat{e}_{23} \in \widehat{L},$ a contradiction. Thus $\widehat{e}_{13} \in \widehat{C}^*$ and hence $|C(\widehat{e}_{13}, \widehat{X}_{22}) \cap \widehat{C}^*| \ge 2.$  Thus there exists $\widehat{f} \in (C(\widehat{e}_{13}, \widehat{X}_{22}) \cap \widehat{C}^*)\backslash \{ \widehat{e}_{13} \}.$  We see that $\widehat{X}_{11}'' = \widehat{X}_{11}' -\widehat{e}_{23} + \widehat{f} \in \B(\widehat{M}_\kappa)$  and $\widehat{X}_{12}'' = \widehat{X}_{12} - \widehat{f} + \widehat{e}_{13} \in \B(\widehat{M}_\kappa).$  Note that $\widehat{e}_{12} \in \widehat{X}_{11}''$ and $\widehat{e}_{13} \in \widehat{X}_{12}''.$  Furthermore, $(\widehat{X}_{11}', \widehat{X}_{12}') \curvearrowright_1 (\widehat{X}_{11}'', \widehat{X}_{12}'').$  Since $\widehat{\X}_1$ is not
 $\widehat{e}_{12}, \widehat{e}_{34}$-switchable, $(\widehat{X}_{11}'', \widehat{X}_{12}'')$ is not $\widehat{e}_{12}, \widehat{e}_{34}$-switchable and it follows that either
\begin{itemize}
\item[i)] $\left\{ \widehat{F}_I(\widehat{X}_{11}''), \widehat{F}_I(\widehat{X}_{12}'') \right\} = \left\{ \{\widehat{e}_{12}, \widehat{e}_{34}, \widehat{e}_{13}, \widehat{e}_{24} \}, \ \{\widehat{e}_{12}, \widehat{e}_{34}, \widehat{e}_{23}, \widehat{e}_{14} \} \right\}$ or
 \item[ii)] $\widehat{F}_I(\widehat{X}_{11}'') =  \widehat{F}_I(\widehat{X}_{12}'') = \{ \widehat{e}_{12}, \widehat{e}_{13}, \widehat{e}_{14} \}.$
 \end{itemize}
  
Suppose i) holds.  Then $$\widehat{F}_I(\widehat{X}_{11}'') = \{ \widehat{e}_{12}, \widehat{e}_{34}, \widehat{e}_{23}, \widehat{e}_{14} \},\ \widehat{F}_I(\widehat{X}_{12}'') = \{ \widehat{e}_{12}, \widehat{e}_{34}, \widehat{e}_{13}, \widehat{e}_{24} \} .$$   Consider the fundamental circuit $C(\widehat{e}_{24}, \widehat{X}_{12}'').$   Since $\widehat{e}_{24} \in \widehat{F}_I(\widehat{X}_{12}''),$ it follows that $\widehat{e}_{13} \in C(\widehat{e}_{24}, \widehat{X}_{12}'').$  If $\widehat{e} \not\in C(\widehat{e}_{24}, \widehat{X}_{12}''),$ then 
$C(\widehat{e}_{13}, \widehat{X}_{22}) = C(\widehat{e}_{24}, \widehat{X}_{12}'')$ and hence
$\widehat{e}_{13} \in \widehat{F}_I(\widehat{X}_{22}),$ a contradiction.  Thus $\widehat{e} \in C(\widehat{e}_{24}, \widehat{X}_{12}'').$  This means that $\widehat{X}_{12}''' = \widehat{X}_{12}'' - \widehat{e} + \widehat{e}_{24} = \widehat{X}_{12} - \widehat{e} - \widehat{f} + \widehat{e}_{13} + \widehat{e}_{24} \in \B(\widehat{M}_\kappa).$
To obtain a contradiction, it suffices to show that $\widehat{X}_{11}''' = \widehat{X}_{11} - \widehat{e}_{23} - \widehat{e}_{24} + \widehat{e} + \widehat{f} \in \B(\widehat{M}_\kappa);$  this is because 
$(\widehat{X}_{11}'', \widehat{X}_{12}'') \curvearrowright_1 (\widehat{X}_{11}''', \widehat{X}_{12}''')$ (straightforward to show) and $(\widehat{X}_{11}''' \cup \widehat{X}_{12}''') \cap E(\widehat{H}) = \{ \widehat{e}_{13}, \widehat{e}_{24} \}.$ It will then follow that $(\widehat{X}_{11}, \widehat{X}_{12})$ is $\widehat{e}_{12}, \widehat{e}_{34}$-switchable, contradicting our assumptions.  Let $\widehat{X}_{11}^{(4)} = \widehat{X}_{11}'' - \widehat{f} + \widehat{e}.$  Then $\widehat{X}_{11}^{(4)} = \widehat{X}_{21} - \widehat{e}_{23} + \widehat{e}_{12}.$  Since $\widehat{e}_{12} \in \widehat{F}_I(\widehat{X}_{21}),$ it follows that $\widehat{X}_{11}^{(4)} \in \B(\widehat{M}_\kappa).$  Furthermore, $\widehat{F}_I(\widehat{X}_{11}^{(4)}) = \widehat{F}_I(\widehat{X}_{21}).$  Thus we see that $\widehat{F}_I(\widehat{X}_{11}'')\backslash \widehat{F}_I(\widehat{X}_{11}^{(4)} \ne \emptyset.$  It now follows by Lemma \ref{lem-Mhat3.1} i) that $\widehat{X}_{11}'' - \widehat{e}_{12} + \widehat{e} \in \B(\widehat{M}_\kappa);$ that is, $\widehat{X}_{11} - \widehat{e}_{23} - \widehat{e}_{24} + \widehat{e} + \widehat{f} \in \B(\widehat{M}_\kappa).$

Suppose ii) holds instead.    The proof that $\widehat{X}_{11} - \widehat{e}_{23} - \widehat{e}_{24} + \widehat{e} + \widehat{f} \in \B(\widehat{M}_\kappa)$ is the same as before.   To show that $\widehat{X}_{12} - \widehat{e} - \widehat{f} + \widehat{e}_{13} + \widehat{e}_{24} \in \B(\widehat{M}_\kappa)$, consider the fundamental circuit $C(\widehat{e}_{24}, \widehat{X}_{12}'').$  Since $\widehat{e}_{24} \not\in \widehat{F}_I(\widehat{X}_{12}''),$ it follows that $\widehat{e}_{13} \not\in C(\widehat{e}_{24}, \widehat{X}_{12}'').$  If $\widehat{e} \not\in C(\widehat{e}_{24}, \widehat{X}_{12}''),$ then 
$C(\widehat{e}_{24}, \widehat{X}_{12}'') \subseteq \widehat{X}_{22},$ a contradiction.  Thus $\widehat{e} \in C(\widehat{e}_{24}, \widehat{X}_{12}'').$  This means that $\widehat{X}_{12}'' - \widehat{e} + \widehat{e}_{24} = \widehat{X}_{12} - \widehat{e} - \widehat{f} + \widehat{e}_{13} + \widehat{e}_{24} \in \B(\widehat{M}_\kappa).$
\end{proof}

\section{The Proof of Theorem \ref{the-inductivestep} a): Part 2}\label{sec-mainproof2}

Recall the sequences $\widehat{\B}_1' = \widehat{\D}_1' \sim_1 \widehat{\D}_2' \sim_1 \cdots \sim_1 \widehat{\D}_p' = \widehat{\B}_2'.$  In this section, we shall assume $\kappa \ge 2$ and
$\widehat{\D}_j',\ j = 1, \dots ,i$ are not $\widehat{e}_{12}, \widehat{e}_{34}$-switchable.  Since $\widehat{\B}_1'$ is not $\widehat{e}_{12}, \widehat{e}_{34}$-switchable and $\widehat{\B}_1' = \widehat{\D}_1' \sim_1 \widehat{\D}_2' \sim_1 \cdots \sim_1 \widehat{\D}_i'$, it follows by Theorem \ref{the-preproextend2} that for $j = 1, \dots ,i,$ if $\widehat{\D}_j'$ is split, then it is $\widehat{e}_{12}, \widehat{e}_{34}$-amenable.  

\subsection{The Sequences $\widehat{\D}_j'',\ j = 1, \dots i$}\label{subsec-Dj''}

We define sequences $\widehat{\D}_j'',\ j = 1, \dots ,i$ as follows:  if $\widehat{\D}_j'$ is fused then $\widehat{\D}_j'' := \widehat{\D}_j';$ otherwise, $\widehat{\D}_j'$ is split and we define $\widehat{\D}_j''$ to be a $\widehat{e}_{12}, \widehat{e}_{34}$-perturbation of $\widehat{\D}_j'.$  The main theorem in this section is the following:

\begin{theorem}
Assume $\kappa \ge 2$  and Theorem \ref{the-kappa31} holds for all $k < \kappa$ and Theorem \ref{the-main2} holds for all $k<\kappa$, if $\kappa \ge 3.$  Suppose $\widehat{\D}_j',\ j = 1,2, \dots ,i$ are not $\widehat{e}_{12}, \widehat{e}_{34}$-switchable.  Let $\D_j'' \in \fB_{M_\kappa}(\widehat{\D}_j''),\ j = 1, \dots ,i.$  Then $\D_1'' \sim \D_2'' \sim \cdots \sim \D_i''.$
\label{the-preproextend3}
\end{theorem}

\noindent{\bf Note}:   We will be able to assume $|E_{G_\kappa}(v)| = m = 2\kappa.$  For suppose that $m<2\kappa.$  Then the edge $e_m$ belongs to $v$-anchored bases in $\D_j'',\ j = 1, \dots ,i.$   We shall show that for all $i' \in \{ 1, \dots, i-1 \}$, $\D_{i'}'' \sim \D_{i'+1}''.$   
Let $i'\in \{ 1, \dots ,i-1\}$ and assume that $\D_{i'}'' \vdash \D_{i'+1}''.$   If $\kappa = 2,$ then $m \le 3$, contrary to our assumptions.  Thus $\kappa \ge 3$ and hence
there exists $j\in \{ 1, \dots ,\kappa \}$ for which $\widehat{D}_{i'j}' = \widehat{D}_{(i'+1)j}'.$  Then we may assume that $\widehat{D}_{i'j}'' = \widehat{D}_{(i'+1)j}''.$  If $\widehat{D}_{i'j}'' \cap \widehat{E} \ne \emptyset,$ then we may choose $D_{(i'+1)j}'' = D_{i'j}''.$  Otherwise, if $\widehat{D}_{i'j}'' \cap \widehat{E} = \emptyset,$  then $D_{i'j}''$ and $D_{(i'+1)j}''$ are $v$-anchored.  By swapping edges in $E_{G_\kappa}(v)$, if necessary, we may assume that $e_m \in D_{i'j}'' \cap D_{(i'+1)j}''$.  Then $D_{i'j}'' = D_{(i'+1)j}''.$  Let $\D_{i'}'''$ (resp. $\D_{i'+1}'''$) be the base sequence obtained from $\D_{i'}''$ (resp. $\D_{i'+1}''$) by deleting $D_{i'j}''$ (resp. $D_{(i'+1)j}''$).  Then $\D_{i'}'''$ and $\D_{i'+1}'''$ are compatible sequences of $\kappa -1$ bases.  Since we are assuming that Theorem \ref{the-main2} holds for all $k< \kappa,$ it follows that $\D_{i'}''' \sim \D_{i'+1}'''.$  We now see that $\D_{i'}'' \sim \D_{i'+1}''.$  

For the rest of this section, we shall assume that  $m = 2\kappa.$ The proof of Theorem \ref{the-preproextend3} will follow directly from the following proposition:

\begin{proposition}
Assume $\kappa \ge 2$  and Theorem \ref{the-kappa31} holds for all $k < \kappa$ and Theorem \ref{the-main2} holds for all $k<\kappa$, if $\kappa \ge 3.$
Let $\widehat{\X}_i = (\widehat{X}_{i1}, \dots , \widehat{X}_{i\kappa})\ i = 1,2$ be sequences of bases in $\widehat{M}_\kappa$ where $\widehat{\X}_i$ is compatible with $\widehat{\B}_i',\ i =1,2,$ and $\widehat{\X}_1 \sim_1 \widehat{\X}_2.$  Suppose $\widehat{\X}_i,\ i = 1,2$ are not $\widehat{e}_{12}, \widehat{e}_{34}$-switchable and $\widehat{\X}_1$ is $\widehat{e}_{12}, \widehat{e}_{34}$-amenable, if it is split.
Then the following hold:  
\begin{itemize}
\item[a)]  If $\widehat{\X}_i,\ i = 1,2$ are split, then $\widehat{\X}_2$ is $\widehat{e}_{12}, \widehat{e}_{34}$-amenable.  Moreover, for $\widehat{e}_{12}, \widehat{e}_{34}$-perturbations $\widehat{\X}_i'$ of $\widehat{\X}_i,\ i = 1,2$ we have
$\X_1' \sim \X_2'$, $\forall \X_1' \in \fB_{M_\kappa}(\widehat{\X}_1')$ and $\forall \X_2' \in \fB_{M_\kappa}(\widehat{\X}_2').$

\item[b)]   If $\widehat{\X}_1$ is split and $\widehat{\X}_2$ is fused, then for the $\widehat{e}_{12}, \widehat{e}_{34}$-perturbations $\widehat{\X}_1'$ of $\widehat{\X}_1$, we have $\X_1' \sim \X_2,$ $\forall \X_1' \in \fB_{M_\kappa}(\widehat{\X}_1')$ and $\forall \X_2 \in \fB_{M_\kappa}(\widehat{\X}_2).$

\item[c)]  If $\widehat{\X}_1$ is fused and $\widehat{\X}_2$ is split, then $\widehat{\X}_2$ is $\widehat{e}_{12}, \widehat{e}_{34}$-amenable.  Moreover, for the $\widehat{e}_{12}, \widehat{e}_{34}$-perturbations $\widehat{\X}_2'$ of $\widehat{\X}_2,$ we have $\X_1 \sim \X_2',$ $\forall \X_1 \in \fB_{M_\kappa}(\widehat{\X}_1)$ and $\forall \X_2' \in \fB_{M_\kappa}(\widehat{\X}_2').$

\item[d)]  If $\widehat{\X}_1$ is fused and $\widehat{\X}_2$ is fused, then $\X_1 \sim \X_2,$ $\forall \X_1 \in \fB_{M_\kappa}(\widehat{\X}_1)$ and $\forall \X_2 \in \fB_{M_\kappa}(\widehat{\X}_2).$
\end{itemize}
\label{pro-extend3}
\end{proposition}

We shall assume that for some $i^* < j^*$, $\widehat{\X}_2$ is obtained from $\widehat{\X}_1$ by exchanging $\widehat{e} \in \widehat{X}_{1i^*}$ with $\widehat{f} \in \widehat{X}_{2j^*};$  that is, $\widehat{X}_{2i^*} = \widehat{X}_{1i^*} - \widehat{e} + \widehat{f}$ and  $\widehat{X}_{2j^*} = \widehat{X}_{1j^*} - \widehat{f} + \widehat{e}$.

\begin{proof}{ a)}  Suppose $\widehat{\X}_i,\ i = 1,2$ are split.  Given that $\widehat{\X}_2$ is split and not $\widehat{e}_{12}, \widehat{e}_{34}$-switchable, it follows by Proposition \ref{pro-extend2} a) that $\widehat{\X}_2$ is also $\widehat{e}_{12}, \widehat{e}_{34}$-amenable.
For $i = 1,2,$ let $\widehat{\X}_i'$ be a $\widehat{e}_{12},\widehat{e}_{34}$-perturbation of $\widehat{\X}_i$.  It now follows by Lemma \ref{lem-noninbases}, that $\X_1' \sim \X_2'$, $\forall \X_1' \in \fB_{M_\kappa}(\widehat{\X}_1')$ and $\forall \X_2' \in \fB_{M_\kappa}(\widehat{\X}_2').$
\end{proof}

\begin{proof}{b)}  Suppose that $\widehat{\X}_1$ is split and $\widehat{\X}_2$ is fused.  We may assume that $i^* = 1, \ j^* = 2,$ $\widehat{e}_{23} \in \widehat{X}_{11}$ and $\widehat{e}_{24} \in \widehat{X}_{12}$.  Furthermore, we may assume $\widehat{e} = \widehat{e}_{23}$ and $\{ \widehat{e}_{23}, \widehat{e}_{24} \} \subset \widehat{X}_{22}.$  Let $\X_1' = (X_{11}', \dots ,X_{1\kappa}') \in \fB_{M_\kappa}(\widehat{\X}_1')$ and let $\X_2 = (X_{11}, \dots ,X_{2\kappa}) \in \fB_{M_\kappa}(\widehat{\X}_2).$ Then we see that $\{ e_{2j-1}, e_{2j} \} \subset X_{1j}',\ j = 1,2.$  Furthermore, we have $\{ e_1, \widehat{f} \}  \subset X_{21}$ and $\{ e_2, e_3 , e_4 \} \subset X_{22}$. Thus $X_{21} = X_{11} - e_2 + \widehat{f}$ and $X_{22} = X_{12} - \widehat{f} + e_2$.  It follows that $\X_1' \sim_1 \X_2.$

\end{proof}

\begin{proof}{c)}  Suppose $\widehat{\X}_1$ is fused and $\widehat{\X}_2$ is split.  Given that $\widehat{\X}_2$ is not $\widehat{e}_{12}, \widehat{e}_{34}$-switchable, it follows from Proposition \ref{pro-extend2} b) that $\widehat{\X}_2$ is $\widehat{e}_{12}, \widehat{e}_{34}$-amenable.  Let $\widehat{\X}_2'$ be a $\widehat{e}_{12}, \widehat{e}_{34}$-perturbation of $\widehat{\X}_2.$  One can now use very similar arguments as in the proof of b) to show that $\X_1 \sim \X_2',$ $\forall \X_1 \in \fB_{M_\kappa}(\widehat{\X}_1)$ and $\forall \X_2' \in \fB_{M_\kappa}(\widehat{\X}_2').$

\end{proof}

\begin{proof}{d)}  Suppose that $\widehat{\X}_i,\ i = 1,2$ are fused. Then $\widehat{X}_i,\ i = 1,2$ are seen to be compatible, non-incidental sequences of bases.  Let $\X_i = (X_{i1}, \dots ,X_{i\kappa}) \in \fB_{M_\kappa}(\widehat{\X}_i),\ i = 1,2$.  Now it follows by Lemma \ref{lem-noninbases} that $\X_1 \sim \X_2.$
%
\end{proof}

\section{The Proof of Theorem \ref{the-inductivestep} a): Part 3}\label{sec-mainproof3}

Suppose that for some $i$, $\widehat{\D}_j',\ j = 1, \dots ,i$ are not $\widehat{e}_{12}, \widehat{e}_{34}$-switchable, but $\widehat{\D}_{i+1}'$ is $\widehat{e}_{12}, \widehat{e}_{34}$-switchable.
Let $\widehat{\D}_j'',\ j = 1, \dots ,i$ be sequences defined as in Section \ref{sec-mainproof2}.  Let $\widehat{\D}_{i+1}''$ be an $\widehat{e}_{12}, \widehat{e}_{34}$-switch of $\widehat{\D}_{i+1}'$ and let $\D_{i+1}'' \in \fB_{M_\kappa}(\widehat{\D}_{i+1}'').$
The main theorem in this section is the following:

\begin{theorem}
Assume $\kappa \ge 2$  and Theorem \ref{the-kappa31} holds for all $k < \kappa$ and Theorem \ref{the-main2} holds for all $k<\kappa$, if $\kappa \ge 3.$  
Suppose that for some $i$, $\widehat{\D}_j',\ j = 1, \dots ,i$ are not $\widehat{e}_{12}, \widehat{e}_{34}$-switchable, but $\widehat{\D}_{i+1}'$ is $\widehat{e}_{12}, \widehat{e}_{34}$-switchable.  Then $\D_1'' \sim \D_2'' \sim \cdots \sim \D_{i+1}''.$\label{the-preproextend4}
\end{theorem}

To prove the above theorem, we note that by Theorem \ref{the-preproextend3}, $\D_1'' \sim \D_2'' \sim \cdots \sim \D_{i}''.$  Thus it remains to show that $\D_i'' \sim \D_{i+1}.$  This will follow from the next proposition.

\sms
\noindent{\bf Note}  We may assume that $|E_{G_\kappa}| =m = 2\kappa.$  We refer the reader to the note in Section \ref{subsec-Dj''}. 
\sms

\begin{proposition}
Assume $\kappa \ge 2$  and Theorem \ref{the-kappa31} holds for all $k < \kappa$ and Theorem \ref{the-main2} holds for all $k<\kappa$, if $\kappa \ge 3.$ Let $\widehat{\X}_i= (\widehat{X}_{i1}, \dots ,\widehat{X}_{i\kappa}), \ i = 1,2$ be sequences of bases in $\widehat{M}_\kappa$ where $\widehat{\X}_i$ is compatible with $\widehat{\B}_i',\ i =1,2,$ $\widehat{\X}_1 \sim_1 \widehat{\X}_2,$
$\widehat{\X}_1$ is not $\widehat{e}_{12}, \widehat{e}_{34}$-switchable and $\widehat{\X}_2$ is $\widehat{e}_{12}, \widehat{e}_{34}$-switchable.
\begin{itemize}
\item Assume that if $\widehat{\X}_1$ is split, then it is $\widehat{e}_{12}, \widehat{e}_{34}$-amenable, and we let $\widehat{\X}_1' = (\widehat{X}_{11}', \dots ,\widehat{X}_{1\kappa}')$ be a $\widehat{e}_{12}, \widehat{e}_{34}$-perturbation of $\widehat{\X}_1$ in such a case.  
\item If $\widehat{\X}_1$ is fused, we define $\widehat{\X}_1' := \widehat{\X}_1.$
\item Let $\widehat{\X}_2' = (\widehat{X}_{21}', \dots ,\widehat{X}_{2\kappa}')$ be a $\widehat{e}_{12}, \widehat{e}_{34}$-switch of $\widehat{\X}_2.$
\item For $i = 1,2,$ let $\X_i' = (X_{i1}', \dots ,X_{i\kappa}') \in \fB_{M_\kappa}(\widehat{\X}_i')$.
\end{itemize}
Then $\X_1' \sim \X_2'.$
\label{pro-extend4}
\end{proposition}

\begin{proof}   
We may assume that  $\X_1' \vdash \X_2'$ and for some $i^* < j^*,$ $\widehat{X}_{2i^*} = \widehat{X}_{1i^*} - \widehat{e} + \widehat{f}$ and
 $\widehat{X}_{2j^*} = \widehat{X}_{1j^*} - \widehat{f} + \widehat{e}.$   We observe that since $\widehat{\X}_1$ is not $\widehat{e}_{12}, \widehat{e}_{34}$-switchable, whereas $\widehat{\X}_2$ is $\widehat{e}_{12}, \widehat{e}_{34}$-switchable, either $\widehat{\X}_i,\ i = 1,2$ are both fused, or both are split.
 
 \begin{noname}
 Suppose $\widehat{\X}_i,\ i = 1,2$ are fused.
 \end{noname}
 
 In this case, $\widehat{\X}_i,\ i = 1,2$ are compatible, non-incidental sequences of bases.  Let $\X_i \in \fB_{M_\kappa}(\widehat{\X}_i),\ i= 1,2.$  Thus it follows by Lemma \ref{lem-noninbases} that $\X_1 \sim \X_2.$  However, we also see that $\X_i \sim \X_i',\ i = 1,2.$  Thus $\X_1' \sim \X_2'.$

\begin{noname}
Suppose $\widehat{\X}_i,\ i = 1,2$ are split.
\end{noname}

We may assume that $\widehat{e}_{23} \in \widehat{X}_{11},$ $\widehat{e}_{24} \in \widehat{X}_{12}.$  Since $\widehat{\X}_1$ is not $\widehat{e}_{12}, \widehat{e}_{34}$-switchable and $\widehat{\X}_2$ is $\widehat{e}_{12}, \widehat{e}_{34}$-switchable, it follows that $i^* \le 2$ and $j^* \ge 3.$  Without loss of generality, we may assume that $i^* =2$ (i.e. $\widehat{e} \in \widehat{X}_{12}$) and $j^* = 3$ (i.e. $\widehat{f} \in \widehat{X}_{13}$).  If $\kappa \ge 4,$ then $\widehat{X}_{1\kappa}' = \widehat{X}_{2\kappa}',$ and (given $\X_1' \vdash \X_2'$)  $\X_1' \bumpeq \X_2'.$  Thus $\X_1' \sim \X_2'$, if $\kappa \ge 4.$  Therefore, we may assume that $\kappa =3,$ $m = 6,$ and $\displaystyle{\cup_{j=1}^3\widehat{X}_{1j} \cap \widehat{E} = \{ \widehat{e}_{23}, \widehat{e}_{24}, \widehat{e}_{56} \}.}$  Let $I_1 = \{ 1,2,3,4 \},\ I_2 = \{ 1,2,5,6 \},$ and $I_3 = \{ 3,4,5,6 \}.$
Since $\widehat{\X}_1$ is $\widehat{e}_{12}, \widehat{e}_{34}$-amenable (and not $\widehat{e}_{12}, \widehat{e}_{34}$-switchable) it follows that $\widehat{F}_{I_1}(\widehat{X}_{11}) = \{ \widehat{e}_{12}, \widehat{e}_{34}, \widehat{e}_{23}, \widehat{e}_{14} \}$ and $\widehat{F}_{I_1}(\widehat{X}_{12}) = \{ \widehat{e}_{12}, \widehat{e}_{34}, \widehat{e}_{13}, \widehat{e}_{24} \}.$  We may assume that $\widehat{e}_{12} \in \widehat{X}_{11}'$ and $\widehat{e}_{34} \in \widehat{X}_{12}'.$  Since $\widehat{X}_{11} = \widehat{X}_{21},$ it follows that $\widehat{F}_{I_1}(\widehat{X}_{21}) = \widehat{F}_{I_1}(\widehat{X}_{11}) = \{ \widehat{e}_{12}, \widehat{e}_{34}, \widehat{e}_{23}, \widehat{e}_{14} \}.$

\begin{nonamenonamenonameb}
Suppose that $\widehat{e} = \widehat{e}_{24}.$
\end{nonamenonamenonameb}

   Then $\widehat{e}_{24} \in \widehat{X}_{23}$ and hence $(\widehat{X}_{21}, \widehat{X}_{23})$ is $\widehat{e}_{12}, \widehat{e}_{34}$-switchable. We may assume that $\widehat{e}_{14} \in \widehat{X}_{21}'$ and $\widehat{e}_{23} \in \widehat{X}_{23}'.$  Suppose $\widehat{e}_{12} \in \widehat{F}_{I_1}(\widehat{X}_{23}).$  Let 
$\widehat{\X}_2'' = (\widehat{X}_{21}'', \widehat{X}_{22}'', \widehat{X}_{23}''),$ where 
$\widehat{X}_{21}'' = \widehat{X}_{21} - \widehat{e}_{23} + \widehat{e}_{34},\ \widehat{X}_{22}'' = \widehat{X}_{22},$ and $\widehat{X}_{23}'' = \widehat{X}_{23} - \widehat{e}_{24} + \widehat{e}_{12}.$  Let $\X_2'' = (X_{21}'', X_{22}'', X_{23}'') \in \fB_{M_\kappa}(\widehat{\X}_2''),$ where $\X_2' \vdash \X_2''.$  Since $\widehat{X}_{22}'' = \widehat{X}_{22} = \widehat{X}_{22}',$ it follows that $\X_2' \bumpeq \X_2''$ and hence $\X_2' \sim \X_2''$.  Let $\widehat{\X}_1'' = (\widehat{X}_{11}'', \widehat{X}_{12}'', \widehat{X}_{13}'')$, where $\widehat{X}_{11}'' = \widehat{X}_{11}' - \widehat{e}_{12} + \widehat{e}_{34}, \ \widehat{X}_{12}'' = \widehat{X}_{12}' - \widehat{e}_{34} + \widehat{e}_{12},$ and $\widehat{X}_{13}'' = \widehat{X}_{13}'.$  Let 
$\X_1'' = (X_{11}'', X_{12}'', X_{13}'') \in \fB_{M_\kappa}(\widehat{\X}_1''),$ where we may assume $\X_1' \vdash \X_1''$ and $X_{11}'' = X_{21}''$.  Then $\X_1'' \bumpeq \X_2''.$  However, we also see that since $\X_1' \bumpeq \X_1''$ since $\widehat{X}_{13}' = \widehat{X}_{13}''$ (and $\X_1' \vdash \X_1''$). Thus it follows that $\X_1' \sim \X_1'' \sim \X_2'' \sim \X_2'.$

From the above, we may assume that $\widehat{e}_{12} \not\in \widehat{F}_{I_1}(\widehat{X}_{23}).$ Arguing in a similar fashion to the above, we may also assume $\widehat{e}_{34} \not\in \widehat{F}_{I_1}(\widehat{X}_{23}).$  From these assumption it follows that $\widehat{F}_{I_1}(\widehat{X}_{23}) = \{ \widehat{e}_{14}, \widehat{e}_{13}, \widehat{e}_{23}, \widehat{e}_{24} \}.$  
Moreover, we may assume that $\widehat{X}_{21}' = \widehat{X}_{21} - \widehat{e}_{23} + \widehat{e}_{14},$ $\widehat{X}_{23}' = \widehat{X}_{23} - \widehat{e}_{24} + \widehat{e}_{23}$, and 
hence $\widehat{F}_{I_1}(\widehat{X}_{23}') = \widehat{F}_{I_1}(\widehat{X}_{23}).$

\begin{nonamenonamenonamebc}
Suppose $\widehat{e}_{56} \in \widehat{X}_{11}'.$
\end{nonamenonamenonamebc}

We have $\{ \widehat{e}_{12}, \widehat{e}_{56} \} \subset \widehat{X}_{11}'$, $\widehat{e}_{34} \in \widehat{X}_{12}'$, and $\widehat{X}_{13}' \cap \widehat{E} = \emptyset.$  Thus $\{ e_3, e_4 \} \subset X_{12}'$ and $X_{13}'$ is $v$-anchored.  We also have $\{ \widehat{e}_{14}, \widehat{e}_{56} \} \subset \widehat{X}_{21}'$, $\widehat{e}_{23} \in \widehat{X}_{23}'$, and $\widehat{X}_{22}' \cap \widehat{E} = \emptyset.$  Thus $\{ e_2, e_3 \} \subset X_{23}'$ and $X_{22}'$ is $v$-anchored.  We see that $X_{13}' \cap E_{G_\kappa}(v) = \{ e_i \}$, for some $i \in \{ 1,2,5,6 \}.$  
Suppose $i \not\in \{ 5,6 \}.$  Since $\widehat{e}_{56} \in \widehat{X}_{11}',$ it follows by Lemma \ref{lem-Mhat1} that either $\widehat{X}_{11}' - \widehat{e}_{56} + \widehat{e}_{i5} \in \B(\widehat{M}_\kappa)$ or $\widehat{X}_{11}' - \widehat{e}_{56} + \widehat{e}_{i6} \in \B(\widehat{M}_\kappa)$.
Thus either $X_{11}' - e_6+ e_i \in \B(M_\kappa)$ or  $X_{11}' - e_5+ e_i \in \B(M_\kappa).$  If the former holds, then, by a symmetric exchange between $X_{11}'$ and $X_{13}'$, we can exchange $e_i$ with $e_6.$  Otherwise, if the latter holds, we can exchange $e_i$ with $e_5$.  Because of this, we may assume that $i \in \{ 5, 6 \}.$  We shall treat only the case $i=5$ (i.e. $e_5 \in X_{13}'$); the case $i=6$ is similar and we leave the details to the reader.

Since $\widehat{e}_{34} \in \widehat{X}_{12}',$ it follows by Lemma \ref{lem-Mhat1} that either $\widehat{X}_{12}' - \widehat{e}_{34} + \widehat{e}_{35} \in \B(\widehat{M}_\kappa)$ or $\widehat{X}_{12}' - \widehat{e}_{34} + \widehat{e}_{45} \in \B(\widehat{M}_\kappa).$
Thus either $X_{12}' - e_4 + e_5 \in \B(M_\kappa)$ or $X_{12}' - e_3 + e_5 \in \B(M_\kappa).$  We may assume that the former holds -- the case where the latter holds can be dealt with similarly.  Let $\X_1'' = ( X_{11}'', X_{12}'', X_{13}''),$ where $X_{11}'' = X_{11}',\ X_{12}'' = X_{12}' - e_4 + e_5$, and $X_{13}'' = X_{13}' - e_5 + e_4.$  

 Given $\widehat{e}_{12} \in \widehat{X}_{11}',$ it follows by Lemma \ref{lem-Mhat1}, that either $\widehat{X}_{11}' - \widehat{e}_{12} + \widehat{e}_{14} \in \B(\widehat{M}_\kappa)$ or $\widehat{X}_{11}' - \widehat{e}_{12} + \widehat{e}_{24} \in \B(\widehat{M}_\kappa)$.  The latter cannot hold since $\widehat{e}_{24} \not\in \widehat{F}_{I_1}(\widehat{X}_{11}').$ Thus the former holds.
 Let $\widehat{\X}_1^{(3)} = (\widehat{X}_{11}^{(3)}, \widehat{X}_{12}^{(3)}, \widehat{X}_{13}^{(3)}),$ where $\widehat{X}_{11}^{(3)} = \widehat{X}_{11}' - \widehat{e}_{12} + \widehat{e}_{14},$ $\widehat{X}_{12}^{(3)} = \widehat{X}_{12}' - \widehat{e}_{34} + \widehat{e}_{35}$ and $\widehat{X}_{13}^{(3)} = \widehat{X}_{13}'.$  Let $\X_1^{(3)} = (X_{11}^{(3)}, X_{12}^{(3)}, X_{13}^{(3)}) \in \fB_{M_\kappa}(\widehat{\X}_1^{(3)}),$ where $\X_2' \vdash \X_1^{(3)}.$  We observe that $X_{12}^{(3)} = X_{12}''$ and $\widehat{X}_{11}^{(3)} = \widehat{X}_{21}'.$  Thus we have $\X_1'' \bumpeq \X_1^{(3)},$ and $X_{11}^{(3)} = X_{21}',$ since $\X_2' \vdash \X_1^{(3)}.$ Thus $\X_1' \sim_1 \X_1'' \bumpeq \X_1^{(3)} \bumpeq \X_2'.$  Consequently, $\X_1' \sim_1 \X_1'' \sim \X_1^{(3)} \sim \X_2'$ and thus $\X_1' \sim \X_2'.$

\begin{nonamenonamenonamebc}
Suppose $\widehat{e}_{56} \in \widehat{X}_{12}'.$
\end{nonamenonamenonamebc}

We have that $\widehat{e}_{12} \in \widehat{X}_{11}',$ $\{ \widehat{e}_{34}, \widehat{e}_{56} \} \subset \widehat{X}_{12}',$ and $\widehat{X}_{13} \cap \widehat{E} = \emptyset.$
Thus $\{ e_1, e_2 \} \subset X_{11}'$ and $X_{13}'$ is $v$-anchored.  We also have $\widehat{e}_{14} \in \widehat{X}_{21}', \ \widehat{e}_{56} \in \widehat{X}_{22}',$ and $\widehat{e}_{23} \in \widehat{X}_{23}'.$  It follows that $\{ e_1, e_4 \} \subset X_{21}', \ \{ e_5, e_6 \} \subset X_{22}',$ and $\{ e_2, e_3 \} \subset X_{23}'.$  Since $X_{13}'$ is $v$-anchored and $\widehat{e}_{34} \in \widehat{X}_{12}',$ we may assume, similar to the previous case, that $e_3 \in X_{13}'$ or $e_4 \in X_{13}'.$  We shall assume $e_3 \in X_{13}'$ and $\{ e_4, e_5, e_6 \} \subset X_{12}'$; the case where $e_4 \in X_{13}'$ can be dealt with by similar arguments.
Since $\widehat{e}_{12} \in \widehat{X}_{11}',$ Lemma \ref{lem-Mhat1} implies that either $\widehat{X}_{11}' - \widehat{e}_{12} + \widehat{e}_{13} \in \B(\widehat{M}_\kappa),$ or $\widehat{X}_{11}' - \widehat{e}_{12} + \widehat{e}_{23} \in \B(\widehat{M}_\kappa).$  The former cannot happen since $\widehat{e}_{13} \not\in \widehat{F}_{I_1}(\widehat{X}_{11}')$ and thus the latter holds.  It follows that $X_{11}' - e_1 + e_3 \in \B(M_\kappa).$   Let $\X_1'' = (X_{11}'', X_{12}'', X_{13}''),$ where $X_{11}'' = X_{11}' - e_1 + e_3,\ X_{12}'' = X_{12}',$ and $X_{13}'' = X_{13}' - e_3 + e_1.$  Then $\X_1' \sim_1 \X_1''.$   Recalling that $\widehat{F}_{I_1}(\widehat{X}_{23}') = \{ \widehat{e}_{14}, \widehat{e}_{13}, \widehat{e}_{23}, \widehat{e}_{24} \},$ we have that
$\widehat{e}_{14} \in \widehat{F}_{I_1}(\widehat{X}_{23}').$  Thus $X_{13}'' - \widehat{f} + e_4 \in \B(M_\kappa).$  Now $\X_1^{(3)} = (X_{11}^{(3)}, X_{12}^{(3)}, X_{13}^{(3)}),$ where $X_{11}^{(3)} = X_{11}'',$ $X_{12}^{(3)} = X_{12}'' -e_4 + \widehat{f},$ and $X_{13}^{(3)} = X_{13}'' - \widehat{f} + e_4$, is seen to be a sequence of bases.  Furthermore, $\X_1'' \sim_1 \X_1^{(3)}.$  Now $\X_1^{(3)} \bumpeq \X_2'$ since $X_{12}^{(3)} = X_{22}'.$  Thus we have $\X_1' \sim_1 \X_1'' \sim_1 \X_1^{(3)} \sim \X_2'.$

\begin{nonamenonamenonamebc}
Suppose $\widehat{e}_{56} \in \widehat{X}_{13}'.$
\end{nonamenonamenonamebc}

In this case, $\widehat{e}_{(2j-1)(2j)} \in \widehat{X}_{1j}',\ j = 1,2,3.$  Thus $\{ e_1, e_2 \} \subset X_{11}'$, $\{ e_3, e_4 \} \subset X_{12}',$ and $\{ e_5, e_6 \} \subset X_{13}'.$  We also see that $\widehat{e}_{14} \in \widehat{X}_{21}'$ and $\{ \widehat{e}_{23}, \widehat{e}_{56} \} \subset \widehat{X}_{23}'.$  Thus we have $\{ e_1, e_4 \} \subset X_{21}'$ and $X_{22}'$ is $v$-anchored.  Given $\widehat{e}_{23} \in \widehat{X}_{23}',$ it follows by applying Lemma \ref{lem-Mhat1}, that we may assume $e_2 \in X_{22}'$ or $e_3 \in X_{22}'.$
Suppose  $e_2 \in X_{22}'$.  Then $\{ e_3, e_5,e_6 \} \subset X_{23}'.$  Now we have that $X_{11}' = X_{21}' - e_4 + e_2.$  Thus $\X_2'' = (X_{21}'', X_{22}'', X_{23}''),$ where $X_{21}'' = X_{11}', \ X_{22}'' = X_{22}' - e_2 + e_4,$ $X_{23}'' = X_{23}'$ is seen to be a base sequence.  Furthermore, $\X_2'' \bumpeq \X_1',$ since $X_{21}'' = X_{11}'.$  Thus we have $\X_2' \sim_1 \X_2'' \sim \X_1'.$  The case where 
$e_3 \in X_{22}$ can be dealt with in a similar fashion.

\begin{nonamenonamenonameb}
Suppose that $\widehat{e} \ne \widehat{e}_{24}.$
\end{nonamenonamenonameb} 

As before, we have $\widehat{F}_{I_1}(\widehat{X}_{21}) = \widehat{F}_{I_1}(\widehat{X}_{11}) = \{ \widehat{e}_{12}, \widehat{e}_{34}, \widehat{e}_{14}, \widehat{e}_{23} \}$, since $\widehat{X}_{21} = \widehat{X}_{11}.$  Since $(\widehat{X}_{11}', \widehat{X}_{12}')$ is not $\widehat{e}_{12}, \widehat{e}_{34}$-switchable, it follows by Theorem \ref{the-note12e34switch}, that the (unique) circuit in $\widehat{X}_{12}' \cup \{ \widehat{e}_{12} \}$ corresponds to a subgraph of $\widehat{G}_\kappa$ consisting of two vertex-disjoint unbalanced cycles joined by a path, each cycle containing exactly one of the edges $\widehat{e}_{12}$ or $\widehat{e}_{34}.$  Using this, it is straight-forward to show that 
$\{ \widehat{e}_{12}, \widehat{e}_{34} \} \cap \widehat{F}_{I_1}(\widehat{X}_{22}) \ne \emptyset.$

Without loss of generality, we may assume $\widehat{e}_{12} \in \widehat{F}_{I_1}(\widehat{X}_{22}).$ Let $\widehat{\X}_1'' = (\widehat{X}_{11}'', \widehat{X}_{12}'', \widehat{X}_{13}''),$ be the base sequence where $\widehat{X}_{11}'' = \widehat{X}_{11}' - \widehat{e}_{12} + \widehat{e}_{34},$ $\widehat{X}_{12}'' = \widehat{X}_{12}' - \widehat{e}_{34} + \widehat{e}_{12},$ and $\widehat{X}_{13}'' = \widehat{X}_{13}'.$  Let $\X_1'' = (X_{11}'', X_{12}'', X_{13}'') \in \fB_{M_\kappa}(\widehat{\X}_1'')$ where $\X_1' \vdash \X_1''.$  Since $\widehat{X}_{13}'' = \widehat{X}_{13}',$ it follows that $\X_1' \bumpeq \X_1''$ and hence $\X_1' \sim \X_1''.$  Let $\widehat{\X}_2'' = (\widehat{X}_{21}'', \widehat{X}_{22}'', \widehat{X}_{23}''),$ be the base sequence where $\widehat{X}_{21}'' = \widehat{X}_{21} - \widehat{e}_{23} + \widehat{e}_{34},$ $\widehat{X}_{22}'' = \widehat{X}_{22} - \widehat{e}_{24} + \widehat{e}_{12},$ and $\widehat{X}_{23}'' = \widehat{X}_{23}.$  Let $\X_2'' = (X_{21}'', X_{22}'', X_{23}'') \in \fB_{M_\kappa}(\widehat{\X}_1'')$ where $\X_2' \vdash \X_2''.$  Since $\widehat{X}_{23}'' = \widehat{X}_{23} = \widehat{X}_{23}',$ it follows that $\X_2' \bumpeq \X_2''$ and thus $\X_2' \sim \X_2''.$  Moreover, we see that $\widehat{X}_{21}'' = \widehat{X}_{11}''$ and $\widehat{X}_{11}'' \cap \widehat{E} = \widehat{X}_{21}'' \cap \widehat{E} = \{ \widehat{e}_{34} \}.$  Thus $X_{11}'' = X_{22}''$ and consequently, $\X_1'' \bumpeq \X_2''$ and hence $\X_1'' \sim \X_2''.$  Now we have $\X_1' \sim \X_1'' \sim \X_2'' \sim \X_2'.$ 

\end{proof}

\subsection{The proof of Theorem \ref{the-inductivestep} a):  Final Step}\label{sec-mainprooffinal}

Noting that $\widehat{\B}_2'$ is $\widehat{e}_{12}, \widehat{e}_{34}$-switchable, let $i_1$ be the smallest index $i' \in \{ 2, \dots ,p \}$ for which $\widehat{\D}_{i'}'$ is $\widehat{e}_{12}, \widehat{e}_{34}$- switchable.
Define $\widehat{\D}_i'',\ i = 1, \dots ,i_1-1$ as before and let $\widehat{\D}_{i_1}''$ be a $\widehat{e}_{12},\widehat{e}_{34}$-switch of $\widehat{\D}_{i_1}'.$  Let $\D_j'' \in \fB_{M_\kappa}(\widehat{\D}_j''),\ j = 1, \dots i_1$, where $\D_1'' = \B_1.$   Then it follows by Theorem \ref{the-preproextend4} that $\B_1  = \D_1'' \sim \D_{i_1}''.$    
 We can apply the same reasoning by reversing the roles of $\B_1$ and $\B_2.$   Let $i_2$ be the largest value of $i \in \{ 1, \dots ,p \}$ for which $\widehat{\D}_i'$ is $\widehat{e}_{14}, \widehat{e}_{23}$-switchable.  Let $\widehat{\D}_{i_2}''$ be a $\widehat{e}_{14}, \widehat{e}_{23}$-switch of $\widehat{\D}_{i_2}'$ and let $\D_{i_2}'' \in \fB_{M_\kappa}(\widehat{\D}_{i_2}'').$  Then $\B_2 \sim \D_{i_2}''$ (using Theorem \ref{the-preproextend4}).  Now it is seen that either 
$\fM_{\D_{i_1}''} = \fM_{\B_2},$ or $\fM_{\D_{i_1}''} = \fM_{\D_{i_2}''}.$  In the former case, we have $\B_1 \sim \D_{i_1}'' \sim \B_2$, and in the latter case, $\B_1 \sim \D_{i_1}'' \sim \D_{i_2}'' \sim \B_2.$   Thus $\B_1 \sim \B_2$ and this completes the proof of Theorem \ref{the-inductivestep} a).

\section{The proof of Theorem \ref{the-inductivestep} b)}\label{sec-kappa31proof}

Theorem \ref{the-kappa31} is clearly true when $\kappa = 1$. We leave it to the reader to verify the theorem when $n=r(M)=2$.  We shall assume that $\kappa \ge 2,$ $n \ge 3$ and Theorem \ref{the-inductivestep} b) holds when $r(M) < n$.  Furthermore, we may assume that Theorem \ref{the-main2} holds for $k = \kappa$, as this follows from Theorem \ref{the-inductivestep} a), proven in the previous sections.  If $h_1 = h_2$, then 
$\B_i,\ i = 1,2$ are seen to be compatible sequences of bases, and thus $\B_1 \sim \B_2$, since we given that Theorem \ref{the-main2} holds $k = \kappa.$  It then follows that $\B_1^+ \sim \B_2^+.$  Thus we shall assume that $h_1 \ne h_2.$

We note that since the bases in $\B_i,\ i = 1,2$ are pairwise disjoint,  it is unnecessary to define $G_\kappa$ and $M_\kappa$  as we did in Section \ref{sec-reductionsI} (since $G = G_\kappa$ and $M = M_\kappa$ here).
Since each base has $n$ edges, the graph $G-h_1$ has $n\kappa$ edges in all and total degree $2n\kappa.$  Furthermore, since $d_{G}(u) = 2\kappa +2,$ it follows by averaging that there is a vertex $v \in V(G)\backslash \{ u \}$ where $d_{G-h_1}(v) \le 2\kappa-1.$  If $d_{G}(v) \le 2\kappa -1,$ or $G$ has a loop at $v,$ then the ensuing arguments will actually be simpler.  For this reason,  we will assume 

\begin{itemize}\item $|E_{G-h_1} (v) | = d_{G-h_1}(v) = 2\kappa-1$, $h_1 \in E_G(v),$ and there are no loops at $v$.\end{itemize}       

\subsection{The $\widehat{u}$-extended sequences $\widehat{\B}_i^,\ i = 1,2$}\label{subsec-widehatuextend}

By Lemma \ref{lemma2}, we may assume that for $i=1,2,$ $\B_i$ is a $v$-reduced sequence of bases.  
Let $\widehat{\Omega} = (\widehat{G}, \widehat{\C})$ be the $v$-deleted biased graph and we let $\widehat{M} = M(\widehat{\Omega}).$   Let $\widehat{u} \in V(\widehat{G})$ be the vertex in $\widehat{G}$ corresponding to $u.$  For $i= 1,2$ we shall define a $\widehat{u}$-extended base sequence $\widehat{\B}_i^+ = (\widehat{\B}_i,  \widehat{h}_i )$ of $\widehat{M}$ where $\widehat{\B}_i = (\widehat{B}_{i1}, \dots , \widehat{B}_{i\kappa})$ and $\widehat{h}_i \in E_{\widehat{G}}(\widehat{u}).$   For $i= 1,2$ and $j = 1, \dots ,\kappa,$ if $B_{ij} \cap E_{G}(v) = \{ e_{k}, e_l \},$ then let $\widehat{B}_{ij}= B_{ij} - e_k - e_l + \widehat{e}_{kl};$ otherwise, if $B_{ij} \cap E_{G}(v) = \{ e_{k} \},$ then let $\widehat{B}_{ij}= B_{ij} \backslash \{ e_k \}.$ 
For $i=1,2,$ if $h_i \not\in E_G(v),$ then let $\widehat{h}_i = h_i.$   Otherwise, if $h_i \in E_{G}(v)$ (which is the case for $h_1$), then we define $\widehat{h}_i$ in the following manner:  Assume $h_i = e_{i^*}$, for some $i^*$.
Since $d_{G-h_i}(v) = 2\kappa -1,$ there exists  exactly one edge $e_{j^*}\in E_{G}(v)$ which belongs to a $v$-anchored base in $\B_i.$  We define
$\widehat{h}_i := \widehat{e}_{i^*j^*}.$   An important observation, to be used later on, is that because $h_1 \in E_{G}(v),$ the vertex $\widehat{u}$ is an endvertex of some edge $\widehat{e}_{ij}.$ 

\subsection{Amenability and switchability of $\widehat{u}$-extended sequences}

Recall that $\widehat{E} = \{ \widehat{e}_{ij}\ \big| \ i,j \in \{ 1, \dots ,m \} \}.$  Let $I \subseteq \{ 1, \dots ,m \}$ where $|I| \ge 2,$ and let $\widehat{H} = \widehat{H}_I.$   Let $\widehat{h} \in E_{\widehat{G}}(\widehat{u}).$  We define $\widehat{F}_I(\widehat{h}) = E_{\widehat{H}}(\widehat{u}),$ if $\widehat{h} \in \widehat{E};$ otherwise, $\widehat{F}_I(\widehat{h}) := \emptyset.$   Suppose $|I| =4$ and $\widehat{e}, \widehat{f}$ are non-incident edges in $E(\widehat{H}).$    Let $\widehat{B} \in \B(\widehat{M})$ and let $\widehat{h} \in E_{\widehat{G}}(\widehat{u})\backslash \widehat{B}.$   We say that the pair $(\widehat{B}, \widehat{h})$ is $\widehat{e}, \widehat{f}$-amenable if $$\{ (\widehat{e}, \widehat{f}), (\widehat{f}, \widehat{e}) \} \cap \left(\widehat{F}_I(\widehat{h}) \times \widehat{F}_I(\widehat{B}) \right) \ne \emptyset.$$  
Furthermore, we say that $(\widehat{B}, \widehat{h})$ is $\widehat{e}, \widehat{f}$-switchable if for some non-incident $\{ \widehat{e}', \widehat{f}' \} \subset E(\widehat{H})\backslash \{ \widehat{e}, \widehat{f} \},$  $(\widehat{B}, \widehat{h})$ is $\widehat{e}', \widehat{f}'$-amenable.  Let $\widehat{\B}^+ = (\widehat{\B}, \widehat{h} )$ be a $\widehat{u}$-extended base sequence where $\widehat{\B} = (\widehat{B}_1, \dots,\widehat{B}_\kappa).$   We say that $\widehat{\B}^+$ is $\widehat{e}, \widehat{f}$-amenable (resp. $\widehat{e}, \widehat{f}$-switchable) if either $\widehat{\B}$ is $\widehat{e}, \widehat{f}$-amenable (resp. $\widehat{e}, \widehat{f}$-switchable) or $(\widehat{B}_i,\widehat{h})$ is $\widehat{e}, \widehat{f}$-amenable (resp. $\widehat{e}, \widehat{f}$-switchable) for some $i \in \{ 1,\dots ,\kappa \}.$   If $\widehat{h} \in \widehat{E}$ and $\widehat{F}_I(\widehat{B}) \ne \emptyset,$ then it is straightforward to show (using Lemma \ref{lem-Mhat2}) that $(\widehat{B}, \widehat{h})$ is $\widehat{e}, \widehat{f}$-switchable iff $\widehat{F}_I(\widehat{B}) \ne \widehat{F}_I(\widehat{h}) = E_{\widehat{H}}(\widehat{u}).$

\subsection{The pull-back of a $\widehat{u}$-extended sequence and the matching graph $\fM_{\B^+}$}

Let $\widehat{\B}^+ = (\widehat{\B},  \widehat{h} )$ be a $\widehat{u}$-extended base sequence of $\widehat{M}.$ 
We define the pull-back $\widehat{\B}^+$ to be the set 
$\fB_M(\widehat{\B}^+)$ of $u$-extended base sequences in $M$ where
$$\fB_M(\widehat{\B}^+) := \left\{ (\B, h ) \ \big| \ \B\in \fB_M(\widehat{\B}),\ h \in \fE_M(\widehat{h}) \cap E_{G}(u) \right\} .$$

\subsubsection{The matching graph $\fM_{\B^+}$}

For a $u$-extended base sequence $\B^+ = (\B, h),$ where $\B = (B_1, \dots ,B_{\kappa}),$ we define a matching graph $\fM_{\B^+}$ in a similar fashion as before, with a slight difference.  If $h\not\in E_{G}(v),$ then we define $\fM_{\B^+} := \fM_{\B}.$  Suppose
$h\in E_{G}(v).$  We may assume that $h = e_{i^*}$, for some $i^*$.  Since $d_{G}(v) = 2\kappa,$ there is exactly one integer $j$ where $B_j$ has exactly one edge incident to $v$, and such an edge is not a loop (since we are assuming that there are no loops at $v$).  Assume $B_j \cap E_{G_\kappa}(v) = \{ e_{j^*} \}.$  Then we define $\sx_{i^*}\sx_{j^*}$ to be an edge in $\fM_{\B^+}$. Furthermore, for all $i$ where $i\ne j,$ if $B_i \cap E_G(v) = \{ e_k, e_l \},$ then we define $\sx_k\sx_l$ to be an edge in $\fM_{\B^+}.$  
 
 For $i = 1,2$, let $\fM_i = \fM_{\B_i^+}.$  We note that $|\fM_1| = |\fM_2| = \kappa.$   By modifying the proof of Theorem \ref{theorem3.2}, one can show that there are extended base sequences $\B_i'^+,\ i= 1,2$ where $\B_i'^+\sim \B_i^+,\ i = 1,2$ and either $\fM_{\B_1'^+} = \fM_{\B_2'^+},$ or $\fM_{\B_1'^+} \ne \fM_{\B_2'^+}$  and  $\fM_{\B_1'^+} \triangle \fM_{\B_2'^+}$ is a $4$-cycle.  We leave the details to the reader.  As such, we may assume that either $\fM_1 = \fM_2,$ or $\fM_1 \ne \fM_2$ and $\fM_1 \triangle \fM_2$ is a $4$-cycle.
 
 \subsection{Induced and linked extended sequences}

Let $\widehat{\B}^+ = (\widehat{\B}, \widehat{h})$ and $\widehat{\B}'^+ = (\widehat{\B'}, \widehat{h}')$ be extended base sequences in $\widehat{M}.$  Let $\B^+ = (\B, h) \in \fB_M(\widehat{\B}^+)$ and
$\B'^+ = (\B', h') \in \fB_M(\widehat{\B}'^+).$  If i) $\B \vdash \B'$ and ii) $h' = h$ if $\widehat{h} = \widehat{h}',$ then we write $\B^+ \vdash \B'^+$ and we say that $\B'^+$ is {\bf induced} by $\B^+.$  We say that $\B^+$ and $\B'^+$ are {\bf linked} if either $h=h'$ or $B = B'$ for some $B\in \B$ and $B' \in \B'.$  We write $\B^+ \bumpeq \B'^+.$  If $h = h',$ then $\B$ and $\B'$ are compatible.  Thus by assumption, $\B \sim \B',$ and consequently $\B^+ \sim \B'^+.$  On the other hand, if $B = B'$ for some $B\in \B$ and $B' \in \B',$ then deleting $B$ from $\B$ and $B'$ from $\B'$ results in extended sequences $\D^+$ and $\D'^+$ respectively, which by assumption satisfy $\D^+ \sim \D'^+.$  It now follows that $\B^+ \sim \B'^+.$  To summarize, if $\B^+ \bumpeq \B'^+,$ then $\B^+ \sim \B'^+.$   

\section{The Proof of Theorem \ref{the-inductivestep} b): The case $\fM_1 = \fM_2$}
  
In this section, we shall assume that $\fM_1 = \fM_2$.  
We may assume that $E(\fM_1) = E(\fM_2) = \{ \sx_{2j-1}\sx_{2j}\ \big| \ j =1, \dots ,\kappa \}.$  Thus $\widehat{e}_{(2j-1)(2j)},\ j= 1, \dots ,\kappa $ are exactly the edges in 
$\widehat{E} \cap \left( \widehat{h}_i \cup \bigcup_{j=1}^\kappa \widehat{B}_{ij} \right).$     Since $r(\widehat{M}) < r(M),$  it follows by assumption that $\widehat{\B}_1^+ \sim \widehat{\B}_2^+.$  Thus there are $\widehat{u}$-extended bases sequences $\widehat{\D}_i^+ = (\widehat{\D}_i, \widehat{g}_i ),\ i = 1, \dots ,p$ in $\widehat{M}$ where $\widehat{\B}_1^+ = \widehat{\D}_1^+ \sim_1 \widehat{\D}_2^+ \sim_1 \cdots \sim_1 \ \widehat{\D}_p^+ = \widehat{\B}_2^+.$   Let $\widehat{\D}_i^+ = (\widehat{\D}_i, \widehat{g}_i),\ i = 1, \dots ,p$ where $\widehat{\D}_i = (\widehat{D}_{i1}, \dots ,\widehat{D}_{i\kappa}).$ We shall show that one can choose $\widehat{u}$-extended sequences $\D_i^+ \in \fB_M(\widehat{\D}_i^+),\ i = 1, \dots ,p$ such that $\B_1^+ = \D_1^+ \sim \D_2^+ \sim \cdots \sim \D_p^+ = \B_2^+.$   

Let  $\D_1^+ = \B_1^+$ and choose $\D_i^+\in \fB_M(\widehat{\D}_i^+),\ i = 2, \dots ,p$ iteratively so that, given $\D_{i-1}^+$ has be chosen, we choose $\D_i^+$ so that $\D_{i-1}^+ \vdash \D_i^+.$ 
Let $\D_j^+ = (\D_j, g_i),\ j = 1, \dots ,p$, where $\D_j = (D_{j1}, \dots ,D_{j\kappa}).$  Suppose that $\D_1^+ \sim \D_2^+ \sim \cdots \sim \D_i^+.$  Our task is to show that $\D_i^+ \sim \D_{i+1}^+.$  First, if $\widehat{g}_{i+1} = \widehat{g}_i,$ then $g_{i+1} = g_i$ (since $\D_i^+ \vdash \D_{i+1}^+$) and hence $\D_i^+ \sim \D_{i+1}^+.$   Thus we may assume that $\widehat{g}_{i+1} \ne \widehat{g}_i.$  This means that $\widehat{\D}_{i+1}^+$ is obtained from $\widehat{\D}_i^+$ by an ({\bf EB}) exchange where $\widehat{g}_i$ is exchanged with an element $\widehat{e}.$   Without loss of generality, we may assume that $\widehat{e} \in \widehat{D}_{i1}.$  Then $\widehat{D}_{(i+1)1} = \widehat{D}_{i1} - \widehat{e} + \widehat{g}_i$, $\widehat{D}_{(i+1)j} = \widehat{D}_{ij},\ j = 2, \dots ,\kappa,$ and $\widehat{g}_{i+1} = \widehat{e}.$   If $\widehat{D}_{ij} \cap \widehat{E} \ne \emptyset,$ for some $j\in \{ 2, \dots ,\kappa\},$ then $D_{(i+1)j} = D_{ij}$ (since $\D_i^+ \vdash \D_{i+1}^+)$ and hence $\D_i^+ \bumpeq \D_{i+1}^+$. Thus in this case, $\D_i^+ \sim \D_{i+1}^+.$
Because of this, we may assume that $\widehat{D}_{ij} \cap \widehat{E} = \emptyset, \ j=2, \dots,\kappa.$  This in turn implies that the bases $D_{ij},\ j = 2,\dots ,\kappa$ and $D_{(i+1)j},\ j = 2,\dots ,\kappa$ are $v$-anchored.

Let $D_{ij} \cap E_{G}(v)  = \{ e_{k_j} \}, \ j = 2, \dots ,\kappa$ and let $D_{(i+1)j} \cap E_{G}(v)  = \{ e_{l_j} \}, \ j = 2, \dots ,\kappa.$  If $e_{k_j} = e_{l_{j'}},$ for some $j,j',$ then by exchanging $e_{k_j}$ with $e_{k_{j'}}$ between $D_{ij}$ and $D_{ij'}$,  we see that $D_{ij'} - e_{k_{j'}} + e_{k_j} = D_{(i+1)j'}.$  Thus we see that $\D_i^+ \sim \D_{i+1}^+.$  As such, we may assume that $\{ e_{k_j} \ \big| \ j = 2, \dots ,\kappa \} \cap \{ e_{l_j} \ \big| \ j = 2, \dots ,\kappa \} = \emptyset.$  

Suppose first that $\widehat{g}_i, \widehat{g}_{i+1} \not\in \widehat{E}.$  Then $\widehat{e}_{(2j-1)2j} \in \widehat{D}_{i1}\cap \widehat{D}_{(i+1)1},\ j = 1, \dots ,\kappa.$  
Suppose that $\{ e_{2j-1}, e_{2j} \} \cap  D_{(i+1)1} = \emptyset$, for some $j \in \{ 2, \dots ,\kappa \}.$  Then $e_{2j-1}$ and $e_{2j}$ belong to $v$-anchored bases in $\D_{i+1}.$   Also, by Lemma \ref{lem-Mhat1}, either $\widehat{D}_{i1} - \widehat{e}_{(2j-1)2j} + \widehat{e}_{1(2j-1)} \in \B(\widehat{M}_\kappa)$ or $\widehat{D}_{i1} - \widehat{e}_{(2j-1)(2j)} + \widehat{e}_{1(2j)} \in \B(\widehat{M}_\kappa).$  Depending on whether the former or latter holds, we may choose $D_{i1}$ such that either $\{ e_1, e_2, e_{2j-1} \} \subset D_{i1}$ and $e_{2j} \not\in D_{i1},$ or $\{ e_1, e_2, e_{2j} \} \subset D_{i1}$ and $e_{2j-1} \not\in D_{i1}.$  In other words, we may choose $\D_{i}^+$ such that at least one of $e_{2j-1}$ or $e_{2j}$ belongs to an anchored base in $\D_i^+.$   Since both these edges belong to anchored bases in $\D_{i+1}^+,$ we see that $\D_i^+ \sim \D_{i+1}^+.$  Because of this, we may assume that $\{ e_{2j-1}, e_{2j} \} \cap  D_{(i+1)1} \ne \emptyset$, for $j =2, \dots ,\kappa.$  In fact, by symmetry, we may assume that
$\{ e_{2j-1}, e_{2j} \} \cap  D_{(i+1)1} \ne \emptyset$, for $j =1, \dots ,\kappa$ and the same assumption can be made with $D_{i1}$ in place of $D_{(i+1)1}.$  Since $|D_{i1} \cap E_{G}(v)| = \kappa +1,$ we may assume that $\{ e_1, e_2 \} \subset D_{i1}$ and $|D_{i1} \cap \{ e_{2j-1}, e_{2j} \} | =1,\ j = 2, \dots ,\kappa.$  In addition, we may assume that $e_3 \in D_{i1}$ and $e_4 \not\in D_{i1}.$ We have that $\{ e_{2j-1}, e_{2j} \} \subset D_{(i+1)1},$ for some $j \in \{ 1, \dots ,\kappa \}.$  

Suppose $\{ e_1, e_2 \} \subset D_{(i+1)1}.$  Since $\{ e_{k_j} \ \big| \ j = 2, \dots ,\kappa \} \cap \{ e_{l_j} \ \big| \ j = 2, \dots ,\kappa \} = \emptyset$, it follows that $e_4 \in D_{(i+1)1}$ and $e_3 \not\in D_{(i+1)1}.$  Using Lemma \ref{lem-Mhat1}, one can show that either $D_{i1} - e_1 +e_4 \in \B(M)$ or $D_{i1} - e_3 +e_4 \in \B(M).$  If the latter holds, then we could choose $D_{i1}$ such that $\{ e_1, e_2, e_4 \} \subset D_{i1}$ and $e_3 \not\in D_{i1}.$  In this case, $e_3$ belongs to unanchored bases in both $\D_i^+$ and $\D_{i+1}^+$ and hence it follows that $\D_i^+ \sim \D_{i+1}^+$ by our previous argument.  Thus we may assume that $D_{i1} - e_1 +e_4 \in \B(M).$  Using the same argument with $e_1$ in place of $e_2$, we may assume $D_{i1} - e_2 +e_4 \in \B(M).$  Thus we can choose 
$\D_i^+$ so that either $e_1$ or $e_2$ belongs to an unanchored base.  However, by using the same arguments, we can make the same assumptions about $\D_{i+1}^+.$  Thus in this case, we can choose $\D_i^+$ and $\D_{i+1}^+$ so that they have a common unanchored base and consequently, $\D_i^+ \sim \D_{i+1}^+.$  

Suppose instead the $\{ e_1, e_2 \} \not\subset D_{(i+1)1}.$
Without loss of generality, we may assume that $\{ e_3, e_4 \} \subset D_{(i+1)1}.$  Furthermore, we may assume $e_2 \in D_{(i+1)1}$ and $e_1 \not\in D_{(i+1)1}.$  Using Lemma \ref{lem-Mhat1}, one can show that either\\ $D_{i1} - e_1 +e_4 \in \B(M)$ or $D_{i1} - e_2 +e_4 \in \B(M).$  If the former holds, then we could choose $\D_i^+$ so that $e_1$ belongs to a $v$- unanchored base, which is the case for $\D_{i+1}^+.$  Thus $\D_i^+ \sim \D_{i+1}^+$ in this case.  Because of this we may assume $D_{i1} - e_2 +e_4 \in \B(M).$
Using Lemma \ref{lem-Mhat1}, it can be shown that either $D_{(i+1)1} - e_4 + e_1 \in \B(M)$ or $D_{(i+1)1} - e_2 + e_1 \in \B(M).$  If the former holds, then we could choose $\D_{i+1}^+$ so that $e_4$ belongs to an $v$-unanchored base, which is the case for $\D_i^+.$  Thus $\D_i^+ \sim \D_{i+1}^+$ in this case.  Therefore, we may assume that $D_{(i+1)1} - e_2 + e_1 \in \B(M).$  However, since $D_{i1} - e_2 +e_4 \in \B(M)$, we may choose $\D_i^+$ and $\D_{i+1}^+$ so that $e_2$ belongs to an $v$-unanchored base in both.  Thus we see that $\D_i^+ \sim \D_{i+1}^+.$  This concludes that case where $\{ \widehat{g}_i, \widehat{g}_{i+1} \} \cap \widehat{E} = \emptyset.$

Suppose $| \{ \widehat{g}_i, \widehat{g}_{i+1} \} \cap \widehat{E}| = 1$.  Without loss of generality, we may assume $\widehat{g}_i \in \widehat{E}$ and $\widehat{g}_{i+1} \not\in \widehat{E}.$  We may assume that $\widehat{g}_i = \widehat{e}_{34}$ and $g_i = e_4.$  Furthermore, we may assume that $\{ e_1, e_2 \} \subset D_{i1}.$  By construction, $e_3 \not\in D_{i1}$, since $\widehat{e}_{34} \not\in \widehat{D}_{i1}.$  In particular, $e_3$ belongs to a $v$-unanchored base in $\D_i.$  By exchanging edges if necessary, we may assume that $\{ e_1, e_2 \} \subset D_{(i+1)1}.$  Moreover, $|\{ e_3, e_4 \} \cap D_{(i+1)1}| =1$.  We may assume $e_3 \in D_{(i+1)1}$ ; otherwise, if $e_3 \not\in D_{(i+1)1},$ then $e_3$ belongs to a $v$-unanchored bases in $\D_i^+$ and $\D_{i+1}^+$, in which case, $\D_i^+ \sim \D_{i+1}^+.$   Using  Lemma \ref{lem-Mhat1}, it can be shown that either $D_{(i+1)1} - e_1 + e_4 \in \B(M),$ or $D_{(i+1)1} - e_3 + e_4 \in \B(M).$  If the latter occurs, then we can choose $\D_{i+1}^+$ so that $e_3$ belongs to an $v$-unanchored base, which also holds for $\D_i^+.$  Thus we see that $\D_i^+ \sim \D_{i+1}^+.$  As such, we may assume that $D_{(i+1)1} - e_1 + e_4 \in \B(M).$  Similarly, the same assumption can be made if one replaces $e_1$ with $e_2.$  Using Lemma \ref{lem-Mhat1}, we see that either $D_{i1} - e_1 + e_3 \in \B(M)$ or $D_{i1} - e_2 + e_3 \in \B(M).$  Thus we can choose $\D_i^+$ so that at least of $e_1$ or $e_2$ belongs to a $v$-anchored base.  However, we may choose $\D_{i+1}^+$ so that either one of $e_1$ or $e_2$ belongs to a $v$-anchored base.  Thus we see that $\D_i^+$ and $\D_{i+1}^+$ can be chosen so that they have a common $v$-anchored base.  Therefore $\D_i^+ \sim \D_{i+1}^+.$  This concludes the case where $| \{ \widehat{g}_i, \widehat{g}_{i+1} \} \cap \widehat{E}| = 1$.

Lastly, suppose $\widehat{g}_i, \widehat{g}_{i+1} \in \widehat{E}.$  We may assume that $\widehat{g}_i = \widehat{e}_{34}$ and $\widehat{g}_{i+1} = \widehat{e}_{12}.$  Furthermore, we may assume that $g_i = e_4$ and $g_{i+1} = e_1.$  Note that here $g_i = e_4$ and $g_{i+1}= e_1$ are parallel edges.  We also have that $e_3 \not\in D_{i1}$ and $e_2 \not\in D_{(i+1)1}.$  By exchanging edges if necessary, we may choose $\D_i^+$ such that $\{ e_1, e_2\} \subset D_{i1}.$  Using Lemma \ref{lem-Mhat1}, we see that either $D_{i1} - e_1 + e_3 \in \B(M)$ or $D_{i1} - e_2 + e_3 \in \B(M).$
Suppose that latter occurs.  Then we can transform $\D_i^+$ by exchanging the edges $e_2$ and $e_3$, so that $e_2$ belongs to a $v$-anchored base in the resulting sequence of bases.  Given that $e_2$ belongs to an anchored base in $\D_{i+1}^+,$ we see that $\D_i^+ \sim \D_{i+1}^+.$  Thus we may assume that $D_{i1} - e_1 + e_3 \in \B(M)$.  Thus we may exchange $e_1$ and $e_3$ in $\D_i$, obtaining a sequence of bases $\D_i'$ where $e_1$ belongs to an anchored base in $\D_i'.$  Given that $e_1$ and $e_4$ are parallel, we may then perform an ({\bf EB}) exchange, exchanging $e_1$ and $e_4$. 
Thus $\D_i^+ \sim \D_{i+1}^+.$

\section{The Proof of Theorem \ref{the-inductivestep} b): The case $\fM_1 \ne \fM_2$}

In this section, we shall assume that $\fM_1 \ne \fM_2$
We will use a number of concepts and their properties introduced in Section \ref{subsec-notequal}.  As before, we may assume that $\fM_1 \triangle \fM_2$ is a $4$-cycle $\sx_1\sx_2\sx_3\sx_4\sx_1$ where
$\sx_1\sx_2, \sx_3\sx_4 \in E(\fM_1)$ and $\sx_1\sx_4, \sx_2\sx_3 \in E(\fM_2)$.   Our construction of $\widehat{\B}_i^+ = (\widehat{\B}_i, \widehat{h}_i),\ i = 1,2$ is the same as before except that now we shall assume that
\begin{itemize}
\item $\widehat{e}_{(2j-1)(2j)} \in \widehat{B}_{ij} \cup \{ \widehat{h}_i \},\ i = 1,2; \ j = 3, \dots ,\kappa.$
\item $\{ \widehat{e}_{12}, \widehat{e}_{34} \} \subset \widehat{B}_{11} \cup \widehat{B}_{12} \cup \{ \widehat{h}_1 \}.$
\item $\{ \widehat{e}_{14}, \widehat{e}_{23} \} \subset \widehat{B}_{21} \cup \widehat{B}_{22} \cup \{ \widehat{h}_2 \}.$
\end{itemize}

\subsection{The $\widehat{u}$-extended sequences $\widehat{\B}_i'^+ = (\widehat{\B}_i', \widehat{h}_i'),\ i = 1,2$}

Let $I = \{ 1,2,3,4 \}$ and let $\widehat{H} = \widehat{H}_I.$   If $\widehat{\B}_1^+$ is $\widehat{e}_{12}, \widehat{e}_{34}$-switchable and $\widehat{\B}_2^+$ is $\widehat{e}_{14}, \widehat{e}_{23}$-switchable, then one can find sequences of bases $\B_i'^+,\ i = 1,2$ where $\fM_{\B_1'^+} = \fM_{\B_2'^+}$ and $\B_i^+\sim \B_i'^+,\ i = 1,2$.    Thus we may assume that $\widehat{\B}_1^+$ is not $\widehat{e}_{12}, \widehat{e}_{34}$-switchable.  Suppose $\widehat{h}_1 \in \{ \widehat{e}_{12}, \widehat{e}_{34} \}.$  Then $\widehat{F}_I(\widehat{h}_1) = E_{\widehat{H}}(\widehat{u}).$  If $\widehat{h}_1 = \widehat{e}_{12},$ then $\widehat{e}_{34} \in \widehat{B}_{1j}$, for some $j\in \{ 1,2 \}.$  Clearly $\widehat{F}_I(\widehat{B}_{1j}) \ne \widehat{F}_I(\widehat{h}_1)$ and hence $\widehat{\B}_1^+$ is $\widehat{e}_{12}, \widehat{e}_{34}$-switchable,  contradicting our assumptions.  A similar argument applies if $\widehat{h}_1 = \widehat{e}_{34}.$  Thus we have that $\widehat{h}_1 \not\in \{ \widehat{e}_{12}, \widehat{e}_{34} \}.$
By Lemma \ref{lem-Mhat3}, we may assume  for $j=1,2,$ $\widehat{F}_I (\widehat{B}_{1j}) = \{ \widehat{e}_{12}, \widehat{e}_{34}, \widehat{e}_{2(j+2)}, \widehat{e}_{1(5-j)} \}$.  
We may also assume that $\widehat{e}_{23} \in \widehat{B}_{21}$ and $\widehat{e}_{14} \in \widehat{B}_{22} \cup \{ \widehat{h}_2 \}.$  Noting that the extended sequences $\widehat{\B}_i^+,\ i = 1,2$ are not compatible, we need to change them slightly to obtain extended sequences $\widehat{\B}_i'^+,\ i = 1,2$ which are compatible. To do this, we may use a construction similar to that in Section \ref{subsec-notequal}.   Arguing as we did there, we may assume that  $\widehat{\B}_i'^+ = (\widehat{\B}_i', \widehat{h}_i'),\ i = 1,2$ where $\widehat{\B}_i' = (\widehat{B}_{i1}', \dots ,\widehat{B}_{i\kappa}'),\ i = 1,2$ are defined such that:
\begin{itemize}
\item $\widehat{B}_{11}' = \widehat{B}_{11} -\widehat{e}_{12} + \widehat{e}_{23},$ $\widehat{B}_{12}' = \widehat{B}_{12} - \widehat{e}_{34} + \widehat{e}_{24}.$ 
\item $\widehat{B}_{21}' = \widehat{B}_{21},$ $\widehat{B}_{22}' = \left\{ \begin{array}{lr} \widehat{B}_{22} -\widehat{e}_{14} + \widehat{e}_{24} & \mathrm{if}\  \widehat{h}_2 \ne \widehat{e}_{14}\\ \widehat{B}_{22} &\mathrm{if}\  \widehat{h}_2 = \widehat{e}_{14} \end{array}\right\}$.
\item $\widehat{h}_1' = \widehat{h}_1,$ $\widehat{h}_2' = \left\{ \begin{array}{lr} \widehat{h}_2 & \mathrm{if}\  \widehat{h}_2 \ne \widehat{e}_{14}\\ \widehat{e}_{24} &\mathrm{if}\  \widehat{h}_2 = \widehat{e}_{14} \end{array}\right\}$.
\item $\widehat{B}_{ij}' = \widehat{B}_{ij},\ i = 1,2;\ j = 3, \dots, \kappa.$
\end{itemize} 

\subsection{The $\widehat{u}$-extended sequences $\widehat{\D}_i'^+ = (\widehat{\D}_i', \widehat{g}_i'),\ i = 1, \dots ,p$}

By construction, $\widehat{\B}_1'^+ \sim \widehat{\B}_2'^+.$   Thus by assumption, there are $\widehat{u}$-extended base sequences $\widehat{\D}_i'^+ = (\widehat{\D}_i', \widehat{g}_i'),\ i = 1, \dots ,p,$  where $\widehat{\B}_1'^+ = \widehat{\D}_1'^+ \sim_1 \widehat{\D}_2'^+ \sim_1 \cdots \sim_1 \widehat{\D}_p'^+= \widehat{\B}_2'^+.$   We shall find $u$-extended base sequences $\D_i'^+ = (\D_i', g_i') \in \fB_M(\widehat{\D}_i'^+),\ i = 1, \dots ,p$ where $\B_1^+ = \D_1'^+ \sim \D_2'^+ \sim \cdots \sim \D_p'^+= \B_2^+.$
Let $\widehat{\D}_i' = (\widehat{D}_{i1}', \dots ,\widehat{D}_{i\kappa}'),$  $i = 1, \dots ,p.$

\begin{nonamenoname}  Suppose for some $i \in \{ 2,\dots ,p\}$, $\widehat{\D}_{i-1}'^+$ is not $\widehat{e}_{12}, \widehat{e}_{34}$-switchable and $\widehat{\D}_{i}'^+$ is split.  Assuming $\widehat{\D}_{i-1}'^+$ is $\widehat{e}_{12}, \widehat{e}_{34}$-amenable if it is split, then $\widehat{\D}_{i}'^+$ is $\widehat{e}_{12}, \widehat{e}_{34}$-switchable or $\widehat{e}_{12}, \widehat{e}_{34}$-amenable.\label{nonanona-notequal1}
\end{nonamenoname}

\begin{proof}
We may assume that $\widehat{\D}_{i}'^+ $ is not $\widehat{e}_{12}, \widehat{e}_{34}$ - switchable.  Therefore, our task is to show that $\widehat{\D}_i'^+$ is $\widehat{e}_{12}, \widehat{e}_{34}$-amenable.  
If $\widehat{g}_{i-1}' \in\{  \widehat{e}_{23}, \widehat{e}_{24} \}$, then $\widehat{\D}_{i-1}'^+$ is split and hence also $\widehat{e}_{12}, \widehat{e}_{34}$-amenable.  However, we also have that $\widehat{F}_I(\widehat{g}_{i-1}') = E_{\widehat{H}}(\widehat{u}).$  If say, we assume $\widehat{D}_{(i-1)1}' \cap \{ \widehat{e}_{23}, \widehat{e}_{24} \} \ne \emptyset,$ then $\widehat{F}_I(\widehat{D}_{(i-1)1}') \ne \widehat{F}_I(\widehat{g}_{i-1}')$, since $\widehat{\D}_{i-1}'^+$ is $\widehat{e}_{12}, \widehat{e}_{34}$-amenable.  From this, it is seen that $\widehat{\D}_{i-1}'^+$ is $\widehat{e}_{12}, \widehat{e}_{34}$-switchable, contradicting our assumptions.  Thus  $\widehat{g}_{i-1}' \not\in \{ \widehat{e}_{23}, \widehat{e}_{24} \}.$  Since $\widehat{\D}_i'^+$ is also not $\widehat{e}_{12}, \widehat{e}_{34}$-switchable,  we also have $\widehat{g}_i' \not\in \{ \widehat{e}_{23}, \widehat{e}_{24} \}.$  

If $\widehat{g}_{i-1}' = \widehat{g}_i,$ then since $\widehat{\D}_{i-1}'^+ \sim_1 \widehat{\D}_i'^+,$ we have $\widehat{\D}_{i-1}' \sim_1 \widehat{\D}_i'.$  It then follows by Proposition \ref{pro-extend2} that $\widehat{\D}_i'$ is $\widehat{e}_{12}, \widehat{e}_{34}$-amenable.  Thus we
 we may assume that $\widehat{g}_{i-1}' \ne \widehat{g}_i.$   Thus $\widehat{\D}_i'^+$ is obtained from $\widehat{\D}_{i-1}'^+$ by an ({\bf EB}) exchange where $\widehat{e} = \widehat{g}_{i-1}'$ is exchanged with $\widehat{f} = \widehat{g}_i'.$  We may assume that $\widehat{f} \in \widehat{D}_{(i-1)1}'.$  Let $j_1,j_2 \in \{ 1, \dots ,\kappa \},$ where $\widehat{e}_{23} \in \widehat{D}_{(i-1)j_1}'$ and $\widehat{e}_{24} \in \widehat{D}_{(i-1)j_2}'.$  If $1 \not\in \{ j_1, j_2 \}$ then $\widehat{D}_{ij_1}' = \widehat{D}_{(i-1)j_1}'$ and $\widehat{D}_{ij_2}' = \widehat{D}_{(i-1)j_2}'$ which would mean that $\widehat{\D}_i'^+$ is $\widehat{e}_{12}, \widehat{e}_{34}$-amenable since $\widehat{\D}_{i-1}'^+$ is.  Thus we may assume that $j_1 = 1;$. that is, $\widehat{e}_{23} \in \widehat{D}_{(i-1)1}'.$
 

We observe that $\widehat{\D}_{i-1}'^+$ cannot be fused since $\widehat{\D}_i'^+$ is split and $\widehat{e} = \widehat{g}_{i}' \not\in \{ \widehat{e}_{23}, \widehat{e}_{24} \}.$ Thus $\widehat{\D}_{i-1}'^+$ is split
and we have $j_2 \ne 1.$  For convenience, we may assume $j_2 =2;$ that is, $\widehat{e}_{24} \in \widehat{D}_{(i-1)2}'.$   Given that $\D_{i-1}'^+$ is not $\widehat{e}_{12}, \widehat{e}_{34}$-switchable, it follows that 
$\widehat{F}_I(\widehat{D}_{(i-1)1}') = \{ \widehat{e}_{12}, \widehat{e}_{34}, \widehat{e}_{23}, \widehat{e}_{14} \}$ and  $\widehat{F}_I(\widehat{D}_{(i-1)2}') = \{ \widehat{e}_{12}, \widehat{e}_{34}, \widehat{e}_{13}, \widehat{e}_{24} \}$.  Thus
$\widehat{F}_I(\widehat{D}_{i2}') = \widehat{F}_I(\widehat{D}_{(i-1)2}') = \{ \widehat{e}_{12}, \widehat{e}_{34}, \widehat{e}_{13}, \widehat{e}_{24} \}$.  If $\widehat{F}_I(\widehat{D}_{i1}')
\cap \{ \widehat{e}_{12}, \widehat{e}_{34} \} \ne \emptyset,$ then clearly $\widehat{\D}_i'^+$ is $\widehat{e}_{12}, \widehat{e}_{34}$-amenable.  Thus we may assume that  $\widehat{F}_I(\widehat{D}_{i1}')
\cap \{ \widehat{e}_{12}, \widehat{e}_{34} \} = \emptyset.$   Since $\widehat{e} = \widehat{g}_{i}' \ne \widehat{e}_{23},$ it follows that
$\widehat{e}_{23} \in \widehat{D}_{i1}',$ implying that $\widehat{e}_{23} \in \widehat{F}_I(\widehat{D}_{i1}').$  Since $\widehat{F}_I(\widehat{D}_{i1}')
\cap \{ \widehat{e}_{12}, \widehat{e}_{34} \} = \emptyset,$ it follows that $\widehat{F}_I(\widehat{D}_{i1}') = \{ \widehat{e}_{13}, \widehat{e}_{14}, \widehat{e}_{23}, \widehat{e}_{24} \}.$
However, this implies that $\widehat{\D}_i'^+$ is $\widehat{e}_{23}, \widehat{e}_{14}$-switchable, a contradiction. 
\end{proof}

\subsubsection{The extended sequences $\widehat{\D}_i''^+$}

Suppose that $\widehat{\D}_i'^+$ is not $\widehat{e}_{12}, \widehat{e}_{34}$-switchable for $i= 1,\dots ,p',$ where $p' <p.$    Then ({\bf \ref{nonanona-notequal1}}) implies that for $i = 1,\dots, p',$ if $\widehat{\D}_i'$ is split then it is $\widehat{e}_{12}, \widehat{e}_{34}$ amenable.  Let $\widehat{\D}_i''^+ = (\widehat{\D}_i'', \widehat{g_i}''),\ i = 1, \dots ,p'$ be such that $\widehat{\D}_i''^+ = \widehat{\D}_i'^+$ if $\widehat{\D}_i'^+$ is fused.  Otherwise, if $\widehat{\D}_i'^+$ is split, then we let $\widehat{\D}_i''^+$ be a $\widehat{e}_{12}, \widehat{e}_{34}$-perturbation of $\widehat{\D}_i'^+.$  We define $\D_i''^+ = (\D_i'', g_i'') \in \fB_M(\widehat{\D}_i''^+),\ i = 1, \dots ,p'$ such that  $\D_1''^+ = \B_1^+,$ and, iteratively, $\D_i''^+ \vdash \D_{i+1}''^+,\ i = 1, \dots ,p'-1.$

\begin{nonamenoname}
$\D_1''^+ \sim \D_2''^+ \sim \cdots \sim \D_{p'}''^+.$\label{nonanona-notequal2}
\end{nonamenoname}

\begin{proof}
 We shall show that $\D_i''^+ \sim \D_{i+1}''^+,\ i = 1, \dots ,p'-1.$  We first observe that 
if $\widehat{\D}_i'^+$ is split, then, given that $\widehat{\D}_i'^+$ is not $\widehat{e}_{12}, \widehat{e}_{34}$-switchable, it follows that $\widehat{g}_i' \not\in \{ \widehat{e}_{23}, \widehat{e}_{24} \}.$  On the other hand, if $\widehat{\D}_i'^+$ is fused, then $\widehat{g}_i' \not\in \{ \widehat{e}_{23}, \widehat{e}_{24} \}.$ Thus $\widehat{g}_i' \not\in \{ \widehat{e}_{23}, \widehat{e}_{24} \},\ i = 1, \dots ,p'.$  In particular, $\widehat{g}_i'' = \widehat{g}_i',\ i =1, \dots ,p'.$ 

  Fix $i \in \{ 1, \dots ,p' \}.$  If $\widehat{g}_i'' = \widehat{g}_{i+1}''$, then $g_i'' = g_{i+1}'',$ since $\D_i''^+ \vdash \D_{i+1}''^+.$  In this case, $\D_i''^+ \bumpeq \D_{i+1}''^+$ and hence $\D_i''^+ \sim \D_{i+1}''^+.$  Thus we may assume that $\widehat{g}_i'' \ne \widehat{g}_{i+1}''$ (and thus $\widehat{g}_i' \ne \widehat{g}_{i+1}'$) and $\widehat{D}_{(i+1)1}' = \widehat{D}_{i1}'-\widehat{g}_{i+1}' + \widehat{g}_i'.$  Then $\widehat{D}_{ij}' = \widehat{D}_{(i+1)j}',\ j = 2, \dots ,\kappa.$  Thus we may assume that $\widehat{D}_{ij}'' = \widehat{D}_{(i+1)j}'',\ j = 2, \dots ,\kappa.$ If $\widehat{D}_{ij}'' \cap \widehat{E} \ne \emptyset,$ for some $j\ge 2,$ then it would follow that $D_{ij}'' = D_{(i+1)j}''$, since $\D_i''^+  \vdash \D_{i+1}''^+$, and thus $\D_i''^+ \bumpeq \D_{i+1}''^+.$  It then follows that $\D_i''^+ \sim \D_{i+1}''^+.$  Thus we may assume that $\widehat{D}_{ij}'' \cap \widehat{E} = \emptyset,$ for $j = 2, \dots ,\kappa.$  In particular, this implies that $\widehat{\D}_i'^+$ must be fused.  If $\widehat{\D}_{i+1}'^+$ is split, then $\widehat{g}_{i+1}' \in \{ \widehat{e}_{23}, \widehat{e}_{24} \},$ contradicting the assumption that $\widehat{\D}_{i+1}'^+$ is not $\widehat{e}_{12}, \widehat{e}_{24}$-switchable.  Thus $\widehat{\D}_{i+1}'^+$ is also fused and $\widehat{\D}_i'' = \widehat{\D}_i'$ and $\widehat{\D}_{i+1}'' = \widehat{\D}_{i+1}'.$  We see that $D_{i2}''$ and $D_{(i+1)2}''$ are $v$-anchored and we observe that $e_1$ belongs to $v$-anchored bases in both $\D_i''^+$ and $\D_{i+1}''^+.$   Thus we may assume that $e_1 \in \D_{i2}'' \cap D_{(i+1)2}'',$  implying that $D_{i2}'' = D_{(i+1)2}''$.  Consequently $\D_i''^+ \bumpeq \D_{i+1}''^+$ and hence $\D_i''^+ \sim \D_{i+1}''^+$. 
\end{proof}

Let $$p' = \max \left\{ i\in \{ 1, \dots ,p-1\} \ \big| \ \widehat{\D}_j'^+ \ \mathrm{not}\ \widehat{e}_{12}, \widehat{e}_{34}-\mathrm{switchable}, \ \forall j\in \{ 1, \dots ,i\} \right\}.$$
We shall choose $\widehat{\D}_i''^+$ and $\D_i''^+,\ i = 1, \dots ,p'$ as before.   By definition,  $\widehat{\D}_{p'+1}'^+$ is $\widehat{e}_{12}, \widehat{e}_{34}$-switchable.  Let $\widehat{\D}_{p'+1}''^+ = (\widehat{\D}_{p'+1}'', \widehat{g}_{p'+1}'')$ be an $\widehat{e}_{12}, \widehat{e}_{34}$-switch of $\widehat{\D}_{p'+1}'^+.$   Let $\D_{p'+1}''^+ \in \fB_M(\widehat{\D}_{p'+1}''^+).$  As before, may assume that $\D_i''^+ \vdash \D_{i+1}'',\ i =1, \dots ,p'.$    

\begin{nonamenoname}
$\D_1''^+ \sim \D_2''^+ \sim \cdots \sim \D_{p'+1}''^+.$\label{nonanona-notequal2}
\end{nonamenoname}

\begin{proof}   By {\bf (\ref{nonanona-notequal1})}, we have $\D_1''^+ \sim \D_2''^+ \sim \cdots \sim \D_{p'}''^+.$  Thus we need only show that $\D_{p'}''^+ \sim \D_{p'+1}''.$  Let $I = \{ 1,2,3,4 \}$ and $\widehat{H} = \widehat{H}_I.$  As before, we have that $\widehat{g}_{p'}' \not\in \{ \widehat{e}_{23}, \widehat{e}_{24} \}$, since $\widehat{\D}_{p'}'$ is not $\widehat{e}_{12}, \widehat{e}_{34}$-switchable.  This means that $\widehat{g}_{p'}'' = \widehat{g}_{p'}'.$  If $\widehat{g}_{p'}'' = \widehat{g}_{p'+1}''$, then $g_{p'}'' = g_{p'+1}''$ (since $\D_{p'}'' \vdash \D_{p'+1}''$) and hence $\D_{p'}'' \bumpeq \D_{p'+1}''.$  Thus in this case, $\D_{p'}'' \sim \D_{p'+1}''.$  Because of this, we may assume that $\widehat{g}_{p'}'' \ne \widehat{g}_{p'+1}''$.  If $\widehat{g}_{p'}' = \widehat{g}_{p'+1}',$ then $\widehat{g}_{p'+1}'' \ne \widehat{g}_{p'+1}'.$  By the way we choose a $\widehat{e}_{12}, \widehat{e}_{34}$-switch of $\widehat{\D}_{p'+1}',$ this can only happen if $\widehat{g}_{p'+1}' \in \{ \widehat{e}_{23}, \widehat{e}_{24} \}$, which in turn implies $\widehat{g}_{p'}' \in \{ \widehat{e}_{23}, \widehat{e}_{24} \},$ a contradiction.  Thus $\widehat{g}_{p'}' \ne \widehat{g}_{p'+1}'.$  We may assume that $\widehat{g}_{p'+1}' \in \widehat{D}_{p'1}'$ and $\widehat{D}_{(p'+1)1}' = \widehat{D}_{p'1} - \widehat{g}_{p'+1}' + \widehat{g}_{p'}'.$  We see that $\widehat{D}_{p'j}' = \widehat{D}_{(p'+1)j}' = \widehat{D}_{(p'+1)j}'',\ j = 2, \dots ,\kappa.$

\begin{nonamenonamenoname}
If $\widehat{g}_{p'+1}' \in \{ \widehat{e}_{23}, \widehat{e}_{24} \},$ then $\D_{p'}''^+ \sim \D_{p'+1}''^+.$
\end{nonamenonamenoname}

\begin{proof}
Suppose $\widehat{g}_{p'+1}' \in \{ \widehat{e}_{23}, \widehat{e}_{24} \}.$  Assume that $\widehat{e}_{24} \in \widehat{D}_{(p'+1)j}',$ for some $j \ge 2.$  Then $\widehat{g}_{p'+1}' = \widehat{e}_{23}$ and $\widehat{e}_{23} \in \widehat{D}_{p'1}'.$  Furthermore, $\widehat{\D}_{p'}'^+$ is split and hence also $\widehat{e}_{12}, \widehat{e}_{34}$-amenable.   Since $\widehat{\D}_{p'}'^+$ is not $\widehat{e}_{12}, \widehat{e}_{34}$-switchable, it follows that $\widehat{F}_I(\widehat{D}_{p'1}') = \{ \widehat{e}_{12}, \widehat{e}_{34}, \widehat{e}_{23}, \widehat{e}_{14} \}$ and  $\widehat{F}_I(\widehat{D}_{p'j}') = \{ \widehat{e}_{12}, \widehat{e}_{34}, \widehat{e}_{13}, \widehat{e}_{24} \}.$
Since $\widehat{D}_{(p'+1)j}' = \widehat{D}_{p'j}',$ it follows that  $\widehat{F}_I(\widehat{D}_{(p'+1)j}') = \{ \widehat{e}_{12}, \widehat{e}_{34}, \widehat{e}_{13}, \widehat{e}_{24} \}.$  We now see that $\widehat{\D}_{p'+1}'^+$ is $\widehat{e}_{12}, \widehat{e}_{34}$-amenable.  Let $\widehat{\D}_{p'+1}^{(3)+} = (\widehat{\D}_{p'+1}^{(3)}, \widehat{g}_{p'+1}^{(3)})$ be a $\widehat{e}_{12}, \widehat{e}_{34}$-perturbation of $\widehat{\D}_{p'+1}'^+$ and 
let $\D_{p'+1}^{(3)+} = (\D_{p'+1}^{(3)}, g_{p'+1}^{(3)}) \in \fB_M(\widehat{\D}_{p'+1}^{(3)+})$ where $\D_{p'+1}^{(3)} = (D_{(p'+1)1}^{(3)}, \dots , D_{(p'+1)\kappa}^{(3)})$.  We may assume that 
$\widehat{D}_{p'j}'' = \widehat{D}_{(p'+1)j}^{(3)}$ (given that $\{ \widehat{e}_{12}, \widehat{e}_{34} \} \subset \widehat{F}_I(\widehat{D}_{p'1}')\cap \widehat{F}_I(\widehat{D}_{p'j}').$  Furthermore, we may choose $\D_{p'+1}^{(3)+}$ so that 
$\D_{p'}''^+ \vdash \D_{p'+1}^{(3)+}$.  Then $D_{(p'+1)j}^{(3)} = D_{p'j}'',$ since $\widehat{D}_{(p'+1)j}^{(3)} = \widehat{D}_{p'j}''.$  Thus $\D_{p'}''^+ \bumpeq \D_{p'+1}^{(3)+},$ and $\D_{p'}''^+ \sim \D_{p'+1}^{(3)+}.$  However, it is also seen that $\widehat{D}_{(p'+1)1}'' =  \widehat{D}_{(p'+1)1}^{(3)},$ and as such we may choose $D_{(p'+1)}^{(3)} = D_{(p'+1)1}''.$  Thus $\D_{p'+1}''^+ \bumpeq \D_{p'+1}^{(3)+}$.  Thus $\D_{p'}''^+ \sim \D_{p'+1}^{(3)+} \sim \D_{p'+1}''^+ .$  If $\widehat{e}_{23} \in \widehat{D}_{(p'+1)j}',$ for some $j \ge 2,$ then we can use similar arguments.  By the above, we may assume
 that $|\{ \widehat{e}_{23}, \widehat{e}_{24} \} \cap \widehat{D}_{(p'+1)1}' | =1.$  
 
 Without loss of generality, we may assume that $\widehat{e}_{23} \in \widehat{D}_{(p'+1)1}'$ and $\widehat{g}_{p'+1}'=\widehat{e}_{24}.$  Note that this means that $\{ \widehat{e}_{23}, \widehat{e}_{24} \} \subset \widehat{D}_{p'1}'$.  Thus $\widehat{\D}_{p'}'$ is fused and hence $\widehat{\D}_{p'}'' = \widehat{\D}_{p'}'.$  Since $\widehat{g}_{p'+1}'=\widehat{e}_{24},$ either $e_2 \in E_G(u)$ or $e_4 \in E_G(u).$  Suppose $e_4 \in E_G(u)$.  Then $\widehat{F}_I(\widehat{g}_{p'+1}') = \{ \widehat{e}_{14}, \widehat{e}_{24}, \widehat{e}_{34} \}.$  Given that $\widehat{\D}_{p'+1}''^+$ could be any $\widehat{e}_{12}, \widehat{e}_{34}$-switch of $\widehat{\D}_{p'+1}'^+,$ we may assume that $\widehat{\D}_{p'+1}'' = \widehat{\D}_{p'+1}'$ and $\widehat{g}_{p'+1}'' = \widehat{e}_{14}.$  Then $g_{p'+1}'' = e_4$ and $e_1$ is seen to belong to an $v$-anchored base in both $\D_{p'}''$ and $\D_{p'+1}''.$  Thus we may assume that $e_1 \in D_{p'2}'' \cap D_{(p'+1)2}''$ and consequently, $D_{p'2}'' = D_{(p'+1)2}''.$  It follows that $\D_{p'}'' \bumpeq \D_{p'+1}''$ and hence $\D_{p'}'' \sim \D_{p'+1}''.$  Suppose $e_2 \in E_G(u).$  Then $\widehat{F}_I(\widehat{g}_{p'+1}') = \{ \widehat{e}_{12}, \widehat{e}_{23}, \widehat{e}_{24} \}$ and consequently, $\widehat{g}_{p'+1}'' \in \{ 
\widehat{e}_{23}, \widehat{e}_{24} \}.$  Suppose $\widehat{g}_{p'+1}'' = \widehat{e}_{23}.$  Then $g_{p'+1}'' = e_2$ and $e_3$ belongs to a $v$-anchored base in $\D_{p'+1}''.$  On the other hand,
since $\widehat{e}_{23} \in \widehat{D}_{p'1}'',$ it follows by Lemma \ref{lem-Mhat1} that either $\widehat{D}_{p'1}'' - \widehat{e}_{23} + \widehat{e}_{12} \in \B(\widehat{M}_\kappa)$ or 
$\widehat{D}_{p'1}'' - \widehat{e}_{23} + \widehat{e}_{13} \in \B(\widehat{M}_\kappa).$  The latter cannot occur since $\widehat{\D}_{p'}' = \widehat{\D}_{p'}''$ is not $\widehat{e}_{12}, \widehat{e}_{34}$-switchable.  Thus $\widehat{D}_{p'1}'' - \widehat{e}_{23} + \widehat{e}_{12} \in \B(\widehat{M}_\kappa)$.  Given this, we may choose $\D_{p'}''$ such that $e_3$ belongs to an anchored base in $\D_{p'}''.$  Because of this, we may assume that $e_3 \in D_{p'2}'' \cap D_{(p'+1)2}''$, from which it follows that $D_{p'2}'' = D_{(p'+1)2}''.$  Again, we have that $\D_{p'}'' \bumpeq \D_{p'+1}''$ and hence $\D_{p'}'' \sim \D_{p'+1}''.$  One can deal with the case where $\widehat{g}_{p'+1}'' = \widehat{e}_{24}$ using similar arguments.
\end{proof}

From the above, we may assume that $\widehat{g}_{p'+1}' \not\in \{ \widehat{e}_{23}, \widehat{e}_{24} \}$ (and  $\widehat{g}_{p'}' \not\in \{ \widehat{e}_{23}, \widehat{e}_{24} \}$) and consequently, $\widehat{g}_{p'+1}'' = \widehat{g}_{p'+1}'.$  Thus $\widehat{\D}_{p'}'^+$  is split if and only if $\widehat{\D}_{p'+1}'^+$ is split.

\begin{nonamenonamenoname}
If $\widehat{\D}_{p'}'^+$ and $\widehat{\D}_{p'+1}'^+$ are fused, then $\D_{p'}''^+ \sim \D_{p'+1}''^+.$
\end{nonamenonamenoname}

\begin{proof}
Suppose $\widehat{\D}_{p'}'^+$ and $\widehat{\D}_{p'+1}'^+$ are fused.  Since $\widehat{\D}_{p'}'^+$ is not $\widehat{e}_{12}, \widehat{e}_{34}$-switchable whereas $\widehat{\D}_{p'+1}'^+$ is, it follows that 
$\{ \widehat{e}_{23}, \widehat{e}_{24} \} \subseteq \widehat{D}_{p'1}'$ and $\{ \widehat{e}_{23}, \widehat{e}_{24} \} \subseteq \widehat{D}_{(p'+1)1}'.$  By definition, $\D_{p'}''^+ \in \fB_{M_\kappa}(\widehat{\D}_{p'}'^+),$ since $\widehat{\D}_{p'}'^+$ is fused.  At the same time, since $\widehat{\D}_{p'+1}'^+$ is fused, there exists $\D_{p'+1}'^+ = (\D_{p'+1}', g_{p'+1}')\in \fB_{M_\kappa}(\widehat{\D}_{p'+1}'^+)$ where $\D_{p'+1}' = (D_{(p'+1)1}', \dots ,D_{(p'+1)\kappa}') \in \fB_{M_\kappa}(\widehat{\D}_{p'+1}').$  We observe that $e_1$ belongs to $v$-anchored bases in both $\D_{p'}''$ and $\D_{p'+1}'$ and as such, we may assume that $e_1 \in D_{p'2}'' \cap D_{(p'+1)2}'.$  Then $D_{p'2}'' = D_{(p'+1)2}'$ and hence $\D_{p'}''^+ \bumpeq \D_{p'+1}'^+.$  On the other hand, since $g_{p'+1}'' = g_{p'+1}',$ it follows that $\D_{p'+1}'^+ \bumpeq \D_{p'+1}''^+.$
Thus $\D_{p'}''^+ \bumpeq \D_{p'+1}'^+ \bumpeq D_{p'+1}''^+$ and hence $\D_{p'}''^+ \sim \D_{p'+1}'^+ \sim D_{p'+1}''^+.$
\end{proof}

Suppose that $\widehat{\D}_{p'}'^+$ and $\widehat{\D}_{p'+1}'^+$ are both split.  Given that $\widehat{\D}_{p'}'^+$ is not
$\widehat{e}_{12}, \widehat{e}_{34}$-switchable, whereas $\widehat{\D}_{p'+1}'^+$ is, we must have that $\{ \widehat{e}_{23}, \widehat{e}_{24} \} \cap \widehat{D}_{p'1}' \ne \emptyset.$
We may assume that $\widehat{e}_{23} \in \widehat{D}_{p'1}'$ and $\widehat{e}_{24} \in \widehat{D}_{p'2}'.$   
Since $\widehat{\D}_{p'}'^+$ is $\widehat{e}_{12},\widehat{e}_{34}$-amenable but not $\widehat{e}_{12},\widehat{e}_{34}$-switchable, it follows that $\widehat{F}_I(\widehat{D}_{p'1}') = \{ \widehat{e}_{12}, \widehat{e}_{34}, \widehat{e}_{23}, \widehat{e}_{14} \}$ and $\widehat{F}_I(\widehat{D}_{p'2}') = \{ \widehat{e}_{12}, \widehat{e}_{34}, \widehat{e}_{13}, \widehat{e}_{24} \}.$   It should be noted that since $\widehat{D}_{(p'+1)2}' = \widehat{D}_{p'2}',$ we have $\widehat{F}_I(\widehat{D}_{(p'+1)2}') = \widehat{F}_I(\widehat{D}_{p'2}').$  We may assume that
$\widehat{D}_{p'1}''= \widehat{D}_{p'1}' - \widehat{e}_{23} + \widehat{e}_{12},$ and $\widehat{D}_{p'2}'' = \widehat{D}_{p'2}' - \widehat{e}_{24} + \widehat{e}_{34}.$  

\begin{nonamenonamenoname}
If $\{ \widehat{e}_{12}, \widehat{e}_{34} \} \cap \widehat{F}_I(\widehat{D}_{(p'+1)1}') \ne \emptyset,$ then $\D_{p'}''^+ \sim \D_{p'+1}''^+.$
\end{nonamenonamenoname}

\begin{proof}
Suppose 
$\widehat{e}_{12} \in \widehat{F}_I(\widehat{D}_{(p'+1)1}').$  
Let $\widehat{\D}_{p'+1}^{(3)+} = (\widehat{\D}_{p'+1}^{(3)}, \widehat{g}_{p'+1}^{(3)})$ where $\widehat{g}_{p'+1}^{(3)} = \widehat{g}_{p'+1}''$, $\widehat{\D}_{p'+1}^{(3)} = (\widehat{D}_{(p'+1)1}^{(3)}, \dots ,\widehat{D}_{(p'+1)\kappa}^{(3)}),$ and $\widehat{D}_{(p'+1)1}^{(3)} = \widehat{D}_{(p'+1)1}' - \widehat{e}_{23} + \widehat{e}_{12}$, $\widehat{D}_{(p'+1)2}^{(3)} = \widehat{D}_{(p'+1)2}' - \widehat{e}_{24} + \widehat{e}_{34}$
 and $\widehat{D}_{(p'+1)j}^{(3)}= \widehat{D}_{p'j}',\ j = 3, \dots ,\kappa.$  Let $\D_{p'+1}^{(3)+} = (\D_{p'+1}^{(3)}, g_{p'+1}^{(3)})$ where
 \begin{itemize}
\item $\D_{p'+1}^{(3)} = (D_{(p'+1)1}^{(3)}, \dots ,D_{(p'+1)\kappa}^{(3)}) \in \fB_{M_\kappa}(\widehat{\D}_{p'+1}^{(3)})$
\item $\D_{p'}''^+ \vdash \D_{p'+1}^{(3)+}.$ 
\end{itemize}
We note that $g_{p'+1}^{(3)} = g_{p'+1}''$ since $\widehat{g}_{p'+1}^{(3)} = \widehat{g}_{p'+1}''$ and $\D_{p'}''^+ \vdash \D_{p'+1}^{(3)+}.$   Thus $\D_{p'+1}^{(3)} \bumpeq \D_{p'+1}''$ and hence $\D_{p'+1}^{(3)} \sim \D_{p'+1}''$.  We also see that $\widehat{D}_{(p'+1)2}^{(3)} = \widehat{D}_{p'2}'',$ and hence we have that $D_{(p'+1)2}^{(3)} = D_{p'2}'',$ since $\D_{p'}''^+ \vdash \D_{p'+1}^{(3)+}.$  Thus $\D_{p'}''^+ \bumpeq \D_{p'+1}^{(3)+}$ and hence $\D_{p'}''^+ \sim \D_{p'+1}^{(3)+}$.   It now follows that $\D_{p'}''^+ \sim \D_{p'+1}^{(3)+}\sim \D_{p'+1}''^+.$  Since $\widehat{e}_{12}$ and $\widehat{e}_{34}$ can be interchanged in $\widehat{\D}_{p'}'^+$, the same proof works if $\widehat{e}_{34} \in  \widehat{F}_I(\widehat{D}_{(p'+1)1}').$
\end{proof}

By the above, we may assume that $\{ \widehat{e}_{12}, \widehat{e}_{34} \} \cap \widehat{F}_I(\widehat{D}_{(p'+1)1}') = \emptyset.$    Thus it follows that $\widehat{F}_I(\widehat{D}_{(p'+1)1}') = \{ \widehat{e}_{13}, \widehat{e}_{14}, \widehat{e}_{23}, \widehat{e}_{24} \}.$  Because of this, we may assume that $\widehat{D}_{(p'+1)1}'' = \widehat{D}_{(p'+1)1}' - \widehat{e}_{23} + \widehat{e}_{13}$ and $\widehat{D}_{(p'+1)2}'' = \widehat{D}_{(p'+1)2}'$.

\begin{nonamenonamenoname}
If $\widehat{D}_{p'1}'' \cap (\widehat{E} \backslash \{ \widehat{e}_{12} \}) \ne \emptyset$ or $\widehat{D}_{p'2}'' \cap (\widehat{E} \backslash \{ \widehat{e}_{34} \}) \ne \emptyset,$ then $\D_{p'}''^+ \sim \D_{p'+1}''^+.$
\end{nonamenonamenoname}

\begin{proof}
Suppose $\widehat{D}_{p'1}'' \cap (\widehat{E}\backslash \{ \widehat{e}_{12} \} ) \ne \emptyset.$  Without loss of generality, we may assume that $\widehat{e}_{56} \in \widehat{D}_{p'1}''.$  By Lemma \ref{lem-Mhat1}, we have that either $\widehat{D}_{p'1}'' - \widehat{e}_{12}  + \widehat{e}_{15} \in \B(\widehat{M}_\kappa)$ or $\widehat{D}_{p'1}'' - \widehat{e}_{12} + \widehat{e}_{25} \in \B(\widehat{M}_\kappa).$  Assume that the former holds. 
Let $\widehat{\D}_{p'}^{(3)} = (\widehat{D}_{p'1}^{(3)}, \dots ,\widehat{D}_{p'\kappa}^{(3)})$ where $\widehat{D}_{p'1}^{(3)} = \widehat{D}_{p'1}'' - \widehat{e}_{12} + \widehat{e}_{15},$ $\widehat{D}_{p'2}^{(3)} = \widehat{D}_{p'2}'' - \widehat{e}_{34} + \widehat{e}_{24},$ and $\widehat{D}_{p'j}^{(3)} = \widehat{D}_{p'j}'',$ for all $j\in \{3, \dots ,\kappa \}.$  Let $\widehat{\D}_{p'}^{(3)+} = (\widehat{\D}_{p'}^{(3)}, \widehat{g}_{p'}^{(3)})$, where 
$\widehat{g}_{p'}^{(3)} = \widehat{g}_{p'}''.$   Note that $\widehat{D}_{p'2}^{(3)} = \widehat{D}_{(p'+1)2}''.$  Let $\D_{p'}^{(3)+} = (\D_{p'}^{(3)}, g_{p'}^{(3)}),$ and $\D_{p'}^{(3)} = (D_{p'1}^{(3)}, \dots ,D_{p'\kappa}^{(3)}),$ where we may assume that $g_{p'}^{(3)} = g_{p'}''$ and $D_{p'2}^{(3)}  = D_{(p'+1)2}''.$   Then we see that $\D_{p'}''^+ \bumpeq \D_{p'}^{(3)+} \bumpeq \D_{p'+1}''^+$, and hence $\D_{p'}''^+ \sim \D_{p'}^{(3)+} \sim \D_{p'+1}''^+.$

On the other hand, suppose $\widehat{D}_{p'1}' - \widehat{e}_{12} + \widehat{e}_{25} \in \B(\widehat{M}).$   Let $\widehat{\D}_{p'}^{(3)} = (\widehat{D}_{p'1}^{(3)}, \dots ,\widehat{D}_{p'\kappa}^{(3)})$ where $\widehat{D}_{p'1}^{(3)} = \widehat{D}_{p'1}'' - \widehat{e}_{12} + \widehat{e}_{25},$ $\widehat{D}_{p'2}^{(3)} = \widehat{D}_{p'2}'' - \widehat{e}_{34} + \widehat{e}_{13},$ and $\widehat{D}_{p'j}^{(3)} = \widehat{D}_{p'j}'',$ for all $j\in \{3, \dots ,\kappa \}.$  Let $\widehat{\D}_{p'}^{(3)+} = (\widehat{\D}_{p'}^{(3)}, \widehat{g}_{p'}^{(3)})$, where $\widehat{g}_{p'}^{(3)} = \widehat{g}_{p'}''.$  Let $\D_{p'}^{(3)+} = (\D_{p'}^{(3)}, g_{p'}^{(3)}),$ and $\D_{p'}^{(3)} = (D_{p'1}^{(3)}, \dots ,D_{p'\kappa}^{(3)}),$ where we may assume that $g_{p'}^{(3)} = g_{p'}''$.  Then $\D_{p'}''^+ \sim \D_{p'}^{(3)+}.$

Let $\widehat{\D}_{p'+1}^{(3)+} = (\widehat{\D}_{p'+1}^{(3)}, \widehat{g}_{p'+1}^{(3)}),$ where $\widehat{g}_{p'+1}^{(3)} = \widehat{g}_{p'+1}'',$ $\widehat{\D}_{p'+1}^{(3)} = (\widehat{D}_{(p'+1)1}^{(3)}, \dots ,\widehat{D}_{(p'+1)\kappa}^{(3)}),$ $\widehat{D}_{(p'+1)1}^{(3)} = \widehat{D}_{(p'+1)1}'' - \widehat{e}_{13} + \widehat{e}_{24},$ 
$\widehat{D}_{(p'+1)2}^{(3)} = \widehat{D}_{(p'+1)2}'' - \widehat{e}_{24} + \widehat{e}_{13},$ and $\widehat{D}_{(p'+1)j}^{(3)} = \widehat{D}_{(p'+1)j}'',\ j = 3, \dots ,\kappa.$  Note that $\widehat{D}_{p'2}^{(3)} = \widehat{D}_{(p'+1)2}^{(3)}.$  Thus choosing $\D_{p'+1}^{(3)+} = (\D_{p'+1}^{(3)}, g_{p'+1}^{(3)})$ so that $g_{p'+1}^{(3)} = g_{p'+1}''$ and choosing $\D_{p'+1}^{(3)} = (D_{(p'+1)1}^{(3)}, \dots ,D_{(p'+1)\kappa}^{(3)})$ so that $D_{(p'+1)2}^{(3)} = D_{p'2}^{(3)},$ we see that $\D_{p'}''^+ \sim \D_{p'}^{(3)+} \sim  \D_{p'+1}^{(3)+} \sim \D_{p'+1}''^+.$

From the above, we may assume that $\widehat{D}_{p'1}'' \cap \widehat{E} =\{ \widehat{e}_{12} \}.$  Suppose $\widehat{D}_{p'2}''\cap (\widehat{E} \backslash \{ \widehat{e}_{34} \} ) \ne \emptyset.$  We may assume that $\widehat{e}_{56} \in \widehat{D}_{p'2}''.$  By Lemma \ref{lem-Mhat1}, either $\widehat{D}_{p'2}'' - \widehat{e}_{34}  + \widehat{e}_{45} \in \B(\widehat{M}_\kappa)$ or $\widehat{D}_{p'2}'' - \widehat{e}_{34} + \widehat{e}_{35} \in \B(\widehat{M}_\kappa).$  Suppose the former occurs.  Let $\widehat{\D}_{p'}^{(3)+} = (\widehat{\D}_{p'}^{(3)}, \widehat{g}_{p'}^{(3)}),$ where $\widehat{g}_{p'}^{(3)} = \widehat{g}_{p'}'',$ $\widehat{\D}_{p'}^{(3)} = (\widehat{D}_{p'1}^{(3)}, \dots ,\widehat{D}_{p'\kappa}^{(3)}),$ $\widehat{D}_{p'2}^{(3)} = \widehat{D}_{p'2} - \widehat{e}_{34} + \widehat{e}_{45},$ and $\widehat{D}_{p'}^{(3)} = \widehat{D}_{p'j}'',\ \forall j\in \{ 1,3,\dots ,\kappa \}.$
Let $\D_{p'}^{(3)+} = (\D_{p'}^{(3)}, g_{p'}^{(3)}) \in \fB_{M_\kappa}(\widehat{\D}_{p'}^{(3)+}),$ where $g_{p'}^{(3)} = g_{p'}'',$ and $\D_{p'}^{(3)} = (D_{p'1}^{(3)}, \dots ,D_{p'\kappa}^{(3)}).$  We observe that $\fB_{M_\kappa}(\widehat{D}_{p'2}^{(3)}) \subseteq \fB_{M_\kappa}(\widehat{D}_{(p'+1)2}'').$  Thus we may choose $D_{p'2}^{(3)} = D_{(p'+1)2}''.$  We now see that $\D_{p'}''^+ \bumpeq \D_{p'}^{(3)+} \bumpeq \D_{p'+1}''^+$ and hence
$\D_{p'}''^+ \sim \D_{p'}^{(3)+} \sim \D_{p'+1}''^+.$  If instead $\widehat{D}_{p'2}'' - \widehat{e}_{34} + \widehat{e}_{35} \in \B(\widehat{M}_\kappa)$, then one can use similar arguments to the above to show that $\D_{p'}''^+ \sim \D_{p'+1}''^+.$  
\end{proof}
From the above, we may assume that $\widehat{D}_{p'1}''\cap \widehat{E} = \{ \widehat{e}_{12} \}$ and $\widehat{D}_{p'2}''\cap \widehat{E} = \{ \widehat{e}_{34} \}$.  

\begin{nonamenonamenoname}
If $\kappa \ge 3,$ then $\D_{p'}''^+ \sim \D_{p'+1}''^+.$
\end{nonamenonamenoname}

\begin{proof}If $\kappa \ge 3,$ then $\widehat{D}_{p'j}'' \cap \widehat{E} \ne \emptyset,$ for some $j\in \{ 3, \dots ,\kappa \}.$  Since $\D_{p'}''^+ \vdash \D_{p'+1}''^+,$ it follows that $D_{p'j}'' = D_{(p'+1)j}''$ and hence $\D_{p'}''^+ \bumpeq \D_{p'+1}''^+.$  It then follows that $\D_{p'}''^+ \sim \D_{p'+1}''^+.$ \end{proof} 

From the above, we may assume that $\kappa = 2$.  Since $D_{p'1}'' \cap \widehat{E} = \{ \widehat{e}_{12}\}$ and $D_{p'2}'' \cap \widehat{E} = \{ \widehat{e}_{34}\}$, it follows that $\{ e_1, e_2 \} \subseteq D_{p'1}''$ and $\{ e_3, e_4 \} \subseteq D_{p'2}''.$  We also see that $D_{(p'+1)1}'' \cap \widehat{E} = \{ \widehat{e}_{13}\}$ and $D_{(p'+1)2}'' \cap \widehat{E} = \{ \widehat{e}_{24} \}$ and hence 
$\{ e_1, e_3 \} \subseteq D_{(p'+1)1}''$ and $\{ e_2, e_4 \} \subseteq D_{(p'+1)2}''.$
Since
$\widehat{e}_{12} \in \widehat{F}_I(\widehat{D}_{p'1}') \backslash \widehat{F}_I(\widehat{D}_{(p'+1)1}'),$ it follows by Lemma \ref{lem-Mhat3.1} i), that 
$\widehat{D}_{p'1}' - \widehat{e}_{23} + \widehat{g}_{p'} \in \B(\widehat{M}_\kappa).$ Thus\\ $D_{p'1}'' - e_1 -e_2 + g_{p'}'' + e_i \in \B(M_\kappa),$ for all $e_i \in E_G(v).$

As noted in Section \ref{subsec-widehatuextend}, the vertex $\widehat{u}$ is an endvertex of some edge in $\widehat{E}.$  By symmetry, we need only consider the case where $\widehat{u}$ is an endvertex of $\widehat{e}_{12}.$ Then either $e_1 \in E_G(u)$ or $e_2 \in E_G(u).$  Suppose $e_1 \in E_G(u)$.   Let $\D_{p'}^{(3)+} = (\D_{p'}^{(3)}, g_{p'}^{(3)})$ be the extended base sequence obtained from an (EB) exchange where $e_1$ is exchanged with $g_{p'}'';$ that is, $\D_{p'}^{(3)} = (D_{p'1}^{(3)}, \dots ,D_{p'\kappa}^{(3)}),\ D_{p'1}^{(3)} = D_{p'1}'' - e_1 + g_{p'}'',\ D_{p'j}^{(3)} = D_{p'j}'',\ j = 2, \dots ,\kappa$, and $g_{p'}^{(3)} = e_1.$  Then $\D_{p'}''^+ \sim_1 \D_{p'}^{(3)+}.$
Let $\D_{p'}^{(4)+} = (\D_{p'}^{(4)}, g_{p'}^{(4)})$ be the extended bases sequence obtained from a (BB) exchange where $e_2$ is exchanged with $e_3$ in $\D_{p'}^{(3)};$ that is, $\D_{p'}^{(4)} = (D_{p'1}^{(4)}, \dots ,D_{p'\kappa}^{(4)}),\ D_{p'1}^{(4)} = D_{p'1}^{(3)} - e_2 + e_3,\ D_{p'2}^{(4)} = D_{p'2}^{(3)} - e_3 + e_2,\  D_{p'j}^{(4)} = D_{p'j}^{(3)},\ j = 2, \dots ,\kappa$, and $g_{p'}^{(4)} = g_{p'}^{(3)} = e_1.$  Thus $\D_{p'}^{(3)+} \sim_1 \D_{p'}^{(4)+}.$  Finally, we see that $\D_{p'+1}''^+$ is obtained from $\D_{p'}^{(4)+}$ by an (EB) exchange where $g_{p'}^{(4)}=e_1$ is exchanged with $g_{p'+1}''.$  Thus we have $\D_{p'}''^+ \sim_1 \D_{p'}^{(3)+} \sim_1 \D_{p'}^{(4)+} \sim_1 \D_{p'+1}''^+$ and hence $\D_{p'}''^+ \sim \D_{p'+1}''^+.$

Suppose instead that $e_2 \in E_G(u).$   Let $\D_{p'}^{(3)+} = (\D_{p'}^{(3)}, g_{p'}^{(3)})$ be the extended base sequence obtained from $\D_{p'}''^+$ by an (EB) exchange where $e_2$ is exchanged with $g_{p'}'';$ that is, $\D_{p'}^{(3)} = (D_{p'1}^{(3)}, \dots ,D_{p'\kappa}^{(3)}),\ D_{p'1}^{(3)} = D_{p'1}'' - e_2 + g_{p'}'',\ D_{p'j}^{(3)} = D_{p'j}'',\ j = 2, \dots ,\kappa$, and $g_{p'}^{(3)} = e_2.$  Then $\D_{p'}''^+ \sim_1 \D_{p'}^{(3)+}.$
Let $\D_{p'+1}^{(3)+} = (\D_{p'+1}^{(3)}, g_{p'+1}^{(3)})$ be the extended base sequence obtained from $\D_{p'+1}''^+$ by exchanging $e_1,e_3$ with $e_2, e_4$; that is, $g_{p'+1}^{(3)} = g_{p'+1}'',$ and 
given $\D_{p'+1}^{(3)} = (D_{(p'+1)1}^{(3)}, \dots ,D_{(p'+1)\kappa}^{(3)} ),$ we have $D_{(p'+1)1}^{(3)} = D_{(p'+1)1}'' - e_1 - e_3 + e_2 + e_4$, $D_{(p'+1)2}^{(3)} = D_{(p'+1)2}'' - e_2 - e_4 + e_1 + e_3$, and $D_{(p'+1)j}^{(3)} = D_{(p'+1)j}'',\ j = 3, \dots ,\kappa.$  Then it is clear that $\D_{p'+1}''^+ \bumpeq \D_{p'+1}^{(3)+}$ and hence $\D_{p'+1}''^+ \sim \D_{p'+1}^{(3)+}.$  Now one can obtain $\D_{p'+1}^{(3)+}$ from $\D_{p'}^{(3)+}$ by first exchanging $e_1$ and $e_4$ in a (BB) exchange, followed by an (EB) exchange where $e_2$ is exchanged with $g_{p'+1}''.$  Thus $\D_{p'}^{(3)+} \sim \D_{p'+1}^{(3)+}$ and hence we have $\D_{p'}''^+ \sim_1 \D_{p'}^{(3)+} \sim \D_{p'+1}^{(3)+} \sim_1 \D_{p'+1}''^+.$  
\end{proof}

\end{document}